\documentclass[11pt,leqno]{amsart}
\usepackage{latexsym, amssymb,hyperref,amssymb,amsfonts}
\usepackage{geometry,color,graphicx,mathrsfs,empheq,enumitem, cite}
\usepackage[usenames,dvipsnames]{xcolor}

\numberwithin{equation}{section}
\newtheorem{theorem}{Theorem}[section]
\newtheorem{proposition}[theorem]{Proposition}
\newtheorem{lemma}[theorem]{Lemma}
\newtheorem{corollary}[theorem]{Corollary}

\newtheorem{definition}[theorem]{Definition}
\theoremstyle{remark}
\newtheorem{remark}[theorem]{Remark}

\makeatletter
\def\l@section{\@tocline{1}{0pt}{1pc}{}{}}
\def\l@subsection{\@tocline{2}{0pt}{1pc}{4.6em}{}}
\def\l@subsubsection{\@tocline{3}{0pt}{1pc}{7.6em}{}}
\renewcommand{\tocsection}[3]{%
  \indentlabel{\@ifnotempty{#2}{\makebox[2.3em][l]{%
    \ignorespaces#1 #2.\hfill}}}#3}
\renewcommand{\tocsubsection}[3]{%
  \indentlabel{\@ifnotempty{#2}{\hspace*{2.3em}\makebox[2.3em][l]{%
    \ignorespaces#1 #2.\hfill}}}#3}
\makeatother

\def\e{\epsilon}

\def\II{{\rm I\kern-0.5exI}}
\def\III{{\rm I\kern-0.5exI\kern-0.5exI}}
\newcommand{\R}{{\mathord{\mathbb R}}}
\newcommand{\Rd}{{\mathord{\mathbb R}^d}}

\newcommand{\id}{{\mathop{\rm \mathbf{id} }}}
\newcommand{\grad}{\nabla}

\def\P{{\mathcal P}}
\def\epsilon{\varepsilon}
\newcommand{\argmin}{\operatornamewithlimits{argmin}}
\newcommand{\argmax}{\operatornamewithlimits{argmax}}

\newcommand{\bt}{\mathbf{t}}

\newcommand{\supp}{{\rm supp \ }}
\newcommand{\PR}{{\mathcal{P}_2(\mathbb{R}^d)}}

\newcommand{\bN}{\mathbf{N}}
\newcommand{\tomega}{\tilde{\omega}}

\definecolor{purple}{rgb}{0.65, 0, 1}

\begin{document}
\title{Congested aggregation via Newtonian interaction}
\date{February 25, 2016}

\author{ Katy Craig}
\address{Department of Mathematics\\University of California, Santa Barbara}
\thanks{Katy Craig is supported by a UC President's Postdoctoral Fellowship and by the NSF grant DMS-1401867}

 \author{Inwon Kim}
\address{Department of Mathematics\\ University of California, 
Los Angeles}
\thanks{Inwon Kim is supported by the NSF grant DMS-1300445}

\author{ Yao Yao}
\address{School of Mathematics\\Georgia Institute of Technology}
\thanks{Yao Yao is supported by the NSF grant DMS-1565480}
\date{}


\maketitle
\begin{abstract}
We consider a \emph{congested aggregation} model that describes the evolution of a density through the competing effects of nonlocal Newtonian attraction and a hard height constraint. This provides a counterpoint to existing literature on repulsive-attractive nonlocal interaction models, where the repulsive effects instead arise from an interaction kernel or the addition of diffusion. We formulate our model as the \emph{Wasserstein gradient flow} of an interaction energy, with a penalization to enforce the constraint on the height of the density. From this perspective, the problem can be seen as a singular limit of the Keller-Segel equation with degenerate diffusion. Two key properties distinguish our problem from previous work on height constrained equations: nonconvexity of the interaction kernel (which places the model outside the scope of classical gradient flow theory) and nonlocal dependence of the velocity field on the density (which causes the problem to lack a comparison principle). To overcome these obstacles, we combine recent results on gradient flows of nonconvex energies with viscosity solution theory. We characterize the dynamics of patch solutions in terms of a Hele-Shaw type free boundary problem and, using this characterization, show that in two dimensions patch solutions converge to a characteristic function of a disk in the long-time limit, with explicit rate of convergence. We believe that a key contribution of the present work is our blended approach, combining energy methods with viscosity solution theory.
 \end{abstract}

\tableofcontents

\section{Introduction} \label{intro}

In recent years, there has been significant interest in physical and biological models with nonlocal interactions. These models describe the pairwise interactions of a large number of individual agents, for which, in the continuum limit, the nonnegative density $\rho(x,t)$ satisfies the \emph{aggregation equation with degenerate diffusion}
\begin{align}  \label{aggdegdiff}
\rho_t = \nabla \cdot(\rho \nabla \mathcal{N} *\rho) + 
 \Delta \rho^m ,
\end{align}
for an interaction kernel $\mathcal{N}: \Rd \to \R$ and 
$m \geq 1$. 
This equation is mass-preserving and, provided that $\mathcal{N}(x)$ possesses sufficient convexity and regularity, it is a \emph{Wasserstein gradient flow} of the energy
\begin{align*} E_m(\rho) =  W(\rho) + 
  S_m(\rho) , 
\end{align*}
where the \emph{interaction energy} $W(\rho)$ and \emph{R\'enyi entropy} $S_m(\rho)$ are given by
\begin{align*} W(\rho) = \frac{1}{2} \int (\mathcal{N}*\rho)(x) \rho(x) dx  \quad \text{ and } \quad S_m(\rho) :=
\frac{1}{m-1} \int \rho(x)^m dx  .
\end{align*}
 See section \ref{Wasserstein intro section} for further background on this gradient flow structure, including Remark \ref{mass not one} for the case when $\int \rho \neq 1$.

Depending on the choice of interaction kernel and diffusion parameter, equations similar to (\ref{aggdegdiff}) arises in a range applications in physics and biology, including models of granular media  \cite{BenedettoCagliotiCarrilloPulvirenti, CarrilloMcCannVillani}, biological swarming \cite{TopazBertozziLewis, BurgerFetecauHuang}, robotic swarming  \cite{PereaGomezElosegui,ChuangHuangDorsognaBertozzi}, molecular self-assembly \cite{DoyeWalesBerry, Wales, RechtsmanStillingerTorquato}, and the evolution of vortex densities in superconductors  \cite{AmbrosioSerfaty, LinZhang, MasmoudiZhang, Poupaud}. 
Of particular interest are kernels and diffusion parameters for which the model exhibits competing repulsive and attractive effects, causing solutions to blow up in finite time or form rich patterns in the asymptotic limit (c.f. \cite{BertozziCarrilloLaurent, BalagueCarrilloYao, BalagueCarrilloLaurentRaoul, Bertozzietal_RingPatterns,BertozziLaurentLeger,FellnerRaoul2, FetecauHuangKolokolnikov,FetecauHuang,SunUminskyBertozzi}). For example, with 
$m \geq 1$ and the interaction is given by the Newtonian interaction kernel
\begin{align} \label{Newtonian kernel}\mathcal{N}(x) = \begin{cases}\frac{1}{2\pi}\log|x| & \text{ for }d=2,\\ \frac{-1}{d(d-2)\alpha_d} |x|^{2-d} & \text{ for } d \neq 2,
\end{cases}  \quad \text{ with }\alpha_d \text{ the volume of the unit ball in } \Rd  , 
\end{align}
equation (\ref{aggdegdiff}) corresponds to the Keller-Segel model for biological chemotaxis \cite{Blanchet, BlanchetCarlenCarrillo, KellerSegel}
\begin{align}  \label{KellerSegel}
\rho_t = \nabla \cdot(\rho \nabla (\mathcal{N} *\rho)) +  \Delta \rho^m .
\end{align}
 In this case, the interaction kernel is purely attractive and competes with the repulsion induced by the degenerate diffusion. If $m > 2 - 2/d$, diffusion dominates at large density, and bounded solutions exist globally in time \cite{Sugiyama}. Otherwise, depending on the choice of initial data, solutions with bounded initial data may blow up in finite time.

In the present work, we consider a diffusion-aggregation model similar to the Keller-Segel equation, but with the role of diffusion instead played by a hard height constraint on the density.  Heuristically, the evolution of $\rho(x,t)$ is given by  the \emph{congested aggregation equation} 
\begin{align}  \label{congagg} \begin{cases} \rho_t = \nabla \cdot(\rho \nabla \bN \rho) \text{ if } \rho(x,t) \leq 1 , \\
\rho(x,t) \leq 1 \text{ always,} \end{cases} 
\end{align}
where $\bN \rho := \mathcal{N} * \rho$ denotes the Newtonian potential of $\rho$. Informally, solutions of \ref{congagg} seek to evolve according to the ``desired velocity field'' $\nabla \bN \rho$, subject to a hard height constraint.
More precisely, we define $\rho(x,t)$ as the Wasserstein gradient flow of the  \emph{constrained interaction energy}
\begin{align} \label{Edef} E_\infty(\rho) := \begin{cases}
 \frac{1}{2}\int \mathbf{N}\rho(x) \rho(x) dx &\text{if $\| \rho\|_\infty \leq 1$,} \\ +\infty & \text{otherwise} .\end{cases}
\end{align}

Our choice of hard height constraint is inspired by the work of Maury, Roudneff-Chupin, Samtambrogio, and Venel \cite{MRS, MRSV}, who introduced such a constraint in their model of pedestrian crowd motion. They considered a \emph{congested drift equation}
\begin{align}  \label{congtrans} \begin{cases} \rho_t = \nabla \cdot(\rho \nabla V) \text{ if } \rho(x,t) \leq 1 , \\
\rho(x,t) \leq 1 \text{ always,} \end{cases} 
\end{align}
for a local drift $V: \Rd \to \R$, where $\grad V$ is the ``desired velocity field'' of the density. As in the present work, they rigorously defined the evolution of the density as the Wasserstein gradient flow of the \emph{constrained potential energy}
\begin{align*} V_\infty(\rho) := \begin{cases}
 \frac{1}{2}\int V(x) \rho(x) dx &\text{if $\| \rho\|_\infty \leq 1$,} \\ +\infty & \text{otherwise} .\end{cases}
\end{align*}
They then showed that this gradient flow satisfies a formulation of the continuity equation, where the velocity field is given by the $L^2$ projection of $\grad V$ onto the set of \emph{admissible velocities} that do not increase the density in the saturated zone, $\{ \rho = 1 \}$ \cite{MRS}. Furthermore, when $V(x)$ is semiconvex (e.g. when $\nabla^2 V(x)$ is bounded below---see section \ref{Wasserstein intro section}) the energy $V_\infty$ is likewise semiconvex and Wasserstein gradient flow theory ensures that this evolution is unique.

Building upon this work, Alexander, Kim, and Yao \cite{AKY} showed  that solutions of the congested drift equation could be approximated by solutions to a corresponding nonlinear diffusion equation
\begin{align} \label{driftdiffintro}
\rho_t = \nabla \cdot( \rho \nabla V) + \Delta \rho^m
\end{align}
as $m \to +\infty$, which are gradient flows of the energy
\begin{align*}
V_m(\rho):=  \int V(x) \rho(x) dx + \frac{1}{m-1}\int \rho(x)^m dx ,
\end{align*}
(Note that, for a fixed $\rho$, $V_\infty(\rho)$ is the limit of $V_m(\rho)$ as $m\to\infty$.)
 They then applied this result to characterize the dynamics of the congested drift equation: given a velocity field satisfying $\Delta V > 0$ and initial data that is a characteristic function on a patch, $\rho(x, 0) = \chi_{\Omega_0}(x)$ for
 \[ \chi_{\Omega_0}(x) := \begin{cases} 1 &\text{ if } x \in \Omega_0, \\0  &\text{ otherwise,} \end{cases} \] the solution remains a characteristic function, and the evolution of the patch is given by a Hele-Shaw type free boundary problem.

In spite of the similarities between our congested aggregation equation (\ref{congagg}) and the congested drift equation (\ref{congtrans}), two key differences prevent its analysis by the same methods.
First, unlike $V_\infty$, the energy $E_\infty$ does not satisfy the semiconvexity assumptions of classical gradient flow theory that ensure uniqueness. This lack of convexity also makes the equation inaccessible by classical approximation methods---specifically, quantitative approximation by the \emph{discrete gradient flow} or \emph{JKO scheme} for semiconvex energies---which was a key tool in Alexander, Kim, and Yao's result on the convergence of the nonlinear diffusion equation (\ref{driftdiffintro}) as $m \to +\infty$ to the congested drift equation. The second major difference between the congested aggregation and congested drift equations is that the velocity field of the former depends nonlocally on the density. This prevents a direct adaptation of Maury, Roudneff-Chupin, and Santambrogio's characterization of solutions in terms of a continuity equation, since their argument relies upon an Euler-Lagrange equation for the discrete gradient flow sequence, the proof of which strongly leverages the local nature of the drift. Finally, the nonlocal nature of the velocity field causes there to be no comparison principle, an important element in Alexander, Kim, and Yao's analysis of the patch dynamics.

To overcome these difficulties, we combine new results on the Wasserstein gradient flow of non-semiconvex energies with a refined approximation of the congested aggregation equation by nonlinear diffusion equations to characterize the dynamics of patch solutions and study their asymptotic behavior. 
To address the lack of convexity, we appeal to recent work by the first author, inspired by the present problem, that proves well-posedness of Wasserstein gradient flows for energies that are merely \emph{$\omega$-convex} and provides quantitative estimates on the convergence of the \emph{discrete gradient flow}.  (See section \ref{Wasserstein intro section}). We apply these results to conclude that if the initial data $\rho_0$ satisfies $\|\rho_0\|_\infty \leq 1$, then there exists a unique Wasserstein gradient flow $\rho_\infty$ of the constrained interaction energy $E_\infty$. 
However, due to the low regularity of $E_\infty$, gradient flow theory doesn't provide a characterization of its evolution in terms of a partial differential equation.

Our goal in this paper is to study the dynamics and asymptotic behavior of $\rho_\infty$.  We focus on the case when the initial data $\rho_0$ is a \emph{patch}, i.e. $\rho_0 = \chi_{\Omega_0}$, where $\Omega_0\subseteq \mathbb{R}^d$ is a bounded domain with Lipschitz boundary, and we seek to answer the following questions:
\begin{enumerate}
\item[1.] If $\rho_0$ is a patch, does $\rho_\infty(\cdot,t)$ remain a patch $\chi_{\Omega(t)}$ for all $t\geq 0$?
\item[2.] If so, what partial differential equation determines the evolution of the set $\Omega(t)$?
\item[3.] What is the asymptotic behavior of $\Omega(t)$ as $t\to\infty$?
\end{enumerate}

To answer these questions, we blend the gradient flow approach with viscosity solution theory. 
 Due to the attractive nature of the Newtonian kernel (\ref{Newtonian kernel}), we show that the solution of the congested aggregation equation $\rho_\infty(x,t)$ indeed remains a patch: $\rho_\infty(x,t) = \chi_{\Omega(t)}(x)$ for a time dependent domain $\Omega(t)$. We then show that $\Omega(t)$ evolves with  outward normal velocity $V = V(x,t)$ satisfying
\begin{align*}
V = -\nu\cdot(\nabla p + \nabla\mathbf{N}\rho_{\infty}) \quad\hbox{ at } x\in\partial\Omega(t),
\end{align*}
where $\nu=\nu(x,t)$ is the outward unit normal at $x\in \partial \Omega(t)$ and, for each $t>0$, $p=p(x,t)$ solves\begin{align*}
-\Delta p(\cdot,t) = 1 \hbox { in } \Omega(t),\quad p(\cdot, t) = 0 \hbox{ outside of } \Omega(t).
\end{align*}
Since, $\Omega(t) = \{p(\cdot,t)>0\}$, this gives a Hele-Shaw type free boundary problem for the pressure variable $p$,
\begin{align} \tag*{(P)} \label{P}
\left\{\begin{array}{lll}
-\Delta p(\cdot,t) = 1 &\hbox{ in } & \ \ \{p>0\};\\
V= -\nu\cdot( \nabla p + \nabla \Phi) &\hbox{ on } &\partial\{p>0\};\\
 \Phi= \mathbf{N} \chi_{\{p>0\}}.
 \end{array}\right.
\end{align}
Provided that $p$ is sufficiently regular (for example, $p\in L^1([0,\infty);H^1(\Rd)$), this would imply that the solution of the congested aggregation is a weak solution of the continuity equation
\begin{equation}\label{transport111}
\rho_t = \nabla\cdot(\rho(  \nabla \bN \rho+ \nabla p)) ,
\end{equation}
where $\nabla p$ is the pressure generated by the height constraint that modifies the ``desired velocity field'' $\nabla \bN \rho$.
  In terms of $p$, 
 $\nu = -\nabla p/|\nabla p|$ and $V = p_t/|\nabla p|$, so in the smooth setting the second condition in \ref{P} can be written as 
 $$
 p_t = |\nabla p|^2 + \nabla p \cdot \nabla \Phi \hbox{ on } \partial\{p>0\}.
  $$

Even if $\Omega_0$ has smooth boundary, the evolving set $\Omega(t)=\{p(\cdot,t)>0\}$ may undergo topological changes such as merging.  Consequently, to describe the evolution of $\Omega(t)$, we require a notion of weak solution for \ref{P}. While viscosity solutions are a natural choice, given their utility in free boundary problems, because of the nonlocal dependence of the outward normal velocity $V$ on $p$ itself, \ref{P} lacks a comparison principle. Instead, we consider an auxillary problem for a \emph{fixed}, nonnegative function $\rho(x,t)\in L^{\infty}(\R^d\times (0,\infty))$,
\begin{align} \tag*{(P)$_\infty$} \label{Pinfty}
\left\{\begin{array}{lll}
-\Delta p(\cdot,t) = 1 &\hbox{ in }& \{p>0\};\\
V= -\nu\cdot( \nabla p + \nabla\Phi) &\hbox{ on }&\partial\{p>0\};\\
 \Phi= \mathbf{N}\rho .
 \end{array}\right.
\end{align}
 We show that the comparison comparison principle holds for \ref{Pinfty}, hence viscosity solution theory applies. We then define $p$ to be a solution of \ref{P} if it is a weak viscosity solution of  \ref{Pinfty} with $\rho = \chi_{\{p>0\}}$ almost everywhere.
 
We now state our first main result, which follows from Theorems \ref{existenceofsolntoP} and \ref{convergence}.
\begin{theorem}[Characterization of dynamics of aggregation patches] \label{first main theorem}
~\\[-0.4cm]
\begin{itemize}

\item[(a)]
Let $\Omega_0 \subseteq \R^d$ be a bounded domain with Lipschitz boundary, and let $\rho_{\infty}(\cdot,t) \in L^{\infty}(\R^n)$ be the gradient flow of $E_{\infty}$ with initial data $\chi_{\Omega_0}$. Consider the free boundary problem \ref{Pinfty} with $\rho$ replaced by $\rho_\infty$, and the initial data $p_0$ given by
\begin{equation}
\label{initial}
-\Delta p_0(\cdot,0) = 1 \hbox { in } \Omega_0,\quad p_0(\cdot, 0) = 0 \hbox{ outside of } \Omega_0.
\end{equation}
Then there is a unique minimal viscosity solution $p(x,t)$ of \ref{Pinfty} with initial data $p_0$. 
\item[(b)] Let $\Omega(t) = \{p(\cdot, t)>0\}$. Then $\rho_\infty(\cdot, t)$ remains a patch for all times, and
$$
\rho_{\infty}(\cdot, t)= \chi_{\Omega(t)} \text{ a.e.} \text{ for all }t\geq 0.
$$
\item[(c)] Therefore, $p$ is a weak solution of \ref{P} in the sense of Definition~\ref{weak_def}.
\end{itemize}

\label{main}

\end{theorem}

Next, we consider the asymptotic behavior of patch solutions as $t \to +\infty$. For any given mass and any dimension, the Riesz rearrangement inequality \cite[Theorem 3.7]{LiebLoss} immediately gives that the global minimizer of the constrained interaction energy \eqref{Edef} must be a characteristic function of a ball. However, this does not guarantee that the gradient flow $\rho_\infty(t)$ of the constrained interaction energy always converges to a translation of the global minimizer as $t\to +\infty$. In particular, the main obstacle is to show the mass of $\rho_\infty(t)$ cannot escape to infinity in the long time limit, which requires us to obtain some compactness estimates on $\rho_\infty(t)$ uniformly in time.

For the Keller-Segel equation \eqref{KellerSegel} with subcritical power $m>2-2/d$, the situation is very similar. Again,  there exists a unique (up to a translation) global minimizer of $E_m$ for any given mass \cite{Lions, LiebYau}, but it is unknown whether solutions converge to it as $t\to +\infty$. In dimension two, convergence has been recently shown by Carrillo, Hittmeir, Volzone, and the third author \cite{chvy}, where compactness is obtained via a uniform in time bound of the second moment, though no explicit convergence rate towards the global minimizer is given. For $d\geq 3$, the only available convergence result towards the global minimizer is work by the second two authors on radial solutions \cite{KimYao}.

In our work, for dimension two, we not only prove convergence of solutions towards the global minimizer of the constrained interaction energy \eqref{Edef}, but also provide explicit estimates on the rate of convergence. We accomplish this by again applying a blended approach, combining the gradient flow structure of the problem with viscosity solution theory and the characterization of patch dynamics from Theorem \ref{first main theorem}. We begin by using a rearrangement inequality of Talenti  \cite{Talenti} to show that the second moment of $\rho_\infty(t)$ is non-increasing in time and is strictly decreasing at time $t$ unless $\Omega(t)$ is a disk.  Then, applying a quantitative version of the isoperimetric inequality due to Fusco, Maggi, and Pratelli \cite{fmp} and our characterization of patch dynamics, Theorem \ref{first main theorem}, we provide explicit estimates on the rate that the second moment is decreasing, in terms of the symmetric difference between $\Omega(t)$ and a disk. Finally, using the gradient flow structure of the problem, we show that as $t\to+\infty$, $\rho_\infty(t)$ strongly converges to a characteristic function of a disk in $L^q$ for any $1\leq q<\infty$, and its energy $E_\infty(\rho_\infty)$ converges to its global minimizer with an explicit rate. This gives our second main result, which combines Theorems \ref{thm:E_conv} and \ref{long_time}. 

\begin{theorem}[Long time behavior in two dimensions] \label{main2} 
Assume $d=2$. Let $\Omega_0 \subseteq \R^2$ be a bounded domain with Lipschitz boundary, and let $\rho_{\infty}$ be the gradient flow of $E_{\infty}$ with initial data $\chi_{\Omega_0}$. Then as $t\to +\infty$, $\rho_{\infty}(\cdot,t)$ converges to $\chi_{B_0}$ in $L^q$ for any $1\leq q<\infty$,  where $B_0$ is the unique disk with the same area and center of mass as those of $\Omega_0$.  Furthermore, we have the following rate of convergence in terms of the free energy,
$$
0\leq E_\infty(\rho_{\infty}(\cdot,t)) - E_\infty(\chi_{B_0}) \leq C(|\Omega_0|, M_2[\Omega_0]) t^{-1/6}.
$$
\end{theorem}

\begin{remark}
Let us point out that our control for the second moment relies on the particular structure for the 2D Newtonian kernel, and we are unable to obtain similar compactness estimates for higher dimensions. For $d\geq 3$, whether $\rho_\infty(t)$ converges to a ball as $t\to\infty$ remains an interesting open question.
\end{remark}

We now describe the key ingredients in our characterization of the dynamics of the congested aggregation equation. At the heart of our analysis is an approximation of this equation as the singular limit of a sequence of nonlinear diffusion equations. This provides the bridge between the gradient flow and viscosity solution approach. In particular, while the gradient flow of $E_\infty$ is merely a curve in the space of  measures, approximating it by a sequence of solutions to nonlinear diffusion equations allows us to bring to bear the tools of viscosity solution theory in the limit.

Following the analogy with Alexander, Kim, and Yao's previous work,  one might hope to approximate the congested aggregation equation by the Keller-Segel equation (\ref{KellerSegel}), which also has a gradient flow structure corresponding to the energy
\begin{align} \label{Keller Segel energy}
E_{KS}(\rho)= \iint \rho(x) \mathbf{N}\rho(x) dx + \frac{1}{m-1}\int \rho(x)^m dx 
\end{align}
Note that for a fixed $\rho$, $E_\infty(\rho)$ is the limit of $E_{KS}(\rho)$ as $m\to\infty$.
However, $E_{KS}$
satisfies neither the classical assumptions for semiconvexity nor the weaker assumptions for $\omega$-convexity. Consequently, we lack the quantitative estimates on the rate of convergence of the discrete gradient flow that are an essential element of our approach. Instead, we replace the nonlocal potential $\bN \rho$ in $E_{KS}$ with a local potential $\Phi_{1/m}(x,t)$ that depends on time, the initial data $\rho(x,0)$, and the diffusion parameter $m \geq 1$. (See Definition \ref{drift def} 
 for a precise definition of this potential.) This leads to the energy
\begin{align*}
 E_{m,t}(\rho) :=
  \int \rho(x) \Phi_{1/m}(x,t) dx +\frac{1}{m-1} \int \rho(x)^m dx ,
\end{align*}
which we can show is \emph{$\omega$-convex}.
We then prove that the (time dependent) gradient flow of this energy, which corresponds to a solution of
\begin{align} \tag*{(PME-D)$_{m}$}\label{pme}
\rho_t = \nabla\cdot(\rho\nabla\Phi_{1/m}) + \Delta \rho^m , 
\end{align}
converges as $m \to +\infty$ to a solution of the congested aggregation equation. (See section \ref{convgradflowsec} for our construction of this time dependent gradient flow.) Then, rewriting \ref{pme} in the form 
\begin{align} \label{pme pressure}
\rho_t = \nabla\cdot(\rho (\nabla\Phi_{1/m} + \nabla p_m)) , \quad \hbox{ for } p_m= \frac{m-1}{m}\rho^{m-1}, 
\end{align}
we use viscosity solution theory to show that, as $m \to +\infty$, $p_m$ converges to a solution of the free boundary problem \ref{P}. By uniqueness of the limit, we conclude the characterization of dynamics of patch solutions of the congested aggregation equation, as stated in Theorem \ref{main}.

Our paper is organized as follows.  In section \ref{convgradflowsec}, we prove that the solutions of the nonlinear diffusion equations \ref{pme} converge as $m \to +\infty$ to the gradient flow of $E_{\infty}$ with an explicit rate depending on $m$. We also provide background on Wasserstein gradient flow, including recent results by the first author on the gradient flows of $\omega$-convex energies. In section \ref{convviscsolsec}, we show that the pressure $p_m$ corresponding to the nonlinear diffusion equations, given in equation (\ref{pme pressure}), converges as $m \to +\infty$ to a solution of \ref{Pinfty}. Combining these results, we show that the gradient flow of $E_\infty$ is a characteristic function of the evolving set $\Omega(t)$ and that $\Omega(t)$ can be obtained from the viscosity solution of $(P)_{\infty}$. 
In section \ref{long time section}, we consider the asymptotic behavior of $\rho_\infty$ in two dimensions, proving that it converges to a disk with explicit rate.  Let us remark that the characterization of $\rho_{\infty}$ by the pressure variable $p$ plays a crucial role in the proof of this asymptotic result.
 Finally, we conclude with an appendix section \ref{appendix section}, which contains proofs of several lemmas from section \ref{convgradflowsec}, as well as definitions of viscosity solutions for the limiting free boundary problem $(P)_{\infty}$.
 
 There are several directions for future work. First, our analysis only addresses solutions that are initially a patch. Results for more general initial data could leverage recent work by Kim and Pozar \cite{KimPozar} and Mellet, Perthame, and Quiros \cite{MelletPerthameQuiros}. Second, in the light of work by Maury, Roudneff-Chupin, and Santambrogio \cite{MRS}, it would be interesting if one could characterize the modified velocity $\nabla p + \nabla \mathbf{N}\rho$ in \eqref{transport111} as the projection of the original velocity $\nabla \mathbf{N}\rho$ onto the space of admissible velocities under the height constraint. At the moment, this appears to be a difficult question, due to the highly nonlinear nature of the projection and its dependence on the solution. A third direction for future work would be to pursue to what extent our analysis extends to nonlocal velocity field generated by kernels aside from the Newtonian $\mathcal{N}$, which arise in a range of biological and physical applications. While our result on the singular limit of the nonlinear diffusion equations extends to a range of kernels (see Remark \ref{choice of kernel remark}), our analysis of the free boundary problem strongly leverages the structure of the Newtonian kernel. A final direction for future work would be to make rigorous the link between the congested aggregation equation and the Keller-Segel equation \eqref{KellerSegel} as $m \to +\infty$, completing the analogy with previous work by Alexander, Kim, and Yao that found that the hard height constraint may be obtained as the limit of slow diffusion.

\section{Convergence of gradient flows: drift diffusion to height constrained interaction} \label{convgradflowsec}
In this section, we show that the gradient flow of the height constrained interaction energy $E_\infty$, defined in equation (\ref{Edef}), may be approximated by solutions of the nonlinear diffusions equations \ref{pme} as $m \to +\infty$. This provides a link between the abstract Wasserstein gradient flow of $E_\infty$, which in general is merely a curve in the space of probability measures, and solutions to partial differential equations.

\begin{remark}[Choice of interaction kernel] \label{choice of kernel remark}
For the sake of continuity with sections \ref{convviscsolsec} and \ref{long time section}, we assume that the interaction kernel $\mathcal{N}$ is Newtonian (\ref{Newtonian kernel}). However, our results in this section may be extended to any kernels that satisfy \cite[Assumption 4.1]{CraigOmega} and the estimates of Proposition~\ref{N mu bounds}. In particular, this includes many repulsive-attractive potentials of interest in the literature.
\end{remark}

\subsection{Preliminary results} \label{Wasserstein intro section}

We begin by collecting some results on the Wasserstein gradient flow of $\omega$-convex energies that will be useful in what follows. For further background on the Wasserstein metric and gradient flows of semiconvex energies, we refer the reader to the books by Ambrosio, Gigli, and Savar\'e \cite{AGS} and Villani \cite{Villani}. For more details on gradient flows of $\omega$-convex energies, see recent work by the first author \cite{CraigOmega}.

Let $\P_2(\Rd)$ denote the set of probability measures on $\Rd$ with finite second moment, i.e. $\int |x|^2 d\mu < +\infty$. If a measure $\mu \in \P_2(\Rd)$ is absolutely continuous with respect to Lebesgue measure ($\mu  \ll \mathcal{L}^d $), we will identify $\mu$ with its density, i.e. $d \mu(x) = \mu(x) dx$. In particular, we write $\|\mu\|_{L^\infty} < +\infty$ if $d\mu(x) = \mu(x) dx$ and $\mu(x) \in L^\infty(\Rd)$.

Given $\mu, \nu \in \P_2(\Rd)$, a measurable function $\bt: \Rd \to \Rd$ \emph{transports $\mu$ onto $\nu$} in case $\int f(\bt(x)) d \mu = \int f(y) d \nu$ for all $f \in L^1(d \nu)$. We then call $\nu$ the \emph{push-forward of $\mu$ under $\bt$} and write $\nu = \bt \# \mu$. If $\mu$ is absolutely continuous with respect to Lebesgue measure (as will be the case for all the measures we consider), then the \emph{Wasserstein distance} from $\mu$ to $\nu$ is given by 
\begin{equation}
\label{eq:def_w2}
 W_2(\mu,\nu) = \inf \left\{ \left( \int |\bt - \id|^2 d \mu \right)^{1/2} : \bt \# \mu = \nu \right\} , 
 \end{equation}
where $\id(x) = x$. Furthermore, the infimum is attained by an \emph{optimal transport map} $\bt = \bt_\mu^\nu$, which is unique $\mu$-almost everywhere.

The metric space $(\P_2(\Rd), W_2)$ is complete, and convergence can be characterized as
\begin{center}
	\begin{tabular}{lcl}
		$W_2(\mu_n,\mu)\to 0$ &  $\iff$  & $\int fd\mu_n \to \int f d\mu$ for all $f\in C(\Rd)$ such that \\[0.25cm]
		&   & $\exists C>0,x_0 \in \Rd$ so that $|f(x)| \leq C(1+ |x-x_0|^2)$.
	\end{tabular}
\end{center}
We will refer to such $f$ as \emph{continuous functions with at most quadratic growth.} Furthermore, for any $f \in C^1(\Rd)$ with uniformly bounded gradient, we can quantify the difference between the integral of $f$ against $\mu$ and the integral of $f$ against $\nu$ using the following elementary lemma.
\begin{lemma} \label{drift continuity estimate0}
For $f \in C^1(\Rd)$ and $\mu,\nu \in \P_2(\Rd)$,
\[ \left| \int  f d \mu - \int f d \nu  \right| \leq \| \grad f \|_\infty W_2(\mu, \nu) .  \]
\end{lemma}
\begin{proof}
For simplicity, suppose that $\mu \ll \mathcal{L}^d$, so there exists an optimal transport map  $\bt_\mu^\nu$. (The proof is identical in the general case, using optimal transport plans.) By Jensen's inequality,
\begin{align*}
\left| \int f  d \mu- \int  f d \nu \right| &\leq \int \left|f - f\circ\bt_\mu^\nu \right| d \mu \leq \|\grad f\|_\infty \left( \int |\bt_\mu^\nu - \id|^2 d \mu \right)^{1/2} = \|\grad f\|_\infty W_2(\mu,\nu) .
\end{align*}
\end{proof}

Along with its metric structure, $(\P_2(\Rd), W_2)$ is a \emph{geodesic space}, since any two measures $\mu_0, \mu_1 \in \P_2(\Rd)$ are connected by a \emph{geodesic} $\mu_\alpha \in \P_2(\Rd)$, $\alpha \in [0,1]$, satisfying
\[ W_2(\mu_\alpha, \mu_\beta) = |\beta - \alpha| W_2(\mu_0, \mu_1) \text{ for all } \alpha, \beta \in [0,1] . \]
If $\mu_0 \ll \mathcal{L}^d$, then the geodesic from $\mu_0$ to any $\mu_1 \in \P_2(\Rd)$ is unique and of the form 
\[ \mu_\alpha = ((1-\alpha) \id + \alpha \bt_{\mu_0}^{\mu_1}) \# \mu_0 . \]
 Unlike a square Hilbertian norm, the square Wasserstein distance is not convex along geodesics ($\alpha \mapsto W_2^2(\nu,\mu_\alpha)$ is not convex) \cite[Example 9.1.5]{AGS}. Consequently, Ambrosio, Gigli, and Savar\'e introduced an expanded class of curves known as \emph{generalized geodesics}, so that, between any two measures, there is always at least one curve along which the square distance is convex \cite[Lemma 9.2.1, Definition 9.2.2]{AGS}. Given $\mu_0,\mu_1,\nu\in \P_2(\Rd)$ with $\nu \ll \mathcal{L}^d$, the \emph{generalized geodesic from $\mu_0$ to $\mu_1$ with base $\nu$} is
 \[ \mu_\alpha = ((1-\alpha) \bt_\nu^{\mu_0} + \alpha \bt_\nu^{\mu_1}) \# \nu , \]
 and along such a curve we have
 \[ W_2^2(\nu,\mu_\alpha) = (1-\alpha) W_2^2(\nu, \mu_0) + \alpha W_2^2(\nu,\mu_1) - \alpha (1-\alpha) \|\bt_{\nu}^{\mu_0} - \bt_\nu^{\mu_1} \|_{L^2(d \nu)}^2 . \]
An additional class of curves along which the square Wasserstein metric is convex are \emph{linear interpolations} of measures,
\[ \mu_\alpha := (1-\alpha) \mu_0 + \alpha \mu_1 . \]
 For any $\mu_0,\mu_1,\nu\in \P_2(\Rd)$, we have
 \begin{equation}
 \label{W2 convex}
 W_2^2(\nu, \mu_\alpha) \leq (1-\alpha)W_2^2(\nu,\mu_0) + \alpha W_2^2(\nu,\mu_1).
 \end{equation}
(See, for example, \cite[Proposition 7.19]{Santambrogio}.)


Due to the fact that $(\P_2(\Rd),W_2)$ is a geodesic space, it induces a natural notion of convexity on energy functionals $E: \P_2(\Rd) \to \R \cup \{ +\infty \}$, i.e. given a geodesic $\mu_\alpha$, the function $\alpha \mapsto E(\mu_\alpha)$ is convex. We recall both this standard notion of convexity, as well as two generalizations: semiconvexity and $\omega$-convexity.
\begin{enumerate}[label = (\roman*)]
\item $E$ is \emph{convex} along $\mu_\alpha$ if $E(\mu_\alpha) \leq (1-\alpha) E(\mu_0) + \alpha E(\mu_1)$.
\item $E$ is \emph{semiconvex} along $\mu_\alpha$ if there exists $\lambda \in \R$ so that \\ $E(\mu_\alpha) \leq (1-\alpha) E(\mu_0) + \alpha E(\mu_1) - \alpha (1-\alpha) \frac{\lambda}{2} W_2^2(\mu_0,\mu_1)$.
\item $E$ is \emph{$\omega$-convex} along $\mu_\alpha$ if there exists $\lambda_\omega \in \R$ and a continuous, nondecreasing function $\omega:[0, +\infty) \to [0,+\infty)$, which vanishes only at $x=0$, so that \\ $E(\mu_\alpha) \leq (1-\alpha) E(\mu_0) + \alpha E(\mu_1) - \frac{\lambda_\omega}{2} [ (1-\alpha) \omega(\alpha^2 W^2_2(\mu_0,\mu_1)) + \alpha \omega((1-\alpha) W_2^2(\mu_0,\mu_1))]$.
\end{enumerate}
If, for any $\mu_0, \mu_1 \in \P_2(\Rd)$, there exists a geodesic $\mu_\alpha$ from $\mu_0$ to $\mu_1$ along which $E$ satisfies one of the above inequalities, we say $E$ is \emph{convex/semiconvex/$\omega$-convex along geodesics}. On the other hand, if for any $\mu_0, \mu_1, \nu \in \P_2(\Rd)$, there exists a generalized geodesic $\mu_\alpha$ from $\mu_0$ to $\mu_1$ with base $\nu$ along which $E$ satisfies one of the above inequalities (replacing $W_2(\mu_0,\mu_1)$ on the right hand side with $\|\bt_\nu^{\mu_0} - \bt_\nu^{\mu_1}\|_{L^2(\nu)}^2$), we say $E$ is \emph{convex/semiconvex/$\omega$-convex along generalized geodesics}. We will also say that $E$ is \emph{proper} if the \emph{domain} of the energy $D(E) = \{ \mu : E(\mu)<+\infty\}$ is nonempty.

A key element of our analysis is that the height constrained interaction energy $E_\infty$ defined in equation (\ref{Edef}) is $\omega$-convex along generalized geodesics. This follows from the following estimates on the Newtonian potential of a bounded, integrable function.
\begin{proposition}[{c.f.\cite[Theorem 2.7]{Loeper}}]\ \label{N mu bounds}
Suppose $\rho, \mu,\nu \in \P_2(\Rd)$ with $\|\rho\|_\infty , \|\mu\|_\infty \leq 1$. Then there exists $C_d \geq 1$, depending only on the dimension, so that 
\begin{align*}  &\|\grad \bN \rho\|_\infty \leq C_d , \ \  \| \Delta \bN \rho \|_\infty \leq 1, \  \ \int \bN \rho d \nu \geq -C_d , \\
  & |\grad \bN \rho(x) - \grad  \bN \rho(y)| \leq C_d \sigma(|x-y|) , \text{ and }  \|\grad \bN \rho - \grad \bN \mu \|_{L^2(\Rd)} \leq W_2(\rho, \mu) .
 \end{align*}
where
\begin{align} \label{sigmadef}
\sigma(x) &:= \begin{cases} 2x |\log x| & \text{ if } 0\leq x \leq e^{(-1-\sqrt{2})/2} , \\ \sqrt{x^2+ 2(1+\sqrt{2})e^{-1-\sqrt{2}}} & \text{ if }x > e^{(-1-\sqrt{2})/2} . \end{cases}
\end{align}
\end{proposition}
\noindent We defer the proof of this proposition to the appendix in Section \ref{further properties of energies}.

By the above estimates and \cite[Theorem 4.3, Proposition 4.4]{CraigOmega}, $E_\infty$ is $\omega$-convex along generalized geodesics with $\lambda_\omega = -C_d$ and $\omega(x)$ a log-Lipschitz modulus of convexity 
\begin{align} \label{omega def}
\omega(x) = \begin{cases} x |\log x| & \text{ if } 0 \leq x \leq e^{-1-\sqrt{2}} , \\ \sqrt{x^2+ 2(1+\sqrt{2})e^{-1-\sqrt{2}}x} & \text{ if }x > e^{-1-\sqrt{2}}.  \end{cases}
\end{align}

The $\omega$-convexity of $E_\infty$ then leads to the following result on the well-posedness of the gradient flow:
\begin{theorem}[{\cite[Theorem 4.3, Proposition 4.4]{CraigOmega}}]
\label{grad_flow_well_posed}
For any $\rho_0\in D(E_\infty)$ (that is, $\rho_0\in \mathcal{P}_2(\mathbb{R}^d)$ with $\|\rho_0\|_\infty \leq 1$), the gradient flow $\rho_\infty(t)$ of $E_\infty$ with initial data $\rho_0$ is well-posed. Specifically $\rho_\infty: (0,+\infty) \to \P_{2}(\Rd)$ is the unique curve that is locally absolutely continuous in time, with $\rho_\infty(t) \xrightarrow{t \to 0} \rho_0$ and
\begin{align} \label{continuous evi}
\frac{1}{2} \frac{d}{dt} W_2^2(\rho_\infty(t), \nu) + \frac{\lambda_\omega}{2} \omega( W_2^2( \rho_\infty(t), \nu)) \leq E(\nu) - E( \rho_\infty(t))  ,  \ \forall \nu \in D(E_\infty) , \text{ a.e. }  t >0 .
\end{align}
\end{theorem}

In order to provide a PDE characterization of $\rho_\infty(x,t)$ in Section \ref{convviscsolsec}, we use the following  higher regularity of $\rho_\infty(x,t)$ and $\grad \bN \rho_\infty(x,t)$, which we prove in appendix Section \ref{further properties of energies}.

\begin{proposition}[time regularity of the gradient flow of $E_\infty$] \label{W2 time lipschitz}
Suppose $\rho_\infty(x,t)$, with initial data $\rho_\infty(x, 0) \in D(E_\infty)$, is a gradient flow of $E_\infty$. Then $W_2(\rho_\infty(t), \rho_\infty(s)) \leq 2C_d |t-s|$, where $C_d > 0$ is as in Proposition \ref{N mu bounds}.
\end{proposition}

\begin{proposition}  \label{propertiesofspecialPhi}
Suppose $\rho_\infty(x,t)$, with initial data $\rho_\infty(x, 0) \in D(E_\infty)$, is a gradient flow of $E_\infty$. Then  $\grad \bN \rho_\infty(x,t)$ is log-Lipschitz in space and $1/2d$-H\"older continuous in time. In particular, with $C_d >0$ and $\sigma(x)$ as in Proposition \ref{N mu bounds},
\begin{align*} |\grad \bN\rho_\infty(x,t) - \grad \bN \rho_\infty(y,t)| &\leq C_d \sigma(|x-y|) \quad \text{ for all } x, y \in \Rd , t \geq 0 , \\
| \grad \bN \rho_\infty(x,t) - \grad \bN \rho_\infty(x,s)| &\leq 10C_d|t-s|^{1/2d}  \text{ for all }  0< |t-s| < e^{(-1-\sqrt{2})/2} , \  x \in \Rd .\end{align*}
\end{proposition}

An important tool in the analysis of Wasserstein gradient flows is a discrete time approximation of gradient flows known as the \emph{discrete gradient flow} or \emph{JKO scheme} \cite{JKO}. This scheme is analogous to the implicit Euler method for approximation of ordinary differential equations in Euclidean space. For any $\mu \in D(E_\infty)$ and time step $\tau>0$, the \emph{discrete gradient flow} of $E_\infty$ is given by 
\begin{align*} 
\rho^n_\tau \in \argmin_{\nu\in \PR} \left\{ \frac{1}{2\tau} W_2^2(\rho_\tau^{n-1}, \nu) + E_\infty(\nu)\right\} \text{ and }  \rho^0_\tau := \rho . 
\end{align*}
By \cite[Theorem 4.3, Proposition 4.4]{CraigOmega}, the discrete gradient flow of $E_\infty$ exists for all $\rho \in D(E_\infty)$ and $\tau>0$, and if $\tau = t/n$ for any $t\geq 0$, the discrete gradient flow converges to the continuous gradient flow,
\[ \lim_{n \to +\infty} W_2(\rho^n_{t/n}, \rho_\infty(t)) = 0. \]

As demonstrated in previous work by the first author \cite{CraigOmega}, well-posedness of the gradient flows of $\omega$-convex eneriges is closely related to the well-posedness of the ODE
\begin{align} \label{omega ode}
 \begin{cases} \frac{d}{dt} F_t(x) = -C_d \omega(F_t(x)) , \\ F_0(x) = x . \end{cases} 
 \end{align}
 For $\omega(x)$ as in equation (\ref{omega def}), $0 \leq x \leq e^{-1-\sqrt{2}}$, and $t \geq 0$, the solution is given by $F_t(x) = x^{e^{C_d t}}$. Furthermore, for all $x,t \geq 0$, $F_t(x)$ is nondecreasing in space and nonincreasing in time.
 
In a similar way, analysis of the discrete gradient flow of $E_\infty$ is closely related to a discrete time approximation of (\ref{omega ode}). In particular, we define
\[ f_\tau(x) := \begin{cases} x- C_d \tau \omega(x) &\text{ if } x \geq 0 , \\ 0 &\text{ if } x \leq 0 , \end{cases}\]
so that $f^{(m)}_\tau(x)$ is the $m$th step of the explicit Euler method with time step $\tau$. In the following proposition, we recall some properties of the function $f_\tau(x)$ that will be useful in our estimates of the discrete time sequences.
\begin{proposition}[properties of $f_\tau(x)$] \label{f tau prop} \
\begin{enumerate}[label = (\roman*)]
\item If $0 \leq x \leq y \leq r$, there exists $c_r >0$ so that $f_\tau(x) \leq f_\tau(y) + C_d^2 c_r^2 \tau^2$.  \label{f tau monotone}
\item For all $x, y \geq 0$, $f_\tau(x+y) \leq f_\tau(x) +y$. \label{f tau break}
\item For all $x, t \geq 0$, $|F_t(x) - f^{(n)}_{t/n}(x)| \leq C_d\omega(x)t/ n$. \label{ode euler estimate}
\end{enumerate}
\end{proposition}
\begin{proof}
\ref{f tau monotone} and \ref{f tau break} are consequences of \cite[Lemma 2.25]{CraigOmega}. \ref{ode euler estimate} is a consequence of \cite[Proposition 2.24]{CraigOmega} and the fact that $F_t(x)$ is nonincreasing in time.
\end{proof}

Finally, we recall a contraction inequality for the discrete gradient flow of an $\omega$-convex energy, which we use to conclude stability of the discrete gradient flow sequences.
\begin{proposition}[contraction inequality] \label{contraction inequality}
Let $E: \P_2(\Rd) \to \R \cup \{+\infty \}$ be proper, lower semicontinuous, bounded below, and $\omega$-convex along generalized geodesics, for $\omega(x)$ as in equation (\ref{omega def}) and $\lambda_\omega \leq 0$. Fix $\rho,\mu \in D(E)$ and, for $\tau>0$, choose $\rho_\tau$ and $\mu_\tau$ satisfying
\begin{align*} 
\rho_\tau \in \argmin_{\nu\in \PR} \left\{ \frac{1}{2\tau} W_2^2(\rho, \nu) + E(\nu)\right\} \text{ and }  \mu_\tau \in \argmin_{\nu\in \PR} \left\{ \frac{1}{2\tau} W_2^2(\mu, \nu) + E(\nu)\right\}. 
\end{align*}
Then there exist positive constants $C$ and $\tau_*$ depending on $W_2(\rho,\mu)$, $\lambda_\omega$, $E(\mu)$, and $E(\nu)$ so that for all $0 < \tau < \tau_*$,
\begin{align*}
f_\tau^{(2)}(W_2^2(\rho_\tau,\mu_\tau)) \leq W_2^2(\rho,\mu) + |\lambda_\omega| \tau \omega(CW_2(\mu,\mu_\tau)) + 2 \tau (E(\rho) - E(\rho_\tau)) + C \tau^2 .
\end{align*}
\end{proposition}

\begin{proof}
This is a particular case of \cite[Theorem 3.2]{CraigOmega}.
\end{proof}

\begin{remark}[Wasserstein gradient flow of measures with mass not equal to 1] \label{mass not one}
We conclude by observing that the gradient flow theory can be easily extended to nonnegative measures whose integral is not equal to 1. For a fixed $A>0$, let $\mathcal{P}_{2,A}(\mathbb{R}^d)$ denote the set of non-negative measures that integrate to $A$ and have finite second moment. For $\mu, \nu \in \mathcal{P}_{2,A}(\mathbb{R}^d)$ (with the same $A$), we can then define $W_2(\mu, \nu)$ in the same way as in \eqref{eq:def_w2}, and given initial data $\rho_0 \in \mathcal{P}_{2,A}(\mathbb{R}^d)$ with $\|\rho_0\|_\infty \leq 1$, the same arguments lead to the well-posedness of a gradient flow $\rho_\infty: (0,+\infty) \to \mathcal{P}_{2,A}(\mathbb{R}^d)$ of $E_\infty$. However, for the sake of simplicity, we will assume that $\rho_0$ is a probability measure for the remainder of this section.
\end{remark}

\subsection{Definitions of energies and discrete time sequences}

We now turn to the definitions of the energies and discrete time sequences that we will use to show that 
solutions of the nonlinear diffusions equations \ref{pme} converge as $m \to +\infty$ to the the gradient flow of the height constrained interaction energy $E_\infty$. 
We begin by defining the local potential $\Phi_{1/m}(x,t)$, which induces the drift in \ref{pme}. As described in the introduction, previous work by Alexander, Kim, and Yao suggests that the gradient flow of $E_\infty$ should be obtained as the limit of the gradient flows of the Keller-Segel energy $E_{KS}$, defined in equation (\ref{Keller Segel energy}). However, we lack sufficient convexity of $E_{KS}$ to prove this rigorously. Instead, we replace the nonlocal potential $\bN \rho$ in $E_{KS}$ with a local potential $\Phi_{1/m}(x,t)$ that depends on time, the initial data $\rho_0(x)$ of the gradient flow of $E_\infty$, and the diffusion parameter $m \geq 1$.

 \begin{definition}[local potential $\Phi_{1/m}(x,t)$] \label{drift def}
Given initial data $\rho_0$, let $\rho_\infty(x,t)$ be the gradient flow of the height constrained interaction energy $E_\infty$.
Fix a mollifier $\psi\in C_c^\infty(\Rd)$ satisfying $\psi \geq 0$ and $\int \psi = 1$, and let $\psi_{1/m}(x) = m^d \psi(mx)$. Then, for any $m > 1$, define
\begin{align} \label{Phidef}
\Phi(x,t) =\bN \rho_\infty(x,t) \text{ and }
 \Phi_{1/m}(x,t) =  \psi_{1/m}* \bN \rho_\infty(x,t) .
\end{align}
\end{definition}

This definition is guided by the following intuition: given initial data $\rho_0$, one heuristically expects that the gradient flow of $E_{KS}$ should converge to $\rho_\infty$. Consequently, if we replace $\bN \rho$ in the definition of $E_{KS}$ by $\bN \rho_\infty$, we expect that the gradient flow of this new energy will still converge to $\rho_\infty$ as $m \to +\infty$.
We include the extra mollification on the potential to  leverage the existing theory on the porous medium equation with drift, which requires the potential to be twice continuously differentiable in space. By Proposition \ref{propertiesofspecialPhi}, $\grad \Phi(x,t)$ is log-Lipschitz in space, hence $\grad \Phi_{1/m}= \psi_{1/m} * \grad \Phi$ converges to $\grad \Phi$ uniformly on $\Rd \times [0, +\infty)$. Furthermore, by Proposition \ref{N mu bounds},
\begin{align}
\label{eq:phi_bound}
 \| \nabla\Phi_{1/m}(\cdot, t) \|_\infty & \leq C_d,  \quad \| \Delta  \Phi_{1/m}(\cdot,t) \|_\infty  \leq 1.
\end{align}

With this precise definition of the drift arising in \ref{pme} in hand, we now turn to the definitions of the the three energy functionals that we use in our analysis of the limit of \ref{pme} as $m \to +\infty$.
\begin{definition}[energy functionals]
Fix $\psi \in C^\infty_c(\Rd)$  as in Definition \ref{drift def} and $\mu \in \P_2(\Rd)$ with $\|\mu\|_\infty \leq 1$. For any $\rho \in \P_2(\Rd)$,  define
\begin{align*}
 E_\infty(\rho) &:= \begin{cases}
\frac{1}{2}\int \mathbf{N}\rho(x) d\rho(x) &\text{if $\| \rho\|_\infty \leq 1$,} \\ +\infty & \text{otherwise} ; \end{cases} \\
\tilde E_\infty(\rho; \mu) &:= 
\begin{cases}
\int \bN \mu(x) d\rho(x)&\text{if } \|\rho\|_\infty \leq 1,\\
+\infty &\text{otherwise};
\end{cases} \\
E_{m}(\rho; \mu) &:= \begin{cases} \frac{1}{m-1} \int \rho(x)^m dx +  \int \psi_{1/m} *\bN \mu(x) d \rho(x) &\text{ if } \rho  \ll \mathcal{L}^d , \\ +\infty &\text{ otherwise. } \end{cases}
\end{align*}
\end{definition}
As shown in previous work by the first author, the gradient flows of the above energies are well-posed  \cite[Theorem 4.3, Proposition 4.4]{CraigOmega}. In particular, while these energies fall outside the scope of the theory of gradients flows of semiconvex energies, all three energies are instead \emph{$\omega$-convex along generalized geodesics} for $\lambda_\omega = -C_d$, as in Proposition \ref{N mu bounds}, and $\omega(x)$  a log-Lipschitz modulus of convexity, as in equation (\ref{omega def}). The third energy is also $\lambda$-convex along generalized geodesics for $\lambda = \lambda(m) \xrightarrow{m \to +\infty} -\infty$ \cite[Proposition 9.3.2, Proposition 9.3.9]{AGS}.

Corresponding to these energies, we consider the following discrete time sequences.
\begin{definition}[discrete time sequences] \label{Discrete time def}
For a fixed time step $\tau >0$ and $\rho \in D(E_\infty)$, define
\begin{enumerate}[label = (\roman*)]
\item \emph{discrete gradient flow of $E_\infty$}: \label{Einfty DGF}
\begin{align*} 
\rho^n_\tau \in \argmin_{\nu\in \PR} \left\{ \frac{1}{2\tau} W_2^2(\rho_\tau^{n-1}, \nu) + E_\infty(\nu)\right\} \text{ and }  \rho^0_\tau := \rho . 
\end{align*}
\item \emph{time varying discrete gradient flow of $\tilde{E}_\infty$:} for $\rho^n_\tau$ as in \ref{Einfty DGF}, \label{tildeEinfty DGF}
\[ \tilde \rho_\tau^n \in \argmin_{\nu\in \PR} \left\{ \frac{1}{2\tau} W_2^2(\tilde \rho_\tau^{n-1}, \nu) + \tilde E_\infty(\nu; \rho_\tau^n)\right\} \text{ and } \tilde \rho^0_\tau := \rho. \]
\item \label{Em DGF} \emph{time varying discrete gradient flow of $E_m$:} for $\rho^n_\tau$ as in \ref{Einfty DGF}  and $m > 1$,
\[  \rho_{\tau,m}^n \in \argmin_{\nu \in \PR} \left\{ \frac{1}{2\tau} W_2^2(\rho_{\tau,m}^{n-1}, \nu) +  E_{m}(\nu; \rho_\tau^n)\right\} \text{ and } \rho^0_{\tau,m} := \rho. \]
\end{enumerate}
\end{definition}
The existence of the above sequences is guaranteed by \cite[Theorem 4.3, Proposition 4.4]{CraigOmega}. However, they are not necessarily unique, and we use the notation $\rho^n_\tau, \tilde{\rho}^n_\tau$, and $\rho^n_{\tau,m}$ to denote any such sequence. Still, using Proposition \ref{contraction inequality}, which provides a contraction inequality for $\omega$-convex functions, we can at least bound the Wasserstein distance between any two such sequences---for example, see Proposition \ref{distance between DGFs Einfty tilde} in the appendix for such an estimate for $\tilde E_\infty$.

If one takes $\tau = t/n$ for $t \geq 0$, then as $n \to +\infty$
the discrete gradient flow of $E_\infty$ converges to the continuous gradient flow of $E_\infty$ with initial data $\rho_\infty(0) = \rho$ \cite[Theorem 4.3, Proposition 4.4]{CraigOmega}. Likewise, $\rho_{t/n,m}^n$ converges to a solution of the nonlinear diffusion equations \ref{pme}, which we denote by $\rho_m(x,t)$, with the same initial data  (see Proposition \ref{time dependent assumption prop}) .
 We refer to $\rho_{\tau,m}^n$ as the ``time varying'' discrete gradient flow of $E_m$ since we change the second argument of $E_m(\cdot; \cdot)$ at each step of the sequence to accommodate the time dependent drift in \ref{pme}.
 
The main goal of this section is to show that $\lim_{m \to +\infty} W_2(\rho_\infty(t), \rho_m(t)) = 0$, which we accomplish by showing that the distance between the sequences $\rho^n_\tau$ and $\rho^n_{\tau,m}$ becomes arbitrarily small as $m \to +\infty$. We use the sequence $\tilde \rho_\tau^n$, defined in \ref{tildeEinfty DGF} above, to serve as a bridge between the two.
In what follows, we will often use the crude estimate $\omega(x) \leq \sqrt{x}$, for $x \geq 0$ sufficiently small. Consequently, the rate of convergence we obtain for $\rho_m(t) \to \rho_\infty(t)$ is certainly not sharp, but the inequalities are much simpler.

We close this introductory section with a few elementary estimates on the above discrete time sequences. In these estimates, as well as in what follows, it will be useful to consider one step of the above sequences:
\begin{definition}[one step minimizers] \label{one step def}
For a fixed time step $\tau>0$, we define
\begin{enumerate}[label = (\roman*)]
\item \emph{one step of discrete gradient flow of $E_\infty$}: given $\rho \in \P_2(\Rd)$,
\begin{align*} 
\rho_\tau \in \argmin_{\nu \in \PR} \left\{ \frac{1}{2\tau} W_2^2(\rho, \nu) + E_\infty(\nu)\right\}
\end{align*}
\item \emph{one step of discrete gradient flow of $\tilde{E}_\infty(\cdot; \mu)$:} given $\rho \in \P_2(\Rd)$ and $\mu \in \P_2(\Rd)$ with $\|\mu\|_\infty \leq 1$, \label{tildeEinfty one step}
\[ \tilde \rho_\tau \in \argmin_{\nu \in \PR} \left\{ \frac{1}{2\tau} W_2^2(\rho, \nu) + \tilde E_\infty(\nu; \mu)\right\} . \]
\item \emph{one step of discrete gradient flow of $E_m(\cdot; \mu)$:} given $\rho \in \P_2(\Rd)$,  $\mu \in \P_2(\Rd)$ with $\|\mu\|_\infty \leq 1$,  and $m > 1$, \label{Em one step}
\[  \rho_{\tau,m} \in \argmin_{\nu \in \PR} \left\{ \frac{1}{2\tau} W_2^2(\rho, \nu) +  E_{m}(\nu;\mu)\right\} . \]
\end{enumerate}
As before, \cite[Theorem 4.3, Proposition 4.4]{CraigOmega} ensures these minimization problems admit at least one solution. Again, these minimizers are not necessarily unique, and we use the notation $\rho_\tau, \tilde{\rho}_\tau$, and $\rho_{\tau,m}$ to denote any such minimizer.
\end{definition}

First, we estimate how the Wasserstein distance, energies, and $L^m$ norms  behave under one step of the discrete gradient flow.

\begin{lemma} \label{W2 one step}
Fix $\rho, \mu \in \P_2(\Rd)$ with $\|\mu \|_\infty \leq 1$. Then for $C_d>0$ as in Proposition \ref{N mu bounds} and any $\tau>0$ and $m \geq 2$,
\begin{enumerate}[label = (\roman*)]
\item If $\|\rho\|_\infty \leq 1$, then  \label{rho tau one step}  $W_2(\rho_\tau, \rho) \leq  2C_d \tau\text{ and } E_\infty(\rho) \leq E_\infty(\rho_\tau) + 2 C_d^2 \tau ;$
\item \label{tilde rho tau one step}  If $\|\rho\|_\infty \leq 1$, then $ W_2(\tilde \rho_\tau, \rho) \leq  2C_d \tau \text{ and }   \tilde E_\infty(\rho; \mu) \leq \tilde E_\infty(\tilde \rho_\tau;\mu) +  2 C_d^2 \tau;$
\item \label{rho m tau one step} For all $\rho \in \P_2(\Rd)$, \\ $W_2(\rho_{\tau,m}, \rho) \leq \sqrt{\frac{2\tau}{m-1} ( \| \rho\|_m^m - \|\rho_{\tau,m} \|_m^m)} + 2C_d \tau$, \quad
 $\frac{1}{m-1} \|\rho_{\tau,m}\|^m_m \leq \frac{1}{m-1} \|\rho\|_m^m + \frac{\tau}{2} C_d^2$, \\ and
$  E_m(\rho;\mu) \leq  E_m(\rho_{\tau,m}; \mu)  + \left(\|\rho\|_m^m - \|\rho_{\tau,m}\|_m^m\right) + C_d \sqrt{\frac{2 \tau}{m-1}  \| \rho\|_m^m }+  2 C_d^2 \tau
$.
\end{enumerate}
\end{lemma}

\begin{proof}
We begin with \ref{tilde rho tau one step}. Taking $\nu = \rho$ in the definition of $\tilde \rho_\tau$ and rearranging,
\begin{align*} 
 W_2^2(\tilde \rho_{\tau},\rho ) \leq 2 \tau \left( \tilde E_\infty (\rho;\mu) - \tilde E_\infty(\tilde\rho_\tau;\mu) \right) .
\end{align*}
Thus, applying Lemma \ref{drift continuity estimate0}, with $f =  \bN \mu$, and Proposition \ref{N mu bounds}
\[ 0 \leq \tilde E_\infty (\rho;\mu) - \tilde E_\infty(\tilde\rho_\tau;\mu)  = \int \bN \mu d\rho - \int \bN \mu d \tilde \rho_\tau \leq C_d W_2(\tilde \rho_\tau, \rho) . \]
Combining the above two inequalities gives the results.

Next, we show \ref{rho tau one step}. Again, taking $\nu = \rho$ in the definition of $\rho_\tau$,
\begin{align*}
W_2^2(\rho_\tau,\rho) \leq 2 \tau(E_\infty(\rho) - E_\infty(\rho_\tau)) .
\end{align*}
Thus, applying Lemma \ref{drift continuity estimate0}, with $f =  \bN \rho$ and $f=  \bN \rho_\tau$, along with Proposition \ref{N mu bounds},
\begin{align*}
0 \leq E_\infty(\rho) - E_\infty(\rho_\tau) = \frac{1}{2} \left(  \int \bN \rho d\rho- \int \bN \rho d\rho_\tau + \int \bN \rho_\tau d\rho  - \int \bN \rho_\tau d \rho_\tau \right) \leq C_d W_2(\rho_\tau, \rho) .
\end{align*}
Combining the above two inequalities again give the results.

It remains to show \ref{rho m tau one step}. For simplicity of notation, let $\Phi_{1/m} = \psi_{1/m}* \bN \mu$. Taking $\nu = \rho$ in the definition of $\rho_{\tau,m}$,
\begin{align} \label{rho m tau one step1}
\frac{1}{2\tau} W_2^2(\rho,\rho_{\tau,m}) + E_m(\rho_{\tau,m}; \mu) \leq E_m(\rho; \mu) .
\end{align}
By definition of $E_m$, Lemma \ref{drift continuity estimate0} with $f =  \Phi_{1/m}$, and Proposition \ref{N mu bounds}, this implies
\begin{align*} 
 \|\rho_{\tau,m}\|^m_m /(m-1)&\leq \|\rho\|^m_m /(m-1) + \left(\int  \Phi_{1/m} \rho -  \int  \Phi_{1/m} \rho_{\tau,m}  \right) -  W_2^2(\rho,\rho_{\tau,m})/(2 \tau) \\
& \leq  \|\rho\|^m_m/(m-1)  +C_d W_2(\rho,\rho_{\tau,m})- W_2^2(\rho,\rho_{\tau,m})/(2 \tau) \nonumber \\
&= \|\rho\|^m_m/(m-1)  - \left(W_2(\rho,\rho_{\tau,m}) - \tau C_d \right)^2 /(2 \tau) + \tau C_d ^2/2 \nonumber
\end{align*}
Dropping the negative term shows the second inequality. Rearranging gives
\begin{align*}
\left(W_2(\rho,\rho_{\tau,m})  - \tau C_d \right)^2 /(2 \tau)&\leq \|\rho\|^m_m/(m-1)  - \|\rho_{\tau,m}\|^m_m/(m-1) + \tau C_d ^2/2 ,
\end{align*}
which, by the subadditivity of $\sqrt{\cdot}$, gives the first inequality.

To show the third inequality, we combine (\ref{rho m tau one step1}) with Lemma \ref{drift continuity estimate0} and use the previous estimate on the Wasserstein distance,
\begin{align*}
0 &\leq E_m(\rho; \mu) - E_m(\rho_{\tau,m}; \mu) \leq \frac{1}{m-1}\left(\|\rho\|_m^m - \|\rho_{\tau,m}\|_m^m\right) + \int \bN \mu d (\rho - \rho_{\tau,m}) \\
&\leq \left(\|\rho\|_m^m - \|\rho_{\tau,m}\|_m^m\right) + C_d \sqrt{\frac{2 \tau}{m-1}  \| \rho\|_m^m }
+ 2 C_d^2 \tau .
\end{align*}
\end{proof}

Iterating the above lemma provides bounds on the Wasserstein distance between the discrete time sequences of $E_\infty$, $\tilde E_\infty$, and $E_m$ and their initial data.
\begin{corollary}
\label{lem:time_conti}
Under the assumptions in Lemma \ref{W2 one step}, given initial data $\rho \in D(E_\infty)$,
\begin{align*}
W_2(\rho_\tau^n, \rho) \leq 2C_d n \tau , \ W_2(\tilde \rho_\tau^n, \rho) \leq 2C_d n \tau, \text{ and }W_2(\rho_{m,\tau}^n,\rho) \leq \sqrt{4n\tau \|\rho\|_m^m + 8C_d^2 n^2 \tau^2}.
\end{align*}
\end{corollary}

\begin{proof}
The first two inequalities are a direct consequence of Lemma \ref{W2 one step} and triangle inequality, so it remains to show the third inequality. By Lemma \ref{W2 one step} and $(a+b)^2\leq 2a^2+2b^2$,
\[
W_2^2( \rho_{\tau,m}^i, \rho_{\tau,m}^{i-1}) \leq \frac{4\tau}{m-1}(\|\rho_{\tau,m}^{i-1}\|_m^m - \|\rho_{\tau,m}^{i}\|_m^m) + 8C_d^2 \tau^2.
\]
The result then follows by the triangle inequality, Cauchy's inequality, and $1/(m-1) \leq 1$,
\begin{align*}
W_2^2(\rho^n_{\tau,m},\rho) &\leq \left(\sum_{i=1}^n W_2(\rho_{\tau,m}^i, \rho_{\tau,m}^{i-1})\right)^2 \leq n \sum_{i=1}^n W_2^2(\rho_{\tau,m}^i, \rho_{\tau,m}^{i-1})  \leq 4n\tau\|\rho\|_m^m+ 8C_d^2 n^2 \tau^2 .
\end{align*}

\end{proof}

In the next three lemmas, we estimate the size of $\rho_{\tau,m}$. These estimates are similar in some respects to the corresponding results in previous work by Alexander with the second and third authors \cite{AKY}. However,  the proofs must be adapted since the semiconvexity of the drift potential $\psi_{1/m}*\bN \mu$ in the energy $E_m(\cdot; \mu)$ deteriorates as $m \to +\infty$, and we must instead use that $E_m(\cdot; \mu)$ is $\omega$-convex uniformly in $m$.

Though we do not, in general, have  $\|\rho_{\tau,m}^1\|_\infty \leq 1$, in the next lemma, we show that the mass of $\rho_{\tau,m}$ above 1 becomes arbitrarily small as $m \to +\infty$.

\begin{lemma}
  \label{lemma:small_mass}
Fix $\rho, \mu \in \P_2(\Rd)$ with both $\|\rho\|_\infty, \|\mu \|_\infty \leq 1$ and consider $\rho_{\tau,m}$ as in Definition \ref{one step def}. Then for $C_d>0$ as in Proposition \ref{N mu bounds} and $0<\tau<1$, $m\geq 2$,
\begin{align*}
\int_{}(\rho_{\tau,m}(x)-1)_+ dx \leq \sqrt{(2 + C_d^2)/m} .
\end{align*}
\end{lemma}

\begin{proof}
By the Cauchy-Schwarz inequality and the fact that $|\{\rho_{\tau,m}\geq 1\}| \leq \int \rho_{\tau,m} = 1$, 
\begin{equation}
\label{eq:temp8}
\int (\rho_{\tau,m}-1)_+  \leq |\{\rho_{\tau,m}\geq 1\}|^{1/2} \left( \int (\rho_{\tau,m}-1)_+^2  \right)^{1/2}\leq   \left( \int (\rho_{\tau,m}-1)_+^2  \right)^{1/2}.
\end{equation}
Furthermore, for $m\geq 2$, the convexity of $f(s)=s^m$ ensures $s^m > 1 + m(s-1) + \frac{m(m-1)}{2}(s-1)^2$ for all $s>1$, which yields $(s-1)_+^2 \leq \frac{2}{m(m-1)}s^m$ for all $s>0$. Consequently, \eqref{eq:temp8} becomes
\begin{equation*}
\int_{\Rd}(\rho_{\tau,m}-1)_+  \leq \left( \frac{2}{m(m-1)} \int_\Rd \rho_{\tau,m}^m \right)^{1/2}.
\end{equation*} 
Since $\|\rho\|_\infty \leq 1$, $m \geq 2$, and $\tau <1$, Lemma \ref{W2 one step} \ref{rho m tau one step} ensures $\frac{1}{m-1} \|\rho_{\tau,m}\|_{m}^m \leq 1 + C_d^2/2 $. Substituting this into the above inequality  gives the result.
\end{proof}

Finally, we use the previous lemma to show that $\rho_{\tau,m}$ is always close to a measure $\nu$ that satisfies $\|\nu \|_\infty \leq 1$ and is almost a one step minimizer.

\begin{lemma}
\label{lemma:small_dist}Under the assumptions of Lemma \ref{lemma:small_mass}, there exists $\nu \in \PR$ with $\|\nu\|_\infty \leq 1$ and $C>0$ depending only on the dimension, so that
\begin{equation} \label{eq:energy_diff0}
W_2(\rho_{\tau,m}, \nu) \leq C m^{-1/4} \quad \text{ and } \quad E_m(\nu; \mu)\leq E_m(\rho_{\tau,m}; \mu) + C m^{-1/2}.
\end{equation}
\end{lemma}

\begin{proof}
Define $a:= \int (\rho_{\tau,m}-1)_+$. Since $\rho_{\tau,m}$ is a probably measure, $a < 1$, and by Lemma \ref{lemma:small_mass}, we also have $a\leq \sqrt{(2 + C_d^2)/m}$.

To construct $\nu$, we decompose $\rho_{\tau,m}$ as $\rho_{\tau,m} = \rho_{\tau,m}^1 + \rho_{\tau,m}^2$, where $\rho_{\tau,m}^1 = \min\{\rho_{\tau,m}, 1-a\}$ and $\rho_{\tau,m}^2 = (\rho_{\tau,m} - (1-a))_+$. First, note that
\[ 1 \geq \int \rho_m \geq (1-a) |\{\rho_m >1-a\}| + \int(\rho_m-1)_+ = (1-a) |\{\rho_m >1-a\}| + a , \]
so subtracting $a$ from both sides and dividing by $1-a$ ensures  $|\{\rho_{\tau,m}>1-a\}|\leq 1$. Thus,
\begin{align} \label{rho2 taum bound}
\int \rho_{\tau,m}^2  &\leq \int (\rho_{\tau,m}-1)_+  + a |\{\rho_{\tau,m}>1-a\}|\leq a+a\cdot 1 =2a.
\end{align}
Now, choose $R_d$ so $g := \frac{1}{2}\chi_{B_{R_d}(0)} \in \P_2(\Rd)$ and define $\nu = \rho_{\tau,m}^1 + g*\rho_{\tau,m}^2 \in \P_2(\Rd)$.
By Young's inequality, the definition of $\rho^1_{\tau,m}$, and inequality (\ref{rho2 taum bound}),
\[
\|\nu\|_\infty \leq \|\rho_{\tau,m}^1\|_\infty +  \|g\|_\infty \|\rho_{\tau,m}^2\|_1\leq   (1-a)+\frac{1}{2}\cdot 2a \leq 1.
\]

It remains to show that $\nu$ satisfies (\ref{eq:energy_diff0}).
To show the first inequality, we construct a transport plan between $\rho_{\tau,m}$ and $\nu$ as follows: keep all the mass of $\rho_{\tau,m}^1$ at its original location and distribute the mass of  $\rho_{\tau,m}^2(x)$ uniformly over the disk $B_{R_d}(x)$. Since $\int \rho_{\tau,m}^2 \leq 2a$, the total cost of this plan is bounded by $2a R_d^2$, which gives
$W_2(\rho_{\tau,m}, \nu) \leq \sqrt{2a}R_d \leq R_d (4(2 + C_d^2)/m)^{1/4}$.

To show the second inequality in (\ref{eq:energy_diff0}), we abbreviate $\Phi_{1/m} = \psi_{1/m}* \bN \mu$. Then,
\begin{align*}
&E_m(\nu;\mu) - E_m(\rho_{\tau,m};\mu) = \frac{1}{m-1}\int \nu^m  - \frac{1}{m-1}\int  \rho_{\tau,m}^m+ \int  \Phi_{1/m} d \nu - \int \Phi_{1/m} d \rho_{\tau,m} \nonumber \\
&\quad \leq \|\nu\|_m^m/(m-1) + \int  (g*\Phi_{1/m}-\Phi_{1/m}) \rho_{\tau,m}^2 \nonumber  \leq (m-1)^{-1} + 2a \|g*\Phi_{1/m} - \Phi_{1/m}\|_\infty \\
&\quad \leq 2m^{-1} + 2a \text{ \rm{ess sup}}_x \left| \int_{y\in B_{R_d}(x)} (\Phi_{1/m}(y)-\Phi_{1/m}(x)) g(x-y) dy  \right| \\
&\quad \leq 2m^{-1} + 2a \|\nabla \Phi_{1/m}\|_\infty R_d \|g\|_1 \leq  2m^{-1} + 2C_d R_d \sqrt{(2+C_d)^2/m} ,
\end{align*}
where in the last inequality we use Proposition \ref{N mu bounds}.
\end{proof}

\subsection{Distance between discrete time sequences of $E_\infty$, $\tilde E_\infty$, and $E_m$}
In this section, we apply the previous results to show that as $m \to +\infty$, $\rho_m(t)$ converges to $\rho_\infty(t)$, with quantitative rates of convergence on bounded time intervals. We accomplish this by first estimating the distance between the discrete time sequences of $E_\infty$ and $\tilde E_\infty$ and then $\tilde E_\infty$ and $E_m$.
We begin by showing that one step of the discrete gradient flow of $E_\infty$ is also one step of the discrete time sequence corresponding to $\tilde E_\infty$. 
\begin{lemma}[one-step comparison between $\rho_\tau$ and $\tilde \rho_\tau$]
\label{lemma:one_step_minimizer}
Given $\tau>0$ and $\rho \in \PR$, if $\rho_\tau$ is a one step minimizer of $E_\infty$, then it is also a one step minimizer of $\tilde E_\infty(\cdot, \rho_\tau)$.
\end{lemma}
\begin{proof}
Assume, for the sake of contradiction, that $\rho_\tau$ is not a one step minimizer of $\tilde E_\infty(\cdot, \rho_\tau)$. Then there exists $\nu \in \PR$ with $\|\nu\|_\infty \leq 1$, such that
\begin{align} \label{contradiction 1}
\frac{1}{2\tau} W_2^2(\rho, \rho_\tau) + \tilde E_\infty (\rho_\tau; \rho_\tau)> \frac{1}{2\tau} W_2^2(\rho, \nu) + \tilde E_\infty (\nu; \rho_\tau).
\end{align}
Define $ \rho^\epsilon := (1-\epsilon) \rho_\tau + \epsilon \nu \in \P_2(\Rd)$, so $\|\rho^\epsilon\|_\infty \leq 1$. We will show that for $\epsilon>0$ small,
\begin{align} \label{contradiction 2}
\frac{1}{2\tau} W_2^2(\rho, \rho_\tau) +  E_\infty (\rho_\tau)> \frac{1}{2\tau} W_2^2(\rho, \rho^\epsilon) + E_\infty (\rho^\epsilon),
\end{align}
which contradicts the fact that $\rho_\tau$ is a one step minimizer of $E_\infty$.

By inequality (\ref{W2 convex}), $W_2^2$ is convex along linear interpolations of measures, hence
\begin{equation}
\label{eq:temp1a}
\begin{split}
W_2^2(\rho,\rho^\epsilon) &\leq  (1-\epsilon) W_2^2(\rho, \rho_\tau) + \epsilon W_2^2(\rho, \nu) = W_2^2(\rho, \rho_\tau) - \epsilon \left(W_2^2(\rho, \rho_\tau) - W_2^2(\rho,\nu)\right).
\end{split}
\end{equation}
Likewise, we use that $2E_\infty(\rho_\tau) = \tilde E_\infty(\rho_\tau; \rho_\tau)$ to estimate the behavior of $E_\infty$ along $\rho^\epsilon$ by
\begin{align}
\label{eq:temp2a}
\begin{split}
E_\infty(\rho^\epsilon) &= \frac{1}{2}\int  \bN \left((1-\epsilon) \rho_\tau + \epsilon \nu\right) d \left( (1-\epsilon) \rho_\tau + \epsilon \nu\right)\\
&= (1-\epsilon)^2 E_\infty(\rho_\tau) + \epsilon (1-\epsilon) \tilde E_\infty(\nu; \rho_\tau) + \epsilon^2 E_\infty(\nu)\\
&= E_\infty(\rho_\tau) - \epsilon ( \tilde E_\infty(\rho_\tau; \rho_\tau)  - \tilde E_\infty(\nu; \rho_\tau)) + D \epsilon^2,
\end{split}
\end{align}
where $D := E_\infty(\rho_\tau) + E_\infty(\nu) - \tilde{E}_\infty(\nu,\rho_\tau)$ is a constant independent of $\epsilon$.
Multiplying \eqref{eq:temp1a} by $1/(2\tau)$ and adding to \eqref{eq:temp2a} yields
\begin{align*}
&\frac{1}{2\tau} W_2^2(\rho, \rho^\epsilon) +  E_\infty(\rho^\epsilon) \\
& \leq \frac{1}{2\tau} W_2^2(\rho, \rho_\tau) + E_\infty(\rho_\tau) - \epsilon\left(\frac{1}{2\tau}(W_2^2(\rho, \rho_\tau) - W_2^2(\rho,\nu)) + \tilde E_\infty(\rho_\tau; \rho_\tau)  - \tilde E_\infty(\nu; \rho_\tau) \right ) + D\epsilon^2 .
\end{align*}
By (\ref{contradiction 1}), the quantity within parentheses is strictly positive, hence we obtain (\ref{contradiction 2}) for $\epsilon$ small.
\end{proof}

Using this lemma and Proposition \ref{contraction inequality}, which proves a contraction inequality for one step of the discrete gradient flows of $\omega$-convex energies, we can bound the distance between the discrete gradient flow of $E_\infty$ and the discrete time sequence corresponding to $\tilde E_\infty$. 

\begin{proposition}[multi-step comparison between $\rho^n_\tau$ and $\tilde \rho^n_\tau$] \label{multistep rho rho tilde}
Given $T>0$ and initial data $\rho \in D(E_\infty)$, there exist positive constants $C$ and $N$ depending on the dimension, $T$, and $E_\infty(\rho)$ so that for $\tau = t/n$, $0\leq t \leq T$, and $n >N$,
\[ W_2(\rho^n_\tau, \tilde \rho^n_\tau) \leq C(n^{-1/2})^{1/2e^{2C_d T}} , \]
\end{proposition}
\begin{proof}
By Lemma \ref{lemma:one_step_minimizer}, $\rho^n_\tau$ is also a time varying discrete gradient flow of $\tilde E_\infty$, in the sense of Definition \ref{Discrete time def} \ref{tildeEinfty DGF}. Hence, by Proposition \ref{distance between DGFs Einfty tilde}, for any $T>0$ there exist positive constants $C$ and $N$ (which we allow to change from line to line), depending on the dimension, $T$, and $E_\infty(\rho)$ so that for $\tau = t/n$, $0 \leq t \leq T$, and $n >N$,
\begin{align*}
 f_\tau^{(2n)}(W_2^2(\rho^n_\tau, \tilde \rho^n_\tau)) &\leq C \omega(t/n) \leq C n^{-1/2} \end{align*} 
 Furthermore, combining Corollary \ref{lem:time_conti}  and the triangle inequality provides the following crude bound for the distance between the two sequences:
\[ W_2(\rho^n_\tau, \tilde \rho^n_\tau) \leq W_2(\rho^n_\tau, \rho) +W_2(\tilde \rho^n_\tau, \rho) \leq 4 C_d T \leq C. \]
Therefore, by Proposition \ref{f tau prop} \ref{ode euler estimate} and the fact that $F_t(x)$ is decreasing in time,
 \begin{align*}
F_{2n\tau}(W_2^2(\rho^n_\tau, \tilde \rho^n_\tau)) &\leq  C  n^{-1/2} +2C_d  \omega(C)T /n \implies F_{2T}(W_2^2(\rho^n_\tau, \tilde \rho^n_\tau))\leq C  n^{-1/2}  \end{align*} 
 Since for $0 \leq x \leq e^{-1-\sqrt{2}}$, $F_t(x) = x^{e^{C_d t}}$, for $n$ sufficiently large, we have
\[ W_2(\rho^n_\tau, \tilde \rho^n_\tau) \leq C(n^{-1/2})^{1/2e^{2C_d T}} , \]
which gives the result.
\end{proof}

Next, we bound the distance between one step of the discrete time sequences corresponding to  $\tilde E_\infty$ and $E_m$. 

\begin{proposition}[one-step comparison between $\tilde \rho_\tau$ and $\rho_{\tau,m}$]
\label{prop:one-step}
Given $\rho, \mu \in \P_2(\Rd)$ with $\|\rho\|_\infty, \|\mu \|_\infty \leq 1$ and $C_d >0$ as in Proposition \ref{N mu bounds}, there exists $C>0$ depending only on the dimension so that for all $0< \tau<1/6 C_d$ and $m \geq 2$,
\begin{equation*}
W_2(\rho_{\tau,m}, \tilde{\rho}_\tau) \leq C m^{-1/8}+ 2e^{-1/(4 C_d \tau)}. 
\end{equation*}
\end{proposition}

\begin{proof}
Let $\nu$ be as in Lemma \ref{lemma:small_dist} and define
\begin{equation*}
\eta := \left(\frac{1}{2} \bt_\rho^\nu + \frac{1}{2} \bt_\rho^{\tilde \rho_\tau} \right) \# \rho
\end{equation*}
to be the midpoint on the generalized geodesic from $\nu$ to $\tilde \rho_\tau$ with base $\rho$. Since the $L^\infty$ norm of a generalized geodesic is bounded by the $L^\infty$ norm of its endpoints (c.f. \cite[inequality (60)]{CraigOmega}), we have $\|\nu\|_\infty \leq 1$.
Furthermore, by the optimality of $\rho_{\tau,m}$ and $\tilde \rho_\tau$,
\begin{align*}
\frac{1}{2\tau} W_2^2(\rho, \eta) + E_m(\eta; \mu) &\geq \frac{1}{2\tau} W_2^2(\rho, \rho_{\tau,m}) + E_m(\rho_{\tau,m}; \mu) , \\
 \frac{1}{2\tau} W_2^2(\rho, \eta) + \tilde E_\infty(\eta; \mu) &\geq \frac{1}{2\tau} W_2^2(\rho,\tilde\rho_\tau) + \tilde E_\infty(\tilde\rho_\tau; \mu).
\end{align*}
Adding these  inequalities together and collecting  the distance and energy terms gives
\begin{align} \label{tw te lower}
T_W + T_E \geq 0,
\end{align}
for
\begin{align*}
T_W &:= \frac{1}{\tau} W_2^2(\rho, \eta) - \frac{1}{2\tau} W_2^2(\rho, \rho_{\tau,m}) - \frac{1}{2\tau} W_2^2(\rho,\tilde\rho_\tau), \\
T_E &:=  E_m(\eta; \mu) +  \tilde E_\infty(\eta; \mu) - E_m(\rho_{\tau,m}; \mu) - \tilde E_\infty(\tilde\rho_\tau; \mu).
\end{align*}

Next, we find upper bounds on $T_W$ and $T_E$.
Define $A:= \| \bt_\rho^\nu - \bt_\rho^{\tilde \rho_\tau} \|_{L^2(\rho)}$.
Since $W_2^2(\rho,\cdot)$ is 2-convex along generalized geodesics with base $\rho$ \cite[Lemma 9.2.1]{AGS},
\begin{equation*}
W_2^2(\rho, \eta) \leq W_2^2(\rho, \nu)/2 + W_2^2(\rho, \tilde\rho_\tau )/2 - A^2/4 .
\end{equation*}
Substituting this in the definition of $T_W$, 
\begin{equation}
\label{eq:tw}
\begin{split}
T_W &\leq \frac{1}{2\tau} (W_2^2(\rho, \nu) - W_2^2(\rho,\rho_{\tau,m})) - \frac{A^2}{4\tau} \leq \frac{1}{2\tau} W_2( \rho_{\tau,m}, \nu) (W_2(\rho, \nu) + W_2(\rho, \rho_{\tau,m})) - \frac{A^2}{4\tau } \nonumber \\
&\leq \frac{1}{2 \tau} W_2(\rho_{\tau,m},\nu) (W_2(\rho_{\tau,m},\nu) + 2W_2(\rho,\rho_{\tau,m}) - \frac{A^2}{4\tau } \leq \frac{C}{2 \tau} m^{-1/4} - \frac{A^2}{4\tau },
\end{split}
\end{equation}
where in the last inequality we apply $W_2(\rho_{\tau,m}, \nu) \leq C m^{-1/4}$ from Lemma \ref{lemma:small_dist}, $W_2(\rho_{\tau,m}, \rho) \leq \sqrt{2\tau} + 2C_d\tau $ from Lemma \ref{W2 one step}, and the facts that $m \geq 2$ and $\tau <1$. We also allow $C>0$, depending only on the dimension, to change from line to line.

In order to bound $T_E$ from above, we first estimate the difference between  $E_m(\tilde \mu,\mu)$ and $\tilde E_\infty(\tilde \mu, \mu)$ for any $\tilde \mu \in \P_2(\Rd)$ with $\|\tilde \mu\|_\infty \leq 1$. As usual, we abbreviate $\Phi := \bN \mu$ and $\Phi_{1/m} : = \psi_{1/m} * \bN \mu$. Given $R_\psi>0$ so that $\supp \psi \subseteq B_{R_\psi}(0)$, for any $x\in \Rd$, Proposition \ref{N mu bounds} ensures
\begin{equation*}
|\Phi_{1/m}(x) - \Phi(x)| = \left| \int_{\Rd} (\Phi(x-y)-\Phi(x)) \psi_{1/m}(y) dy \right| \leq R_\psi \|\nabla \Phi\|_{\infty} m^{-1} \leq C_d m^{-1}.
\end{equation*}
Consequently,
\begin{equation*}
\left|E_{m}(\tilde \mu; \mu) - \tilde E_\infty(\tilde \mu; \mu)\right| \leq \frac{1}{m-1}\int \tilde \mu^m + \int \left|\Phi_{1/m} - \Phi\right| d \tilde \mu \leq (m-1)^{-1} + \|\Phi_{1/m} - \Phi\|_\infty \leq (2+C_d)m^{-1}.
\end{equation*}
Therefore, first applying Lemma \ref{lemma:small_dist} to the definition of $T_E$ and then using the above inequality,
\begin{align}
\label{eq:temp10}
T_E &\leq E_m(\eta; \mu) + \tilde E_\infty(\eta; \mu) - E_m(\nu, \mu) - \tilde E_\infty(\tilde \rho_\tau; \mu) + C m^{-1/2}\\
&\leq 2\tilde E_\infty(\eta; \mu) - \tilde E_\infty(\nu, \mu) - \tilde E_\infty(\tilde \rho_\tau; \mu) + C m^{-1/2}. \nonumber
\end{align}
Since $\tilde E_\infty(\cdot, \mu)$ is $\omega$-convex along generalized geodesics and $\eta$ is the midpoint along the generalized geodesic from $\nu$ and $\tilde \rho_\tau$ with base $\rho$,
\[
2\tilde E_\infty(\eta;\mu) - \tilde E_\infty(\nu;\mu)  - \tilde E_\infty(\tilde \rho_\tau;\mu)\leq \frac{C_d}{2} \omega\left(\frac{A^2}{4}\right).
\]
Substituting this into inequality \eqref{eq:temp10} gives
\begin{equation*}
T_E \leq C m^{-1/2} + C_d \omega\left(A^2/4\right).
\end{equation*}
Finally, combining our upper bounds on $T_W$ and $T_E$ with inequality (\ref{tw te lower}), we obtain
\begin{equation}
\label{ineq:A0}
A^2 \leq C m^{-1/4} + 4  C_d \tau \omega \left(A^2/4\right) .
\end{equation}

We now claim that
\begin{align} \label{eq:goal_A}
A \leq \sqrt{2 C} m^{-1/8} + 2 e^{-1/(4C_d \tau)} 
\end{align}If $A^2/4 > e^{-1-\sqrt{2}}$, then combining inequality (\ref{ineq:A0}) and $\tau< 1/(6 C_d)$ implies
\[A^2 \leq C m^{-1/4} + 4 \tau C_d \omega \left(A^2/4\right) \leq C m^{-1/4} + 3\tau C_d A^2   \implies A \leq \sqrt{2C}m^{-1/8}  ,\]
hence (\ref{eq:goal_A}) holds.
Alternatively, if $A^2/4 \leq e^{-1-\sqrt{2}}$,
\begin{equation}
A^2 \leq C m^{-1/4} + 4 \tau C_d \omega \left(A^2/4\right) = C m^{-1/4} - C_d \tau A^2 \log \left(A^2/4\right) .
\end{equation}
 If \eqref{eq:goal_A} is violated, we have $A> \sqrt{2C} m^{-1/8}$ and $A> 2e^{-1/(4 C_d \tau)}$, so  $C m^{-1/4}< A^2/2$ and $-C_d \tau A^2 \log(A^2/4) < \frac{A^2}{2}$. Adding these together would contradict \eqref{ineq:A0}, so again (\ref{eq:goal_A}) holds.
 
 Since  $A= \| \bt_\rho^\nu - \bt_\rho^{\tilde \rho_\tau} \|_{L^2(\rho)} =  \| \bt_\rho^\nu \circ \bt_{\tilde \rho_\tau}^\rho - \id \|_{L^2(\tilde \rho_\tau)} $ and $\bt_\rho^\nu \circ \bt_{\tilde \rho_\tau}^\rho \# \tilde \rho_\tau = \nu$, we have $W_2(\nu, \tilde \rho_\tau) \leq A$. Therefore, using \eqref{eq:goal_A} and Lemma \ref{lemma:small_dist}, we may conclude the result,
 \begin{align*}
 W_2(\rho_{\tau,m}, \tilde \rho_\tau) &\leq  W_2(\rho_{\tau,m}, \nu)  +  W_2(\nu, \tilde \rho_\tau)  \leq C m^{-1/4} + A \\
 &\leq C m^{-1/4}+  \sqrt{2C} m^{-1/8} + 2e^{-1/(4 C_d \tau)} 
  \leq C m^{-1/8}+2e^{-1/(4 C_d \tau)}.
 \end{align*}
 \end{proof}

\begin{proposition}[multi-step comparison between $\tilde \rho^n_\tau$ and $\rho^n_{\tau,m}$]
\label{prop:multi_step}  
Given $T>0$ and initial data $\rho \in D(E_\infty)$, there exist positive constants $C$, $N$, and $M$ depending on the dimension, $T$, $E_\infty(\rho)$, and $\psi$ so that for $\tau = t/n$, $0 \leq t \leq T$,  $n > N$, $m >M$, and $n = o(m^{1/8})$,
\[W_2(\rho_{\tau,m}^n, \tilde \rho_\tau^n) \leq C(n^{-1/4} + n m^{-1/8})^{1/2e^{2C_d T}} \]
\end{proposition}

\begin{proof}
Define $d_i := W_2(\rho_{\tau,m}^i, \tilde \rho_\tau^i)$ for any $i = 1, \dots, n$. Using Corollary \ref{lem:time_conti} and $\|\rho\|_m^m \leq 1$, we have the crude bound 
\begin{align} \label{dn crude bound}
d_i \leq W_2(\rho_{\tau,m}^i,\rho) + W_2(\tilde \rho_\tau^i,\rho)  \leq \sqrt{4T + 8C_d T^2} + 2C_d T. 
\end{align}

The one step estimates from Proposition \ref{prop:one-step} allow us to control the distance between one-step minimizers of $E_m$ and $\tilde E_\infty$ when they have the \emph{same} initial data. In particular, for
\begin{align} \label{delta def}
 \delta: = C m^{-1/8}+ 2e^{-1/(4 C_d \tau)} ,
 \end{align}
we have $d_1 \leq \delta$. In order to apply Proposition \ref{prop:one-step} to control $d_i$ for $i = 2, \dots, n$, we  use a sequence of densities $\eta^i$ to serve as a bridge between $\rho_{\tau,m}^i$ and $\tilde \rho_\tau^i$, following the tree structure in Figure \ref{fig:tree}. Specifically, we choose $\eta^i \in \PR$ so that, by Proposition \ref{prop:one-step},
\begin{equation}
\label{eq:temp00081}
\eta^i \in \argmin_{\nu \in \PR} \left\{ \frac{1}{2\tau} W_2^2(\tilde \rho_\tau^{i-1}, \nu) + E_m(\nu;\rho_\tau^i)\right\} \implies W_2(\tilde \rho_\tau^i, \eta^i) \leq \delta.
\end{equation}

\begin{figure}
\begin{center}
\includegraphics[scale=1.2]{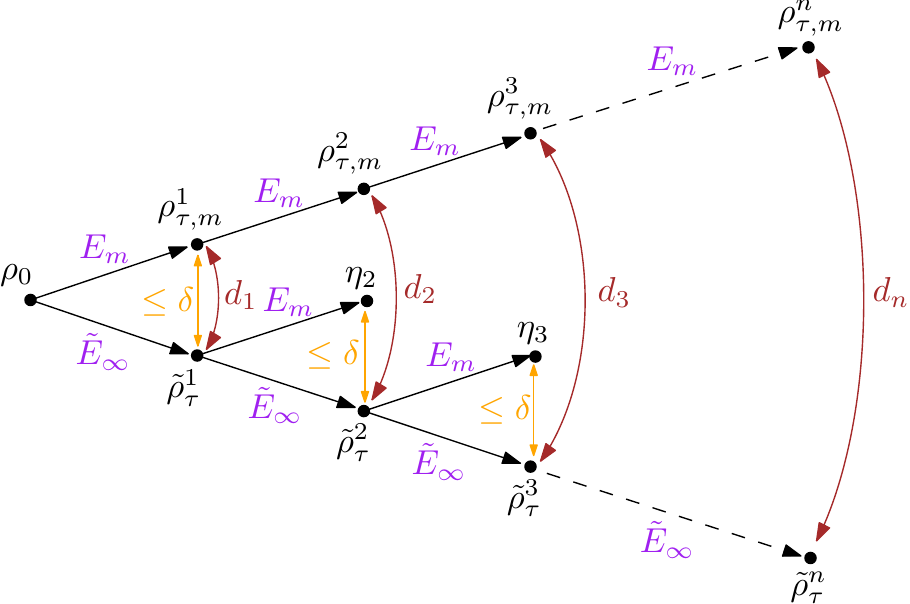}
\end{center}
\caption{An illustration of the tree structure used in the multi-step comparison between $\tilde \rho^n_\tau$ and $\rho^n_{\tau,m}$ \label{fig:tree}}
\end{figure}

Since $\eta^i$ and $\rho_{\tau,m}^i$ are one-step minimizers of the same energy  $E_m(\cdot, \rho_\tau^i)$ with different initial data ($\tilde \rho_\tau^{i-1}$ and $\rho_{\tau,m}^{i-1}$ respectively), we may control their distance using Proposition \ref{contraction inequality}.

First, we obtain a few elementary bounds on how the energy changes along the discrete time sequence.
Combining Lemma \ref{drift continuity estimate0}, Proposition \ref{N mu bounds}, Lemma \ref{W2 one step}, and the definition of $\rho^{i-1}_{\tau,m}$ as a minimizer,
\begin{align*}
  E_m( \rho^{i-1}_{\tau,m};  \rho^{i}_\tau) &=  E_m( \rho^{i-1}_{\tau,m};  \rho^{i-1}_\tau) + E_m( \rho^{i-1}_{\tau,m};  \rho^{i}_\tau) - E_m( \rho^{i-1}_{\tau,m};  \rho^{i-1}_\tau) \\
 &=E_m( \rho^{i-1}_{\tau,m};  \rho^{i-1}_\tau)  +  \int \psi_{1/m}* \bN  \rho^{i-1}_{\tau,m} d (\rho^{i}_\tau - \rho^{i-1}_\tau)  \leq  E_m( \rho^{i-2}_{\tau,m};  \rho^{i-1}_\tau) + C_d W_2(\rho^i_\tau, \rho^{i-1}_\tau) \\
 & \leq E_m( \rho^{i-2}_{\tau,m};  \rho^{i-1}_\tau) +2 C_d^2 \tau \leq \dots \leq E_m(\rho; \rho_\tau^1) + 2 C_d^2 T  .
\end{align*}
Likewise, we may control the first term on the right hand side of the last inequality by
\begin{align*}
  E_m (\rho; \rho_\tau^1)  &= \|\rho\|_m^m/(m-1) + \int \psi_{1/m}* \bN \rho_\tau^1 d \rho  \leq 1+ 2 E_\infty(\rho) + \int \bN \rho d (\psi_{1/m}*\rho^{1}_\tau - \rho) \\
  &\leq 1+2 E_\infty(\rho) + C_d W_2(\psi_{1/m}*\rho^1_\tau, \rho)  \leq 1+ 2 E_\infty(\rho) + C_d (W_2(\psi_{1/m}* \rho^1_\tau, \rho^1_\tau) + W_2(\rho^1_\tau, \rho)) \\
  & \leq 1+2 E_\infty(\rho)  + C_d ((1/m)M_\psi + 2 C_d \tau)
  \end{align*}
  where, in the last step, we apply \cite[Lemma 7.1.10]{AGS}, which ensures
  \begin{align} \label{measure conv}
W_2(\mu* \psi_{1/m}, \mu) \leq \frac{1}{m} \left( \int |x|^2 \psi(x) dx \right)^{1/2} =: \frac{1}{m} M_\psi .
\end{align}
Combining the above two inequalities, we conclude that there exists constant $C>0$ (which we allow to change from line to line) depending on the dimension, $T$, $E_\infty(\rho)$, and $\psi$ so that
  \[ E_m( \rho^{i-1}_{\tau,m};  \rho^{i}_\tau) \leq C \text{ for all } i = 1, \dots, n. \]
  Furthermore, by Proposition \ref{N mu bounds}, we also have that $E_m(\cdot;\cdot)$ is uniformly bounded below.
  
  Using these estimates on the energy, we may now apply Proposition \ref{contraction inequality} to conclude that there exist positive constants $C$ and $N$ depending on the dimension, $T$, $E_\infty(\rho)$, and $\psi$, which we allow to change from line to line, so that for $\tau = t/n$, $0 \leq t \leq T$, and $n > N$,
\begin{equation*}
 f_\tau^{(2)} (W_2^2(\eta^i, \rho_{\tau,m}^i)) \leq d_{i-1}^2 +C_d \tau \omega( C W_2(\eta^i,\tilde \rho_\tau^{i-1} )) + 2 \tau (E_m(\rho_{\tau,m}^{i-1}; \rho_\tau^i) - E_m(\rho_{\tau,m}^i; \rho_\tau^i) ) +  C \tau^2.
\end{equation*}
By Lemma \ref{W2 one step} \ref{rho m tau one step}, we have the following bounds for two quantities on the right hand side:
\begin{align*}
 W_2( \eta^i, \tilde \rho_\tau^{i-1}) &\leq \sqrt{\frac{2 \tau}{m-1}\|\tilde \rho_\tau^{i-1}\|_m^m} + 2C_d \tau \leq C\sqrt{\tau} , \\
 E_m(\rho_{\tau,m}^{i-1}; \rho_\tau^i) - E_m(\rho_{\tau,m}^i; \rho_\tau^i) &\leq \left( \|\rho_{\tau,m}^{i-1}\|_m^m - \|\rho_{\tau,m}^{i}\|_m^m \right) + C \sqrt{\tau} .
 \end{align*}
Therefore,
\begin{align} \label{simplified contraction 1}
 f_\tau^{(2)} (W_2^2(\eta^i, \rho_{\tau,m}^i)) \leq d_{i-1}^2 + 2 \tau \left( \|\rho_{\tau,m}^{i-1}\|_m^m - \|\rho_{\tau,m}^{i}\|_m^m \right) + C \tau^{5/4}
\end{align}

We now use this estimate to bound $d_i=W_2(\rho_{\tau,m}^i, \tilde \rho_\tau^i)$. By the triangle inequality and  \eqref{eq:temp00081},
\begin{align} \label{multi m pre contraction}
d_i^2 &\leq \left(W_2(\eta^i, \rho_{\tau,m}^i) + W_2(\tilde \rho^i_\tau, \eta^i) \right)^2 \leq W_2^2(\eta^i, \rho_{\tau,m}^i) + \left(2W_2(\eta^i, \rho_{\tau,m}^i)   + \delta\right)\delta
\end{align}
Furthermore, by Lemma \ref{W2 one step} \ref{rho m tau one step}, inequality (\ref{dn crude bound}), and equation (\ref{delta def})
\begin{align*}
W_2(\eta^i, \rho_{\tau,m}^i) \leq W_2(\eta^i, \tilde \rho_\tau^{i-1} ) +d_{i-1} +\delta \leq C,
\end{align*}
Thus, by Proposition \ref{f tau prop}, we may apply $f_\tau^{(2)}$ to both sides of the (\ref{multi m pre contraction}) to obtain
\[ f_\tau^{(2)}(d_i^2) \leq f_\tau^{(2)}(W_2^2(\eta^i,\rho^i_{\tau,m})) +C \delta+C \tau^2.  \]
Combining this with \eqref{simplified contraction 1} gives, for all $i=1, \dots, n$,
\begin{align} \label{multi step base case}
f_\tau^{(2)}(d_i^2 )&\leq d_{i-1}^2 + 2 \tau \left( \|\rho_{\tau,m}^{i-1}\|_m^m - \|\rho_{\tau,m}^{i}\|_m^m \right) + C \tau^{5/4}+  C \delta.
\end{align}

We claim that the result will follow if we can show that, for all $j = 1, \dots, n$,
\begin{align} \label{multi step induction}
 f_\tau^{(2j)}(d_{n}^2) \leq  d_{n-j}^2 + 2 \tau (\|\rho^{n-j}_{\tau,m} \|_m^m - \|\rho^n_{\tau,m}\|_m^m ) + 2C \tau^{5/4} j + C \delta j .
\end{align}
In particular, if this holds, then taking $j=n$ and using that $e^{-1/4C_d \tau} = O(\tau^{5/4})$ gives
\[ f_\tau^{(2n)}(d_{n}^2) \leq   2 \tau \|\rho \|_m^m + 2CT \tau^{1/4}  + C \delta n \leq C (n^{-1/4} + n m^{-1/8} ).\]
By Proposition \ref{f tau prop} \ref{ode euler estimate} and the fact that $F_{t}(x)$ is decreasing in $t$,
\[ F_{2n\tau}(d_n^2) \leq C(n^{-1/4} + n m^{-1/8}) + 2C_d \omega(C) T/n \implies F_{2T}(d_n^2) \leq C(n^{-1/4} + n m^{-1/8}) . \]
For $0 \leq x \leq e^{-1-\sqrt{2}}$, we have $F_t(x) = x^{e^{C_d t}}$. Thus, for $n$ and $m$ sufficiently large, depending on the dimension, $T$, $E_\infty(\rho)$, and $\psi$, and with $n = o(m^{1/8})$, we have
\[ d_n \leq C(n^{-1/4} + n m^{-1/8})^{1/2e^{2C_d T}} , \]
which gives the result.

It remains to show (\ref{multi step induction}). We proceed by induction. The base case for $j=1$ follows from (\ref{multi step base case}), so we assume the result holds for $j-1$,
\begin{align*}
 f_\tau^{(2(j-1))}(d_{n}^2) \leq  d_{n-j+1}^2 + 2 \tau (\|\rho^{n-j+1}_{\tau,m} \|_m^m - \|\rho^n_{\tau,m}\|_m^m ) + 2C \tau^{5/4} (j-1) + C \delta (j-1) .
\end{align*}
For any $j =1, \dots, n$, the right hand side is bounded by a constant depending on the dimension, $T$, $E_\infty(\rho)$, and $\psi$. Thus, by Proposition \ref{f tau prop}, we may apply $f_\tau^{(2)}$ to both sides to conclude\begin{align*}
 f_\tau^{(2j))}(d_{n}^2) &\leq f_\tau^{(2)}( d_{n-j+1}^2) + 2 \tau (\|\rho^{n-j+1}_{\tau,m} \|_m^m - \|\rho^n_{\tau,m}\|_m^m ) + 2C \tau^{5/4} (j-1) +C \delta (j-1) +C^2 \tau^2 \\
 &\leq d_{n-j}^2 + 2 \tau(\|\rho^{n-j}_{\tau,m} \|_m^m - \|\rho^n_{\tau,m}\|_m^m ) +2 C \tau^{5/4}j + C\delta j 
\end{align*}
where, in the second inequality, we apply (\ref{multi step base case}) and the fact that $C^2 \tau^2 \leq C \tau \sqrt{\tau}$.
\end{proof}

Combining the previous propositions, we obtain our main result.

\begin{theorem}[convergence of $\rho_m(t)$ to $\rho_\infty(t)$]
\label{thm:conv_w2}
Given $T>0$ and initial data $\rho \in D(E_\infty)$, there exist positive constants $C$ and $M$ depending on $d$, $T$, $E_\infty(\rho)$, and $\psi$ so that for all $0 \leq t \leq T$ and $m \geq M$,
\[ W_2(\rho_m(t), \rho_\infty(t)) \leq C m^{-1/144e^{4C_d T}} . \]
\end{theorem}
\begin{proof}
Combining Proposition \ref{multistep rho rho tilde}, Proposition \ref{prop:multi_step}, \cite[Theorem 3.8]{CraigOmega}, and Proposition \ref{time dependent assumption prop}, there exist positive constants $C$ and $N$ depending on $d, T, E_\infty(\rho),$ and $\psi$ so that for $\tau = t/n$ and all $n \geq N$, $m \geq d+1$, $0 \leq t \leq T$, and $n = o(m^{1/8})$,
\begin{align*}
W_2(\rho^n_\tau, \tilde \rho^n_\tau) &\leq Cn^{-1/4e^{2C_d T}} , \quad &
W_2(\rho_{\tau,m}^n, \tilde \rho_\tau^n) &\leq C(n^{-1/4} + n m^{-1/8})^{1/2e^{2C_d T}} \\
W_2(\rho^n_{\tau},\rho_\infty(t)) &\leq  C n^{-1/16 e^{2C_d T}} , \quad &
W_2(\rho^n_{t/n,m}, \rho_m(t)) &\leq C  n^{-1/16e^{4C_dT}}.
\end{align*}
Hence,
\[ W_2(\rho_m(t), \rho_\infty(t)) \leq C ( n^{-1/16e^{4C_dT}}+n^{1/2e^{2C_dT}}m^{-1/16e^{2C_dT}})   \]
Taking $n = m^{1/9}$ gives the result.
\end{proof}

\section{Convergence of viscosity solutions: drift diffusion pressure to free boundary problem} \label{convviscsolsec}

 In the previous section, we showed that the gradient flow of the height constrained interaction energy $E_\infty$, which is merely a curve in the space of measures, may be approximated by solutions of the nonlinear diffusion equations \ref{pme} as $m \to +\infty$. This approximation provides the bridge by which we are able to unite the energy methods approach with viscosity solution approach. In the present section, we use this approximation to characterize the dynamics of patch solutions in terms of a Hele-Shaw type free boundary problem. We accomplish this by considering the nonlinear diffusion equations in terms of their pressure variables: given $\rho_m$ a weak solution of \ref{pme}, the pressure variable $p_m:= \frac{m}{m-1}(\rho_m)^{m-1}$ uniquely solves
\begin{align}  \tag*{\ref{P}$_m$} \label{Pm}
  (p_m)_t - (m-1)p_m(\Delta p_m + \Delta\Phi_{1/m}) -\nabla p_m \cdot(\nabla p_m + \nabla\Phi_{1/m})=0 .
\end{align}
For initial data given by \eqref{initial}, we show that as $m \to +\infty$ the half-relaxed limits of viscosity solutions of \ref{Pm} satisfy sub- and supersolution properties of \ref{Pinfty}.  The comparison principle of \ref{Pinfty} then yields that these half-relaxed limits are ordered with respect to the viscosity solution $p$ of \ref{Pinfty} with the same initial data.  In terms of the density variable, we show that  $\rho_m$ uniformly converges to $\chi_{\Omega(t)}$ away from $\partial\Omega(t)$, where $\Omega(t) = \{p(\cdot,t)>0\}$.
It follows that $\rho_{\infty} = \chi_{\Omega(t)}$ almost everywhere, and thus \ref{Pinfty} identifies with \ref{P}.  Due to the fact that the link between $(P)_{\infty}$ and $(P)$ lacks a priori stability estimates as the initial data varies, we must introduce additional perturbations and approximations into our proof of this final result.  

\begin{remark}

The lack of the comparison principle for the original problem \ref{P} is not the main reason we consider \ref{pme}. We could have considered  the drift term given by $\Phi:= \mathcal{N} * \rho_m$, and thus proved the convergence of the Keller-Segel equation  to our problem, if we had known that the corresponding solutions $\rho_m$ converged to $\rho_{\infty}$ as $m\to +\infty$. Obtaining such convergence seems to require a uniform $L^\infty$ bound on the gradient flow solutions of \ref{pme}, which is an open question at the moment.

\end{remark}

\subsection{Basic properties of viscosity solutions of \ref{Pm} and \ref{Pinfty}}
We refer the reader to Alexander, Kim, and Yao \cite[Section 3]{AKY} and Kim and Lei \cite[Section 2.1]{KimLei} for the definitions of classical and viscosity solutions of \ref{Pm}, and we refer the reader to appendix section \ref{Pinftyviscdef} for the definition of viscosity solutions of \ref{Pinfty}.  To clarify our notion of weak solutions for the original free boundary problem $(P)$, we make the following definition:

\begin{definition}\label{weak_def}
$p$ is a weak solution of $(P)$ if it is a viscosity solution of $(P)_{\infty}$ with initial data $p_0$ and $\rho_{\infty} = \chi_{\{p>0\}}$ almost everywhere.
\end{definition}

We now recall the several results on well-posedness of viscosity solutions of \ref{Pm} and the $L^1$ contraction of the corresponding density variable. 
\begin{lemma}\label{PMEDwellposed}
Consider the porous medium equation with drift and source terms,
\begin{equation}\label{pme_source}
\rho_t = \nabla \cdot(\rho\nabla \Phi_{1/m}) +\Delta \rho^m + \rho f,
\end{equation}
with $f \in L^1$ and bounded initial data.
\begin{enumerate}[label = (\alph*)]
\item If $\rho_1$ and $\rho_2$ are weak solutions of \eqref{pme_source} with source terms $f_1$ and $f_2$, then for all $t \geq 0$, \label{PMEDL1contraction}
$$
\| \rho_1(\cdot, t) - \rho_2(\cdot,t) \|_{L^1(\Rd)} \leq \|\rho_1(\cdot,0)- \rho_2(\cdot,0)\|_{L^1(\Rd)}  + \int_0^t\int_{\Rd} |\rho_1f_1 - \rho_2f_2| .
$$
 \item
Let $\rho$ be a weak solution of \eqref{pme_source} for any continuous, compactly supported initial data $\rho_0$ and continuous function $f$. Then the pressure variable $p_m:= \frac{m-1}{m}\rho^{m-1}$ is a viscosity solution to
$$
(p_m)_t - (m-1)p_m(\Delta p_m + \Delta\Phi_{1/m}+f) -\nabla p_m \cdot(\nabla p_m + \nabla\Phi_{1/m})=0.
$$ 
\end{enumerate}
\end{lemma}

\begin{proof}
(a) is due to \cite[Section 3.2.2]{Vazquez}, and (b) follows from \cite[Corollary 2.11]{KimLei}.
\end{proof}

We now turn to the following estimates on the size and support of solutions to \ref{Pm}, which are uniform in $m$. The first ensures that if the initial data is bounded uniformly in $m$, it remains so on bounded time intervals. The second ensures that if the support of the initial data is bounded uniformly in $m$, it likewise remains so on bounded time intervals.

\begin{lemma}[Estimates on size and support of solutions to \ref{Pm}]\label{compact_spt} 
Let $p_m$ be a viscosity solution of \ref{Pm} with continuous, compactly supported initial data $p_m(\cdot, 0)$. Suppose that there exists $R_0 \geq 1$ sufficiently large so that $\{ p_m(\cdot, 0) >0 \} \subseteq B_{R_0/2}(0)$ and $p_m(\cdot, 0) \leq R_0^2/4d$. Define $R(t) :=  (R_0+\frac{C_d}{d})e^{t/d} - \frac{C_d}{d}$, with $C_d>0$ as in \eqref{eq:phi_bound}. Then,
\begin{enumerate}[label = (\alph*)]
\item $\{p_{m}(\cdot,t)>0\} \subseteq B_{R(t)}(0)$ for all $t \in [0,T]$; \label{cpt spt part}
\item $p_m(x,t) \leq R(t)^2/2d$ for all $t>0$. \label{pm unif bdd}
\end{enumerate}
\end{lemma}

\begin{proof}
We prove the result by comparison with a classical supersolution of  \ref{Pm}.
Define 
\[ h(x) = \begin{cases} \frac{1-|x|^2}{2d} &\text{ for } |x| < 1, \\ 0 &\text{ for } |x| \geq1, \end{cases} \]
so that $h(x)$ satisfies $-\Delta h = 1$ in $|x|<1$ and $h=0$ in $|x|\geq1$. 
Let $\phi(x,t):= R(t)^2h(x/R(t))$, where $R(t)$ solves $R'(t)= \frac{R(t)}{d}+C_d$ with $r(0)=R_0$, and $C_d$ is the upper bound of $\|\nabla \Phi_{1/m}\|_\infty$ given by \eqref{eq:phi_bound}. We claim that such $\phi$ is a classical supersolution of \ref{Pm} for all $m$. To check this, direct computation gives that in the support of $\phi$,
\begin{equation}
\phi_t(x,t) = 2R(t) R'(t) \frac{1-\frac{|x|^2}{R(t)^2}}{2d} + R(t)^2 \left(-\frac{x}{dR(t)}\right) \cdot \left(-\frac{xR'(t)}{R(t)^2} \right) = \frac{R(t)R'(t)}{d}.
\label{eq:temp111}
\end{equation}
and $\|\nabla \phi(\cdot, t)\|_\infty = R(t) \|\nabla h\|_\infty = \frac{R(t)}{d}$. In addition, since $\Delta \phi = -1$ in its support and  $\Delta \Phi_{1/m} \leq 1$ for all $m$, we have
\begin{equation}
(m-1) \phi\underbrace{(\Delta \phi + \Delta \Phi_{1/m})}_{\leq 0} - \nabla \phi \cdot(\nabla \phi + \nabla \Phi_{1/m}) \leq \frac{R(t)}{d}\left(\frac{R(t)}{d}+C_d\right).
\label{eq:temp222}
\end{equation}
Comparing \eqref{eq:temp111} with \eqref{eq:temp222} gives that $\phi$ is a classical supersolution if $R'(t)= \frac{R(t)}{d}+C_d$. With $R(0)=R_0$, we have $p_m(\cdot, 0) \leq \phi(\cdot,0)$ for all $m$, so comparison principle yields that $\{p_m(\cdot, t)>0\} \subseteq B_{R(t)}(0)$ for all $t$, and $p_m(x,t) \leq R(t)^2 /2d$ for all $x,t$. 
%
\end{proof}

\begin{remark} \label{remark:bound_rho_m}
Lemma \ref{compact_spt} \ref{pm unif bdd} and the fact that $\rho_m = (\frac{m-1}{m} p_m)^{1/(m-1)}$ directly lead to the bound 
\begin{equation}
\label{eq:bound_rho_m}
\limsup_{m\to\infty} \|\rho_m(\cdot, t)\|_\infty \leq 1 \quad\text{ for all } t\geq 0,
\end{equation}
which we will make use of in what follows.
\end{remark}

A key property of viscosity solutions of \ref{Pm} is that they satisfy a comparison principle, which we now recall. We say two functions $f,g: \mathbb{R}^d \to [0,\infty)$ are \emph{strictly separated}, denoted by $f\prec g$, if $f< g$ in $\{f>0\}$, and $\overline{\{f>0\}}$ is a compact subset of $\{g>0\}$.
\begin{theorem}[comparison theorem for \ref{Pm}] 
Suppose $u$ and $v$ are viscosity sub- and supersolutions of \ref{Pm}. If the initial data are strictly ordered, i.e.
\[ u(\cdot,0) < v(\cdot,0) \hbox{ in } \overline{\{u(\cdot,0)>0\}} \hbox{ and } \overline{\{u(\cdot,0)>0\}}\hbox{ is a compact subset of } \{v(\cdot,0)>0\}, \]
  then $u(\cdot,t) \prec v(\cdot,t)$ for all $t>0$. 
\end{theorem}

\begin{proof}
The result follows from  \cite[Theorem 2.25]{KimLei}.
\end{proof}

We also have the following comparison theorem for solutions to \ref{Pinfty}, which we prove at the end of this section.

\begin{theorem}[comparison theorem for \ref{Pinfty}] \label{comparison}
Suppose $(u, \Sigma)$ and $v$ are respectively viscosity sub- and supersolutions of \ref{Pinfty}. If the initial data are strictly ordered, i.e.
\begin{align} \label{strictlyordered} u(\cdot,0) < v(\cdot,0) \hbox{ in } \overline{\Sigma\cap\{t=0\}} \hbox{ and }\overline{\Sigma\cap\{t=0\}} \hbox{ is a compact subset of } \{v(\cdot,0)>0\}, 
\end{align}
then $u(\cdot,t) \prec v(\cdot,t)$ and $\bar{\Sigma} \cap \{t\} \subset\{v(\cdot, t)>0\}$ for all $t>0$. 
\end{theorem}

While the above theorem almost provides uniqueness of \ref{Pinfty}, the requirement that the initial data be strictly ordered prevents us from concluding this result. However, combining the comparison principle with Perron's method yields the following:

\begin{theorem} \label{existenceofsolntoP}
For any bounded open set $\Omega_0 \subseteq \R^d$ with Lipschitz boundary, there exists  minimal and maximal viscosity solutions of \ref{Pinfty}.
\end{theorem}

\begin{proof}
The result follows from \cite{Kim}.
\end{proof}

We will use this comparison theorem, as well as the $L^1$ contraction theorem for $\rho_m$, to obtain our first main result: we identify $\rho_{\infty}$ with the characteristic function on the support of the minimal viscosity solution of \ref{Pinfty}, when the initial data $p_0$ is given by \eqref{initial}.

\subsubsection{Comparison theorem for \ref{Pinfty}}
To conclude this section on basic properties of viscosity solutions of \ref{Pm} and \ref{Pinfty}, we sketch the proof of the comparison principle for \ref{Pinfty}, Theorem \ref{comparison}.
Our approach is to consider the first contact time for regularizations of the sub- and supersolutions, obtained by considering their sup and inf convolutions over space-time smooth sets. Such regularizations are often used to prove comparison principles for free boundary problems (c.f. \cite{CaffarelliVazquez, CaffarelliSalsa, Kim, AKY}), as they ensure that, when the free boundaries intersect for the first time, the free boundaries have both the interior and exterior ball property at the contact point. This provides sufficient regularity to consider a first-order asymptotic expansion of the free boundary graph at the contact point. In many ways, our proof parallels previous work by Alexander, Kim, and Yao \cite[Theorem 2.7]{AKY}, the main differences being that our drift term $\grad \Phi_{1/m}$ has less regularity uniformly in $m$. This makes \ref{Pinfty} more susceptible to perturbations, so we must carefully choose our regularization procedure so that the regularized solutions remain sub- and supersolutions of the original problem. 

We now describe the details of these regularizations of the sub- and supersolutions. Fix $r_0 \in[0,e^{(-1-\sqrt{2})/2})$. Let $r(t)$ be the unique solution to 
\begin{align} \label{rdef}
\begin{cases} r'(t) &= -2C_d \sigma(r(t)) ,\\
r(0) &= r_0 ,\end{cases}
\end{align}
with $C_d$ and $\sigma(x)$ as defined in Proposition \ref{propertiesofspecialPhi}.
Given $(u, \Sigma)$ and $v$ as in Theorem \ref{comparison}, define the spatial sup and inf convolutions
\begin{equation}\label{convolution1}
u^r(x,t):= \sup_{\overline{B_{r(t)}(x)}} u(y,t), \quad v^r(x,t):=\inf_{\overline{B_{r(t)}(x)}} v(y,t), \quad \Sigma^r:= \cup_{t>0} \overline{\Omega^{r(t)}(t)} \times \{t \} ,
\end{equation}
where $\overline{\Omega^{r(t)}(t)} := \{ x: d(x,\Omega(t)) \leq r(t) \}$. Next we define the spacetime sup and inf convolutions
\begin{equation}\label{convolution2}
\tilde{u}^r(x,t):= \sup_{\overline{B_{r^*}(x,t)}} u^r(y,s), \ \ \tilde{v}^r(x,t):= \inf_{\overline{B_{r^*}(x,t)}} v^r(y,s) , \ \ \tilde{\Sigma}^r := \{ (x,t) : d((x,t), \Sigma^r) < r^* \}.
\end{equation}
for fixed $r_*>0$.

\begin{lemma}\label{super_sub}
Let $r^*:= [\sigma(r(T))/11]^{2d}$. Then $(\tilde{u}^r, \tilde{\Sigma}^r)$ and $\tilde{v}^r$ are viscosity sub- and supersolutions of \ref{Pinfty}  in  $\R^n\times (r_0, T-r^*)$.
\end{lemma}

\begin{proof}
We will show that $(\tilde{u}^r, \tilde{\Sigma}^r)$ is a subsolution of \ref{Pinfty} according to Definition \ref{subsolution}. Parallel arguments apply to prove that $\tilde{v}^r$ is a supersolution. 

We begin by showing that $(u^r,\Sigma^r)$ is a subsolution of \ref{Pinfty}. It follows quickly from the definition that $u^r$ is upper semicontinuous and $(u^r,\Sigma^r)$ satisfies \ref{doesnotjumpup}, so we devote our attention to \ref{touchabove}.

Suppose  $u^r-\varphi$ has a local maximum zero at $(x_0,t_0)$ with respect to $\overline{\Sigma^r} \cap \{t \leq t_0\}$. In case \ref{ta1}, if either $x_0 \in \Omega^r(t_0)$ or $u^r(x_0,t_0)>0$, fix $y_0\in \overline{B}_{r(t_0)}(x_0)$ so that $u(y_0,t_0) = u^r(x_0,t_0)$ and either $y_0 \in \Omega(t_0)$ or $u(y_0,t_0) >0$. Let $\tilde{\varphi}(x,t) = \varphi(x+x_0-y_0,t)$, so $u-\tilde{\varphi}$ has a local max zero in $\overline{\Sigma}\cap \{t \leq t_0\}$ at $(y_0,t_0)$. Since $(u, \Sigma)$ is a subsolution, we have $-\Delta\varphi(x_0,t_0)=-\Delta\tilde{\varphi} (y_0,t_0) \leq 1.$

Now consider case \ref{ta2}, so $u^r(x_0,t_0)=0$, with $(x_0,t_0)\in\partial \Sigma^r$ and $|\grad \varphi|(x_0, t_0)\neq 0$. Choose $y_0\in\partial B_{r(t_0)}(x_0)$ so that $(y_0,t_0)\in\partial \Sigma$ and $u(y_0,t_0) = 0$.  Since $|\nabla\varphi|(x_0,t_0) \neq 0$ and $u^r - \varphi$ has a local maximum at $(x_0,t_0)$,  $ y_0-x_0$ is parallel to $v:=\nabla\varphi/|\nabla\varphi|(x_0,t_0)$.

Define $\tilde{\varphi}(x,t):= \varphi(x-r(t)v,t)$. Then $\tilde{\varphi}(y_0,t_0) = \varphi(x_0,t_0)$ and $u - \tilde{\varphi}$ has a local maximum zero at $(y_0,t_0)$ with respect to $\overline{\{u>0\}} \cap \{ t \leq t_0\}$. Consequently,
$$
\min(-\Delta\tilde{\varphi}-1, \tilde{\varphi}_t -|\nabla\tilde{\varphi}|^2 - \nabla\tilde{\varphi} \cdot \nabla\Phi)(y_0,t_0) \leq 0.
$$
By Proposition \ref{propertiesofspecialPhi},
$$
|\nabla\Phi(y_0,t_0) -\nabla \Phi(x_0,t_0)| \leq C_d \sigma(r(t_0)) . 
$$
Since $\tilde{\varphi}_t(y_0,t_0) = \varphi_t (x_0,t_0) -r'(t_0)|\grad \varphi|(x_0,t_0)
$, we have
\begin{equation*}
\min(-\Delta\varphi-1, \varphi_t -|\nabla\varphi|^2 - \nabla\varphi \cdot \nabla\Phi -r'|\grad \varphi| - C_d \sigma(r) |\grad \varphi|)(x_0,t_0) \leq 0.
\end{equation*}
Hence, since $r'(t) = -2C_d \sigma(r(t))$,
\begin{equation} \label{strong ta2}
\min(-\Delta\varphi-1, \varphi_t -|\nabla\varphi|^2 - \nabla\varphi \cdot \nabla\Phi  + C_d\sigma(r)|\nabla\varphi|)(x_0,t_0) \leq 0,
\end{equation}
which, in particular, implies \ref{ta2}.
Therefore, $(u^r,\Sigma^r)$ is a subsolution of \ref{Pinfty}.

Now we use this fact to show that $(\tilde{u}^r, \tilde{\Sigma}^r)$ is a subsolution of \ref{Pinfty}. Again, $\tilde{u}^r$ is upper semicontinuous and $(\tilde{u}^r, \tilde{\Sigma}^r)$ satisfies \ref{doesnotjumpup}, so we devote our attention to \ref{touchabove}. Property \ref{ta1} follows as above, so we only discuss \ref{ta2}, the barrier property on the free boundary.
Suppose  $\tilde{u}^r-\varphi$ has a local maximum zero in $\overline{\tilde{\Sigma}^r}\cap \{ t \leq t_0\}$ at $(x_0,t_0)\in\partial \tilde{\Sigma}^r$ with $\tilde{u}^r(x_0, t_0) =0$  and $|\grad \varphi|(x_0,t_0) \neq 0$. Choose $(y_0,s_0)\in\partial B_{r^*}(x_0,t_0)$ so  $(y_0,s_0)\in\partial \Sigma^r$ and $u^r(y_0, s_0) = 0$. Since $|\grad \varphi| (x_0, t_0) \neq 0$ and $\tilde{u}^r - \varphi$ has a local maximum zero at $(x_0, t_0)$, $(y_0 - x_0, s_0-t_0)$ is parallel to $w:= \grad_{x,t} \varphi /|\grad_{x,t} \varphi|(x_0,t_0)$.

Define $\tilde{\varphi}(x,t) := \varphi((x,t) - r^* w)$. Then $\tilde{\varphi}(y_0, s_0) = \varphi(x_0,t_0)$ and $u^r- \tilde{\varphi}$ has a local maximum zero at $(y_0, s_0)$ with respect to $\overline{ \{u^r>0\}} \cap \{ t \leq t_0\}$. Consequently, by inequality (\ref{strong ta2}) above,
\begin{align*}
\min(-\Delta\tilde{\varphi}-1, \tilde{\varphi}_t -|\nabla\tilde{\varphi}|^2 - \nabla\tilde{\varphi} \cdot \nabla\Phi + C_d \sigma(r)|\nabla\tilde{\varphi}|)(y_0,s_0) \leq 0.
\end{align*}
For $0 \leq x \leq e^{(-1-\sqrt{2})/2}$, we have $-\log(x) \leq x^{-1/2} \leq x^{-1+1/2d}$, hence $\sigma(x) \leq x^{1/2d}$. Therefore, by Proposition \ref{propertiesofspecialPhi}, 
\[ |\grad \Phi(y_0,s_0) - \grad \Phi(x_0,t_0)| \leq 11 C_d (r^*)^{1/2d} = C_d \sigma(r(T)) \leq C_d\sigma(r(s_0)), \] and thus it follows that
\begin{equation*}
\min(-\Delta\varphi-1, \varphi_t -|\nabla\varphi|^2 - \nabla\varphi \cdot \nabla\Phi )(x_0,t_0) \leq 0,
\end{equation*}
which concludes the proof.
\end{proof}

The rest of the proof of Theorem~\ref{comparison} parallels  \cite[Theorem 2.7]{AKY}, so we omit the proof.

\subsection{Convergence of \ref{Pm} to \ref{Pinfty}}

In this section, we show that, as $m \to +\infty$, viscosity solutions $p_m$ of \ref{Pm} approach a solution $p$ of \ref{Pinfty} and use this to show that patch solutions to the congested aggregation equation satisfy $\rho_\infty = \chi_{\Omega(t)}$ almost everywhere, where $\Omega(t) = \{p(\cdot,t)>0\}$.

 We begin with the following lemma which states that $\rho_m$ converges to $\rho_\infty$ weakly even if $\rho_m$ has initial data $(\frac{m}{m-1} p_0)^{1/(m-1)}$, instead of requiring the initial data of $\rho_m$ to coincide with the initial data of $\rho_\infty$, as proved in Theorem \ref{thm:conv_w2}.
\begin{lemma}\label{lem:weak_conv}
Let $\Omega_0 \subseteq \mathbb{R}^d$ be a bounded domain with Lipschitz boundary, and let $\rho_\infty(\cdot, t) $ be the gradient flow of $E_\infty$ with initial data $\rho_0 = \chi_{\Omega_0}$.  Let $\rho_m$ be the weak solution of \ref{pme} with initial data $(\frac{m}{m-1}p_0)^{1/(m-1)}$, where $p_0$ is as in \eqref{initial}.   Then for any $t\geq 0$ and any $f\in C(\mathbb{R}^d)$, we have
\begin{equation}
\label{eq:weak_conv}
\lim_{m\to\infty}\int_{\mathbb{R}^d} \rho_m(x,t) f(x) dx = \int_{\mathbb{R}^d} \rho_\infty(x,t) f(x) dx \quad\text{ for all }t\geq 0.
\end{equation}
\end{lemma}

\begin{proof}
We will first prove \eqref{eq:weak_conv} for all $f\in C(\mathbb{R}^d)\cap L^\infty(\mathbb{R}^d)$, and at the end of the proof we will  extend it to all (possibly unbounded) continuous functions.
 
Let $\tilde \rho_m$ be the weak solution of \ref{pme} with initial data $\chi_{\Omega_0}$. Theorem \ref{thm:conv_w2} then yields that $\lim_{m\to\infty} W_2(\tilde \rho_m(t), \rho_\infty(t)) = 0$ for any $t>0$. By \cite[Remark 7.1.11 and Remark 5.1.2]{AGS}, convergence in $W_2$ distance implies that
\begin{equation}
\label{eq:weak1}
\lim_{m\to\infty}\int_{\mathbb{R}^d} \tilde \rho_m(x,t) f(x) dx = \int_{\mathbb{R}^d} \rho_\infty(x,t) f(x) dx \quad\text{ for all }f\in C(\mathbb{R}^d)\cap L^\infty(\mathbb{R}^d).
\end{equation}

To relate $\tilde \rho_m$ with $\rho_m$, note that they are both weak solutions to \ref{pme}, with different initial data $\chi_{\Omega_0}$ and $(\frac{m}{m-1}p_0)^{1/(m-1)}$ respectively. Since $(\frac{m}{m-1}p_0)^{1/(m-1)}\to \chi_{\Omega_0}$ pointwise as $m\to +\infty$, we have
\[
\lim_{m\to\infty}\Big\|\Big(\frac{m}{m-1}p_0\Big)^{1/(m-1)} - \chi_{\Omega_0}\Big\|_{L^1(\mathbb{R}^d)}=0
\] by dominated convergence theorem. Also, recall that for any $m>1$ and $t\geq 0$, the $L^1$ contraction result in Lemma \ref{compact_spt} gives
\[
\|\tilde \rho_m(\cdot, t) - \rho_m(\cdot, t)\|_{L^1(\mathbb{R}^d)} \leq \Big\|\Big(\frac{m}{m-1}p_0\Big)^{1/(m-1)} - \chi_{\Omega_0}\Big\|_{L^1(\mathbb{R}^d)}.
\]
Combining the above two equations yields $\lim_{m\to\infty} \|\tilde \rho_m(\cdot, t) - \rho_m(\cdot, t)\|_{L^1(\mathbb{R}^d)} =0$, hence
\begin{equation}
\label{eq:weak2}
\lim_{m\to\infty}\int_{\mathbb{R}^d} (\tilde \rho_m(x,t)- \rho_m(x,t)) f(x) dx  =0 \quad\text{ for all }f \in L^\infty(\mathbb{R}^d).
\end{equation}
Putting \eqref{eq:weak1} and \eqref{eq:weak2} together gives us \eqref{eq:weak_conv} for all $f\in C(\mathbb{R}^d) \cap L^\infty(\mathbb{R}^d)$. To remove the requirement $f\in L^\infty(\mathbb{R}^d)$, recall that Lemma \ref{compact_spt} gives that $\tilde \rho_m(\cdot, t)$ is supported in some bounded set $B(0,R(t))$ for all $m$, and as a result $\rho_\infty(\cdot,t)$ is supported in it too. If $f\in C(\mathbb{R}^d)$ is unbounded, we can simply set $\tilde f = f \eta$, where $\eta$ is a smooth cut-off function that is 1 in $B(0,R(t))$ and $0$ outside of $B(0,R(t)+1)$. We then have \eqref{eq:weak_conv} holds for $\tilde f$. Since changing $\tilde f$ to $f$ will not change the integrals in \eqref{eq:weak_conv}, we know \eqref{eq:weak_conv} holds for $f$ too.
\end{proof}

We begin our study of the limit of solutions of \ref{Pm} with the following result, which shows that the half relaxed ``limit infimum'' of solutions of \ref{Pm} is a supersolution of \ref{Pinfty}.

\begin{proposition}\label{super}
Suppose $p_m(x,t)$ is a viscosity solution of \ref{Pm} with initial data $p_m(\cdot, 0) = p_0$ as given in \eqref{initial}.
Then the half relaxed limit
\begin{align} \label{u2def}
u_2(x,t): = \liminf_*p_m(x,t)=\lim_{n\to +\infty}\inf_{\substack{m>n \\ |(x,t) - (y,s)| < 1/n}} {p_m}(y,s)
\end{align}
is a viscosity supersolution of \ref{Pinfty}.  
\end{proposition}
\begin{proof} 
First, note that $u_2$ is lower semicontinuous.
Next, suppose that $u_2 - \varphi$ has a local minimum zero at $(x_0,t_0)$ with respect to $\Rd \cap \{ t \leq t_0\}$. By subtracting $\delta(x-x_0)^2+\delta(t_0-t)$ from $\varphi$ with $\delta >0$ sufficiently small, we may assume that $u_2 - \varphi$ has a strict  minimum at $(x_0,t_0)$  with respect to a parabolic neighborhood $Q$ of $(x_0, t_0)$. Define $(x_m, t_m) := \argmin_Q p_m-\varphi$, $C_m := p_m(x_m,t_m) - \varphi(x_m,t_m)$, and $\varphi_m := \varphi +C_m$, so $p_m -\varphi_m$ has a local minimum zero at $(x_m,t_m)$. As in \cite[Theorem 3.4]{AKY} (see paragraph A.2), up to a subsequence, we have $(x_m, t_m) \to (x_0, t_0)$, $\lim_{m \to +\infty} p_m(x_m,t_m)= \liminf_*p_m(x_0,t_0) = u_2(x_0,t_0)$, and $(x_m,t_m) \in \overline{\{p_m>0\}}$.
Since $p_m$ is a viscosity supersolution of \ref{Pm},
\begin{equation} \label{eqn:p2a}
  (\varphi_t - (m-1)p_m(\Delta \varphi + \Delta \Phi_{1/m}) -\nabla \varphi \cdot(\nabla \varphi + \nabla\Phi_{1/m}))(x_m,t_m) \geq 0 \quad\text{ for all }m.
  \end{equation}

First, consider the case when $u_2(x_0,t_0) >0$, and assume $-\Delta \varphi(x_0,t_0) <1$ for the sake of contradiction.
Since $\varphi \in C^{2,1}(Q)$, we have $\varphi_t - \nabla \varphi\cdot(\nabla\varphi + \nabla \Phi_{1/m})|_{(x_m,t_m)}$ is uniformly bounded for all $m$. Hence if we can show that 
\begin{equation}\label{eq:temp_contra}
\limsup_{m\to\infty}  p_m(x_m,t_m) \big(\Delta \varphi(x_m,t_m) + \Delta \Phi_{1/m}(x_m,t_m)\big)> c>0
\end{equation}
for some positive $c$, it would imply that the left hand side of \eqref{eqn:p2a} goes to $-\infty$ along a subsequence of $m\to\infty$,  contradicting \eqref{eqn:p2a}.
Let us take a subsequence of $m\to +\infty$ such that $(x_m, t_m) \to (x_0, t_0)$ and $p_m(x_m,t_m)>\frac{1}{2}u_2(x_0,t_0)>0$ for all terms in this subsequence. (We still denote this subsequence by $m$ for notational simplicity.) Since $\lim_{m\to\infty} \Delta \varphi(x_m,t_m) = \Delta\varphi(x_0,t_0)>-1$, to show \eqref{eq:temp_contra}, it suffices to show that $\Delta \Phi_{1/m}(x_m, t_m) = 1$ for all sufficiently large $m$ in this subsequence. Since $(x_m, t_m)$ minimizes $p_m-\varphi$ in $Q$, combining this with  $p_m(x_m,t_m)>\frac{1}{2}u_2(x_0,t_0)$ yields that $p_m(x, t) > \frac{1}{4} u_2(x_0,t_0)>0$ for all $(x,t)\in Q \cap B_{r_0}(x_m,t_m)$, where $r_0>0$ is a sufficiently small constant only depending on $u_2(x_0,t_0)$ and $\|\varphi\|_{C^{1,1}(Q)}$. Also, note that we have $B_{r_0/2}(x_0,t_0) \subseteq B_{r_0}(x_m,t_m)$ for sufficiently large $m$ due to the fact that $(x_m,t_m)\to (x_0,t_0)$. In other words, for all sufficiently large $m$, we have
\[
 \rho_m(x,t) = \left(\frac{m}{m-1} p_m(x,t)\right)^{\frac{1}{m-1}} \geq \left(\frac{u_2(x_0,t_0)}{8}\right)^{\frac{1}{m-1}}=:c_2^{\frac{1}{m-1}} \text{ for all }(x,t)\in Q \cap B_{r_0/2}(x_0,t_0).
\]
Since $\lim_{m\to\infty}c_2^{\frac{1}{m-1}}=1$, combining the above inequality with the weak convergence of $\rho_m$ towards $\rho_\infty$ in Lemma~\ref{lem:weak_conv} yields that $\rho_\infty(x,t) \geq 1$ almost everywhere in $ Q \cap B_{r_0/2}(x_0,t_0)$ (indeed, we must have $\rho_\infty = 1$ in this set since $\rho_\infty \leq 1$ by definition), hence
\[
\Delta\Phi_{1/m}(x_m,t_m)= \psi_{1/m}*\rho_{\infty}(x_m,t_m)  = 1
\]
for all sufficiently large $m$ in this subsequence, which yields \eqref{eq:temp_contra} and thus contradicts \eqref{eqn:p2a}.

\medskip

Now, we consider the case when $(x_0,t_0) \in \partial \{ u_2 >0\}$, $u_2(x_0,t_0) = 0$ and $\varphi$ satisfies (\ref{bdryTouchingCond}).  For a given small $a>0$, let us consider a (second-order) perturbed version of $\varphi$, 
$$
\tilde{\varphi}(x,t):= \varphi(x,t) - a(x-x_0)\cdot\nu -aC((x-x_0)^T)^2+ M((x-x_0)\cdot\nu)^2.
$$ 
Here $\nu = D\varphi(x_0,t_0)\neq 0$, $(x-x_0)^T = (x-x_0) - [(x-x_0)\cdot\nu] \nu$, $C= \max |D^2\varphi|$. We choose $M$ sufficiently large so that 
\begin{equation}\label{eqn000}
\Delta\tilde{\varphi} (x_0,t_0)>0.
\end{equation}

 Note that in a small neighborhood of $(x_0,t_0)$ that depends on $a$ and $M$, $\tilde{\varphi} \leq \max[\varphi,0].$ Thus $\tilde{\varphi}$ still touches $u_2$ from below at $(x_0,t_0)$. We will show that, for any choice of $a>0$, 
$$
(\tilde{\varphi}_t - |\nabla\tilde{\varphi}|^2-\nabla \tilde{\varphi} \cdot \nabla \Phi)(x_0, t_0) \geq 0,
$$
which yields the desired conclusion for $\varphi$.

\medskip

 Assume, for the sake of contradiction, that  there is $c_0>0$ so
\begin{align} \label{u2super1}
\tilde{\varphi}_t - |\nabla\tilde{\varphi}|^2-\nabla \tilde{\varphi} \cdot \nabla \Phi(x_0, t_0) < -c_0 < 0.
\end{align}

Let $(x_m,t_m)$ chosen as before but with $\tilde{\varphi}$ instead of $\varphi$. First, suppose  $p_m(x_m,t_m)>0$ for sufficiently large $m$. Then \eqref{eqn:p2a} and \eqref{eqn000}-\eqref{u2super1} yields a contradiction.

Lastly  suppose $(x_m,t_m) \not \in \{p_m>0\}$ for $m$ sufficiently large. Since  $(x_m,t_m) \in \overline{\{p_m>0\}}$, up to a subsequence, we have $(x_m, t_m) \in \partial \{p_m >0\}$. This contradicts \cite[Lemma 3.3]{AKY}. (Though the proof of this lemma in \cite{AKY} only considers the time independent case, it continues to hold in our setting.)

\end{proof}

Next we proceed to show that taking a ``limit supremum" of $p_m$  yields a subsolution of \ref{P}. Here we need to be a bit careful, due to the fact that subsolution property is based on maximum points  only in the support of the subsolution. (See Definition \ref{subsolution}.) Indeed, due to the nature of one-phase problem it is not possible to perturb a smooth test function $\varphi$ to create a strict maximum of $u^1-\varphi$ without restricting the domain to $\overline{\{u^1>0\}}$. This can create technical difficulties with arguments along the lines of above proof to ensure that the local maximum points are stable under the limit $m\to\infty$, especially when the support of $p_m$ degenerates as $m\to\infty$. 
To overcome this obstacle, we work with a modified notion of viscosity subsolutions, which are comprised of a pair $(u, \Sigma)$. This allows the set evolution $\Sigma$ to be larger than the support of $u$. (See Definition~\ref{subsolution} for details.)

\begin{proposition} \label{sub}
Suppose $p_m(x,t)$ is a viscosity solution of \ref{Pm}. Define
\begin{equation}
\label{eq:st_def}
S(t) := \overline{\cap_{M \geq 1}(\cup_{m>M}{\{p_m(\cdot,t)>0\}})}, \quad \Sigma_1 :=  \cup_{t>0} ( \mathring S(t) \times \{t\})
\end{equation}
Then if $u_1$ is the half-relaxed limit of $p_m$,
\begin{align*}
u_1(x,t): = \limsup^*p_m(x,t) =\lim_{n \to + \infty} \sup_{\substack{m>n \\ |(x,t) - (y,s)| < 1/n}} {p_m}(y,s) ,
\end{align*}
$(u_1, \Sigma_1)$ is a viscosity subsolution of \ref{Pinfty}.  
\end{proposition}

\begin{proof}
Since $u_1$ is upper semicontinuous and $\mathring S(t)$ is open, it remains to check properties \ref{doesnotjumpup}-\ref{touchabove} of Definition~\ref{subsolution}. By definition, $S(t)^c \subseteq \{ u_1(\cdot, t)  = 0 \}$, hence $\{u_1(\cdot,t)>0 \} \mathring{} \subseteq \mathring S(t)$. By \cite[Theorem B.1]{AKY}, for all $t_0>0$, 
\begin{align} \label{jumpingout}
\text{ if } (x_0,t_0) \in \Sigma_1 \text{ then } (x_0,t_0) \in \left( \Sigma_1 \cap \{ t \geq t_0 \} \right)\mathring{} \end{align}
In particular, we have,
$\Sigma \cap \{ t \leq t_0 \} \subseteq \overline{\Sigma \cap \{t < t_0\}}$ for all $t_0 >0$.

Now we turn to property \ref{touchabove}. Let $u- \varphi$ have a local maximum zero at $(x_0,t_0)$ in $\overline{\Sigma_1}\cap \{t \leq t_0\}$. First we consider part \ref{ta1}, where either $(x_0,t_0) \in \Sigma_1$ or $u_1(x_0,t_0)>0$.  By adding $\delta(x-x_0)^2+\delta(t_0-t)$ to $\varphi$ with $\delta >0$ sufficiently small, we may assume there is a parabolic neighborhood $Q$ of $(x_0, t_0)$ so that $u_1 - \varphi$ has a strict maximum with respect to $\overline{Q} \cap \overline{\Sigma_1} \cap \{ t \leq t_0 \} = \overline{Q} \cap \overline{\Sigma_1}$.

First, note that if suffices to consider the case when $u_1(x_0,t_0)>0$. In particular, if $(x_0,t_0) \in \Sigma_1$, then by  (\ref{jumpingout}), we may assume that $Q$ is sufficiently small so that $\overline{Q} \subseteq \overline{\Sigma_1} $. This implies  $u_1(x_0,t_0) > 0$, since otherwise $\varphi(\cdot,t_0)$ has a local minimum zero at $x_0$, contradicting the fact that it is superharmonic.
Likewise, we may assume that $u_1 - \varphi$ has a strict maximum zero at $(x_0,t_0)$ with respect to $\overline{Q}$, since $\varphi(x_0,t_0) = u_1(x_0,t_0)>0$ and $u_1 =0$ in ${(\overline{\Sigma_1})}^c$. 

We now show  that
\begin{equation}\label{interior}
-\Delta \varphi(x_0,t_0) \leq f(x_0) \text{ for any } f \in C(\Rd) \text{ such that } f \geq \rho_\infty(\cdot, t_0) \text{ almost everywhere.}
\end{equation}
In particular, this implies that $-\Delta \varphi(x_0,t_0) \leq 1$, which gives the result. Suppose for the sake of contradiction that $-\Delta \varphi(x_0,t_0)>f(x_0)$ for some $f$.

Let $(x_m, t_m) := \argmax_{\overline{Q}} p_m-\tilde{\varphi}$, $C_m := p_m(x_m,t_m) - \tilde \varphi(x_m,t_m)$, and $\tilde \varphi_m := \tilde \varphi +C_m$, so $p_m -\tilde \varphi_m$ has a maximum zero at $(x_m,t_m)$ with respect to $\overline{Q}$.  
As in \cite[Theorem 3.4]{AKY} (see paragraph A.2), up to a subsequence, we have $(x_m, t_m) \to (x_0, t_0) \in Q$, $\lim_{m \to +\infty} p_m(x_m,t_m)= \limsup^*p_m(x_0,t_0) = u_1(x_0,t_0) >0$.
Since $p_m$ is a viscosity subsolution of \ref{Pm},
\begin{equation} \label{eqn:p2b}
  (\varphi_t - (m-1)p_m(\Delta \varphi + \Delta \Phi_{1/m}) -\nabla \varphi \cdot(\nabla \varphi + \nabla\Phi_{1/m}))(x_m,t_m) \leq 0.
  \end{equation}
Because $\Delta \Phi_{1/m} = \psi_{1/m} * \rho_\infty \leq f + o(1)$ and $-\Delta \varphi(x_0,t_0) >f(x_0)$, we have 
$$(\Delta \varphi + \Delta \Phi_{1/m})(x_m,t_m) <0$$ for sufficiently large $m$, which is a contradiction.

Now we consider part \ref{ta2}, where $(x_0,t_0) \in \partial \Sigma_1$, $u_1(x_0,t_0) = 0$, and $|\grad \varphi|(x_0,t_0) \neq 0$.
Suppose, for the sake of contradiction, that 
\begin{equation}\label{supereqn}
-\Delta\varphi(x_0,t_0)>1 \text{ and }  (\varphi_t - |\grad \varphi|^2 - \grad \varphi\cdot \grad \Phi )(x_0,t_0)>0 .
\end{equation}
 We can now apply parallel argument as in the proof of Theorem 3.4 in [AKY] to conclude.

\end{proof}

We now show that the initial data of the half-relaxed limits coincides with the initial data given in equation (\ref{initial}). Specifically, this ensures that the initial data of $u_1$ and $u_2$ coincides with the initial data of the sequence $p_m$, in spite of the time regularization inherent in the definitions of these half-relaxed limits. In what follows, we make frequent use of the following  inner and outer approximations:
\begin{align} \label{inner outer}
\Omega^{-r}:= \{x: d(x,\Omega^c) > r\} \text{ and }\Omega^r:= \{x: d(x,\Omega)< r\}.
\end{align}

\begin{lemma} \label{u1u2initialdata}
Consider a bounded domain $\Omega_0 \subseteq \Rd$ with  the``no-crack" property $\mathring{\overline{\Omega_0}}= \Omega_0$. Suppose $p_m$ are viscosity solutions of \ref{Pm} with initial data $p_0$ as given in \eqref{initial}. Then, for the half-relaxed limits $u_1$ and $u_2$, we have $u_{1}(x,0)= u_2(x,0) =p_0(x)$.
\end{lemma}

\begin{proof} 
1. We begin by proving  the following claim on the support of $p_m$ for given $\e>0$:
\begin{align} \label{inclusion}
\text{There is } T_\epsilon >0 \text{ such that } \Omega_0^{-\e} \subseteq \{u_i(\cdot,t)>0\}\subseteq \Omega_0^{\e} \hbox{ for all } t \in [0,T_\e] \hbox{ and }  i=1,2 .\end{align} 
We begin by showing $\{p_m(\cdot,t)>0\}\subseteq \Omega_0^{\e}$ for $0\leq t\leq t_\epsilon$ for some $t_\epsilon\in (0,1)$ that is independent of $m$. 
Suppose $ x_0 \in (\Omega_0^\e)^c$, so that $p_m=0$ in $B_\e(x_0)$ for all $m$.  

 Let us define
\[
\phi(x,t) = \Big(\mathcal{N}(x-x_0) - \frac{|x-x_0|^2}{2d} + f(t)\Big)_+,
\]
where $f$ is an increasing function which we will determine momentarily. Let $f(0) = -\mathcal{N}(\epsilon) + C_1 + 1$, where $C_1>0$ is such that $p_m(\cdot,t) \leq C_1$ for all $m$ and $t\in [0,1]$, given by Lemma~\ref{compact_spt}(b).  Such choice of $f$ guarantees that $\phi(x,0) \geq C_1 \geq p_m(x,0)$ in $\{\epsilon \leq |x-x_0|\leq 1\}$ (hence $\phi(\cdot, 0)\geq p_m(\cdot,0)$ in $B_{1}(x_0)$), and $\phi(x,t) \geq C_1 \geq  p_m(x,t)$ on $\partial B_{1}(x_0)$ for all $t\in [0,1]$.

\medskip

We claim that if we let $f(t)$ increase sufficiently fast, $\phi$ would be a classical supersolution of (P)$_m$ for all $m$ in $B_1(x_0)\times [0,t_\epsilon]$ for some $t_\epsilon>0$. Note that at time $t$, $\phi(\cdot, t)$ has support $\{r(t) \leq |x-x_0| \leq 1\}$, where $r(t) \in (0,1)$ solves $\mathcal{N}(r(t)) - \frac{r(t)^2}{2d} + f(t) = 0$, hence it satisfies $r(t) > \mathcal{N}^{-1}(-f(t)+1) > 0$. (Here $\mathcal{N}^{-1}$ is the inverse function of $\mathcal{N}$). By definition, $\Delta \phi = -1$ in its support. Thus in order to make $\phi$ a classical supersolution of (P)$_m$, all we need is $\phi_t \geq |\nabla \phi| (|\nabla \phi| + |\nabla \Phi_{1/m}|)$ everywhere in its support. In the support of $\phi$, we have $\phi_t = f'(t)$, and $|\nabla \phi| \leq \mathcal{N}'(r(t)) + 1 \leq \mathcal{N}'( \mathcal{N}^{-1}(-f(t)+1)) + 1$. Finally, let 
$$
f'(t) = (\mathcal{N}'( \mathcal{N}^{-1}(-f(t)+1)) + 1 + C_d)^2,
$$ where $C_d$ is the bound for $|\nabla \Phi_{1/m}|$ as in \eqref{eq:phi_bound}. The standard ODE theory guarantees that $f$ is finite in some $[0,t_\epsilon]$ (where $t_\epsilon>0$ depends on $\epsilon$, $C_1$ and $d$), hence $r(t)>0$ in $[0,t_\epsilon]$. By comparing $p_m$ with $\phi$ in the domain  $B_1(x_0) \times [0,t_\epsilon]$ and using the definition of viscosity solutions, we conclude that $p_m = 0$ in $B_{r(t)}(x_0)$ for $t\in [0,t_\epsilon]$.  In particular, $x_0 \in \{p_m(\cdot, t)>0\}^c$ for all $m$ and all $t\in [0,t_\epsilon]$, and since $x_0 \in (\Omega_0^\e)^c$ was arbitrary, this gives the result.

\medskip


\medskip

Similarly we show $\Omega_0^{-\e} \subseteq \{p_m(\cdot,t)>0\}$ for small times by constructing a classical subsolution of \ref{Pm}. Suppose $y_0 \in \Omega_0^{-\e}$, so that $B_\e(y_0) \subseteq \Omega_0$. 
 Let $h_m(x,t)$ solve $-\Delta h_m(\cdot,t) = \frac{1}{m}$ in $B_{r(t)}(y_0)$, with $h_m(\cdot,t) = 0$ on $\partial B_{r(t)}(y_0)$. Here $r(t):=\epsilon-Mt$, and $M$ is a large constant to be determined later. 
Note that $h_m$ takes the explicit expression
$
h_m(x,t) := \left( \frac{r(t)^2 - |x-y_0|^2}{2dm}\right)_+,
$ thus $|\nabla h_m(\cdot, t)| \leq r(t)/dm$ in its support. So the following holds in the support of $h_m$:
\[
\begin{split}
(m-1)h_m (\Delta h_m+\Delta \Phi_{1/m}) + \nabla h_m \cdot (\nabla h_m + \nabla \Phi_{1/m}) &\geq -\frac{m-1}{m}h_m - \frac{r(t)}{dm} \left(\frac{r(t)}{dm} + C_d\right)\\
&\geq -\frac{r(t)^2}{2dm} - \frac{r(t)}{dm} \left(\frac{r(t)}{dm} + C_d\right),
\end{split}
\]
where $C_d$ is the bound for $\|\nabla \Phi_{1/m}\|$ by \eqref{eq:phi_bound}, and we also used $\Delta \Phi_{1/m}\geq 0$ in the first inequality.  Since $(h_m)_t = r(t)r'(t)/dm$ in its support, in order for $h_m$ to be a classical subsolution of (P)$_m$, all we need is 
$r' \leq -r/2 - (r+\|\nabla \Phi_{1/m}\|_\infty)$, so we can simply let $r(t)=\epsilon-Mt$ with $M=1+C_d$.
%

\medskip

Since $p_m(\cdot, 0) \geq h_m(\cdot, 0)$ for all $m>1$, comparison principle yields that $p_m \geq h_m$ for $0\leq t\leq \frac{\e}{2M}$. It follows that $p_m\geq h_m \geq \frac{\epsilon^2}{16dm}$ in $\Sigma_\e:=B_{\e/4}(y_0) \times [0,  \frac{\e}{2M}]$. Even though this lower bound of $p_m$ is not uniformly positive in $m$, we can still conclude that $\liminf_{m\to\infty}\rho_m \geq 1$ in $\Sigma_\e$ by definition of $\rho_m = (\frac{m}{m-1} p_m)^{1/(m-1)}$. Given the weak convergence of $\rho_m$ to $\rho_{\infty}$ in Lemma \ref{lem:weak_conv}, we have that $\rho_{\infty} = 1$ almost everywhere in $\Sigma_\e$. This implies that $\Delta \Phi_{1/m} = \rho_\infty * \psi_{1/m} \equiv 1$ in $B_{\e/8}(y_0) \times [0, \frac{\e}{2M}]$ for all sufficiently large $m$ (more precisely, for all $m>8/\e$).

\medskip

With this information on $\Delta \Phi_{1/m}$, we can now define a new subsolution $\varphi(x,t)$ that solves $-\Delta \varphi = 1$ in $B_{\tilde r(t)}(y_0)$, with $\varphi(\cdot, t) = 0$ on $\partial B_{\tilde r(t)}(y_0)$, where $\tilde r(t) = \epsilon/8 - Mt$, and $M=1+C_d$. One can check that $\varphi$ is a classical subsolution of (P)$_m$, hence $p_m \geq \varphi \geq c_\epsilon$ for some $c_\epsilon>0$ (that is independent of $m$) in $B_{\epsilon/32}\times [0,T_\e]$ for all sufficiently large $m$, where $T_\epsilon := \frac{\epsilon}{16M}$, yielding \eqref{inclusion}.

\medskip

2. To show that $u_1(\cdot,0)=u_2(\cdot,0)=p_0$, we construct our first barrier as follows: suppose $h_\e(x)$ solves 
$$
-\Delta h_\e = 1+\e \hbox{ in } \Omega^\e, \quad h_\e = 0 \hbox{ on } \partial\Omega^\e.
$$
By (\ref{inclusion}), $S(t) \subseteq \Omega^{\e/2}$ for $t \in [0,\tilde T_\e]$ for some $\tilde T_\e >0$. Thus, $u_1 \leq h_\e$ in $S(t)$. Furthermore, since $u_1 =0$ in $(S(t))^c$, we conclude that
\begin{equation}\label{order_1}
u_1(\cdot,t) \leq h_\e\hbox{ in }\Omega^{\e} \times [0,\tilde T_\e].
\end{equation}

Next, comparison of $p_m$ with the classical subsolution $\varphi$ given above yields that
\begin{equation}\label{lower_bd}
p_m  \geq c_\e\hbox{ in }\Omega^{-\e}\times [0,T_\e].
\end{equation}
Now we construct our second barrier using \eqref{lower_bd}. Consider $g(x)$ solving
$$
-\Delta g= 1-\e \hbox{ in } \Omega^{-\e}, \quad g = c_\e \hbox{ on } \partial\Omega^{-\e}.
$$
Since $u_2$ is a supersolution of \ref{Pinfty} and (\ref{inclusion}) ensures that $\Omega^{-\e} \subseteq \{u_2>0\}$ for $t \in [0,T_\e]$, we have $-\Delta u_2 \geq 1$ in $\Omega^{-\e}\times[0,T_\e]$. Furthermore, (\ref{lower_bd}) ensures that $g \leq u_2$ on $\partial \Omega^{-\e} \times [0,T_\e]$. Therefore,
\begin{equation}\label{order_2}
g\leq u_2 \hbox{ in } \Omega^{-\e}\times [0,T_\e].
\end{equation}

Combining inequalities \eqref{order_1} and \eqref{order_2}  and sending $\epsilon \to 0$, we can conclude.
\end{proof}

We now show our main convergence theorem.

\begin{theorem}\label{convergence}
Let $\Omega_0 \subseteq \Rd$ be a bounded domain with Lipschitz boundary, and let $p_m$ solve \ref{Pm} with initial data $p_0$ as given in \eqref{initial}. Let $\rho_m$ be the density variable corresponding to $p_m$, and  let $U$ be the unique maximal solution of \ref{Pinfty} with initial data $p_0 $, i.e.
\[ U(x,t):= (\inf \{ w: w \text{ is a viscosity supersolution of \ref{Pinfty} with }  w(\cdot, 0)\geq p_0 \})_*.\]

Then the following holds for each $t>0$:
\begin{enumerate}[label = (\alph*)]
\item $\rho_{\infty}(\cdot,t)=\chi_{\{u_1(\cdot,t)>0\}} =\chi_{\{u_2(\cdot,t)>0\}} =\chi_{\{U(\cdot,t)>0\}}$ almost everywhere; \label{convergence part a}
\item Let $\Omega(t):=\{u_2(\cdot, t)>0\}$, and $\Omega^1(t) := \{u_1(\cdot, t)>0\}$. Then for every $t\geq 0$, $\Omega(t)$ is an open set with $|\partial \Omega(t)|=0$, and we also have $|\partial \Omega^1(t)|=0$; \label{convergence part b}
\item $\rho_m$ converges to $1$ uniformly in $\Omega(t)$ away from its boundary---that is, the convergence is uniform in any compact set $Q \subseteq \{(x,t): x\in \Omega(t)\}$. Furthermore, we have $\lim_{m\to\infty}\|\rho_m(\cdot, t) - \chi_{\Omega(t)}\|_{L^1(\mathbb{R}^d)} = 0$ for every $t\geq 0$. 
\end{enumerate}
\end{theorem}

\begin{remark}
The fact that $U$ is a solution of \ref{Pinfty} is a consequence of a standard Perron's method argument. 
\end{remark}

\begin{proof}1. To begin with, let us define two families of functions that are approximations of $p_m$.
For $n\in \mathbb{N}$, let $p_0^{n,-}(x)$ and $p_0^{n,+}(x)$ be solutions to \eqref{initial} with $\Omega_0$ replaced by $\Omega_0^{-1/n}$ and $\Omega_0^{1/n}$ (as defined in \eqref{inner outer}) respectively. Note that $p_0^{n,-} \prec p_0 \prec p_0^{n,+}$.   
We then let $p_m^{n,-}$ be the viscosity solution to \ref{Pm} with initial data $p_0^{n,-}$, and denote by  $\rho_m^{n,-}$ the corresponding density function. We let $p_m^{n,+}$ solve a modified version of \ref{Pm} with an extra source term, namely
\[
p_t = (m-1) p(\Delta p + \Delta \Phi_{1/m}) + \nabla p \cdot (\nabla p + \nabla \Phi_{1/m}) + p f_n,
\]
where 
$$
 f_n:= \chi_{\Omega_n- \Omega}, \quad \Omega:=\{u_2>0\} \hbox{ and } \Omega_n:=\{(x,t):d((x,t),\Omega)\leq\frac{1}{n}\},
$$
and denote by  $\rho_m^{n,+}$ the corresponding density function. Finally we let $(u_1^{n,-}, S^{n,-}(t))$ and $u_2^{n,+}$ denote the corresponding  half-relaxed limits for $p_m^{n,-}$ and $p_m^{n,+}$. 

The motivation of these two family of functions is as follows: in step 2, we will show that
\begin{equation}
\label{eq:comp3}
u_1^{n,-} \leq U \leq u_2 \leq u_1 \leq u_2^{n,+} \quad\text{ for any }n\in \mathbb{N},
\end{equation}
and it turns out that in order to show the last inequality we have to let $p_m^{n,+}$ solve the equation with the extra source term, rather than \ref{Pm}.
In step 3, we will use $L^1$ contraction result between $\rho_m^{n,-}$ and $\rho_m^{n,+}$ to show that for any $t\geq 0$,
\begin{equation}
\label{eq:contraction_sup}
A_n(t) :=\big|\supp u_2^{n,+}(\cdot, t)\setminus \supp u_1^{n,-}(\cdot,t)\big|  \to 0 \text{ as }n\to\infty,
\end{equation}
and by combining it with \eqref{eq:comp3} we have that $U$, $u_1$ and $u_2$ are supported on the same set.
\medskip

2. In this step we aim to prove \eqref{eq:comp3}. The second inequality is a direct consequence from the minimality of $U$ and the fact that $u_2$ is supersolution of \ref{Pinfty} with initial data $p_0$ (which follows from Proposition \ref{super} and Lemma \ref{u1u2initialdata}). The third inequality immediately follows from the definition of the half relaxed limits $u_2$ and $u_1$. As for the first inequality, note that by Proposition \ref{sub}, $(u_1^{n,-}, S^{n,-}(t))$ is a subsolution of \ref{Pm}. (Proposition \ref{sub} does not require the initial data be the same as $p_0$.) In addition, we have $u_1^{n,-}(\cdot, 0) = p_0^{n,-}$ via the same argument as in Lemma \ref{u1u2initialdata}. Since $p_0^{n,-}\prec p_0$, combining the above discussion on $u_1^{n,-}$ with Proposition \ref{super} and the comparison principle in Theorem \ref{comparison} yields
   \begin{equation}\label{last0}
u_1^{n,-}\prec u_2 \hbox{ with } \overline{S_1^{n,-}(t)} \subseteq \{u_2(\cdot,t)>0\},
\end{equation}
which gives us the first inequality.

The last inequality of \eqref{eq:comp3} is more difficult to obtain. We point out that this is not a direct consequence of the comparison principle for \ref{Pinfty}, since we do not know that $u_2^{n,+}$ is a supersolution of \ref{Pinfty} due to the fact that $p_m^{n,+}(\cdot, 0) \neq p_0$. (In order to apply Proposition \ref{super}, the initial data must be the same as $p_0$.)

To overcome this difficulty, we will show that $u_1$ and $u_2^{n,+}$ are sub- and supersolutions of another free boundary problem, for which the comparison principle also holds.
  From the proof (in particular \eqref{interior}) of  Proposition \ref{sub}, it follows that in addition to $(u_1, S(t))$ being a viscosity subsolution of \ref{Pinfty}, $u_1$ satisfies
\begin{equation}\label{claim001}
-\Delta u_1 (\cdot,t) \leq \rho_{\infty} 
\end{equation} 
in the integral sense. On the other hand, parallel arguments as in Proposition~\ref{super} yield that $u_2^{n,+}$ satisfies supersolution properties of \ref{Pinfty}(see Definition~\ref{supersolution}), but with the interior operator $-\Delta-1$ replaced by $-\Delta - \rho_{\infty}$. In particular we have
  \begin{equation}\label{claim002}
  -\Delta u_2^{n,+} \geq \rho_{\infty} \text{ in } \{ u_2^{n,+} >0\}
  \end{equation} 
  in the integral sense. 
As a result, $(u_1, S(t))$ and $u_2^{n,+}$ are respectively viscosity sub- and supersolutions of 
\begin{align} \tag*{($\tilde{P}$)$_\infty$} \label{Pinfty2}
\left\{\begin{array}{lll}
-\Delta p(\cdot,t) = \rho_{\infty} &\hbox{ in }& \{p>0\};\\
V= -\nu\cdot( Dp + D\Psi) &\hbox{ on }&\partial\{p>0\};\\
 \Psi= \mathcal{N}*\rho_{\infty}.
 \end{array}\right.
\end{align}

Using this fact, one can modify the proof of comparison principle for \ref{Pinfty} to show that, for any $n\in\mathbb{N}$,
   \begin{equation}\label{last}
 u_1 \prec u_2^{n,+} \hbox{ with } \overline{S_1(t)} \subseteq \{u_2^{n,+}(\cdot,t)>0\}.
  \end{equation}
  
 The proof of \eqref{last} is parallel to that of Theorem 2.7 in \cite{AKY} with the only difference lies showing the second inequality in the interior operator which we discuss below. Let us give a heuristic sketch of the proof under the assumption that $S(t)$ and $\{u_2^{n,+}(\cdot,t)>0\}$ have smooth boundaries: the actual proof is carried out with regularizations as given in \eqref{convolution1}-\eqref{convolution2} which generate strict subsolution and supersolution of \ref{Pinfty2}.  As usual in the proof of comparison principle, we begin with the scenario that $u_1$ crosses $u_2^{n,+}$ from below at some time and yield a contradiction. More precisely we suppose that the first crossing time is finite, i.e.
 $$
 t_0: = \sup \{ t : u_1(\cdot,s) \prec u_2^{n,+}(\cdot, s) \hbox{ and } S(t) \subseteq \{u_2^{n,+}(\cdot,t)>0\} \text{ for } s \leq t \}<\infty .
 $$ 
 
 Note that $t_0>0$ since $S(0)=\overline{\Omega_0}=\overline{\{u_1(\cdot,0)>0\}}$ due to Lemma~\ref{u1u2initialdata} and
 $u_1 \prec u_2^{n,+}$ at $t=0$ from the construction.  Observe also that \eqref{claim001}-\eqref{claim002} rules out the possibility that the crossing  occurs at an interior point, i.e.,
$$
u_1(\cdot,t) < u_2^{n,+}(\cdot,t) \hbox{ in } \{u_1(\cdot,t)>0\} 
$$
as long as $\{u_1(\cdot,t)>0\} \subseteq S(t) \subseteq \{u_2^{n,+}>0\}.$

Hence this means that the set $S(t)$ touches the boundary of $\overline{\{u_2^{n,+}(\cdot,t)>0\}}$ for the first time at some point $(x_0,t_0)$. Then the normal velocity law for the sets $S(t)$ and $\{u_2^{n,+}(\cdot,t)>0\}$, as well as the fact that $u_1(\cdot,t_0) \leq u_2^{n,+}(\cdot,t_0)$ yields a contradiction.

Note that \eqref{claim002} and the definition of $\Omega_n$ ensures that the source term for $u_1$ remains smaller than that of $u_2^{n,+}$ after the regularization process given in \eqref{convolution1}-\eqref{convolution2} if $r(t)$ is sufficiently small. Based on this fact, the rest of the proof is the same as to that of Theorem 2.7 in \cite{AKY}.

\medskip
3. Next we will show \eqref{eq:contraction_sup}. First, note that $\rho_m^{n,-}$ satisfies \eqref{pme_source} with no source term, while $\rho_m^{n,+}$ satisfies \eqref{pme_source} with source term $\rho_m^{n,+} f_n$ which is non-negative. Since their initial data is also ordered, comparison principle for \eqref{pme_source} yields that $\rho_m^{n,-} \leq \rho_m^{n,+}$. We then define 
\[
A_m^n(t) := \int (\rho_m^{n,+}(x,t) - \rho_m^{n,-}(x,t))dx,
\]
which is nonnegative.
We can apply the $L^1$ contraction  property of \eqref{pme_source} in Lemma~\ref{PMEDwellposed} to conclude that, for any $t>0$ and any $m> 1$,
\begin{equation}
\label{eq:temp_contra}
A_m^n(t) \leq \int_0^t\int \rho_m^{n,+}(x,s) f_n(x,s) dx ds + A_m^n(0).
\end{equation}
By \eqref{last0} and the definition of $S_1^{n,-}$, for all $t\geq 0$ and sufficiently large $m$, we have
\begin{equation}
\label{eq:tempasdf}
\supp{\rho_m^{n,-}(t)} = \supp{p_m^{n,-}(t)} \subseteq \supp u_2(t) =: \Omega(t).
\end{equation}
Thus for all sufficiently large $m$, the spatial integral in \eqref{eq:temp_contra} can be controlled as
\[
\int \rho_m^{n,+}(x,s) f_n(x,s) dx  =  \int_{\Omega_n(s) \setminus \Omega(s)} \rho_m^{n,+}(x,s)dx \leq A_m^n(s),
\]
where in the last step we used that $\rho_m^{n,-}(x,s) \equiv 0$ in $\Omega_n(s) \setminus \Omega(s)$ for all large $m$, which follows from \eqref{eq:tempasdf}. Plugging this into \eqref{eq:temp_contra} yields
$
A_m^n(t) \leq \int_0^t A_m^n(s)ds + A_m^n(0),
$
thus Gronwall's inequality immediately yields that $A_m^n(t) \leq A_m^n(0) e^t$. It is easy to check that $\lim_{m\to\infty} A_m^n(0) = |\Omega_0^{1/n} \setminus \Omega_0^{-1/n}| \leq C/n$, which yields $\liminf_{m\to\infty} A_m^n(t) \leq Ce^t/n$ for all $n\in \mathbb{N}, t\geq 0$. 

Next we claim  
\[\liminf_{m\to\infty} A_m^n(t) \geq |\supp u_2^{n,+}(\cdot, t) \setminus \supp u_1^{n,-}(\cdot, t)|.
\]
To show this, it suffices to show that 
\begin{equation}
\label{2ineqs}
\liminf_{m\to\infty} \int \rho_m^{n,+}(x,t) dx \geq |\supp u_2^{n,+}(\cdot,t)| ~~~\text{ and }~~~ \limsup_{m\to\infty} \int \rho_m^{n,-}(x,t) dx \leq |\supp u_1^{n,-}(\cdot,t)|.
\end{equation}
For the first inequality, note that by definition of the half-relaxed limit, for any $x \in \supp u_2^{n,+}(\cdot,t)$, we have $\liminf_{m\to\infty} p_m^{n,+}(x,t)> 0$, thus $\liminf_{m\to\infty} \rho_m^{n,+}(x,t)\geq 1$. Therefore $$\int\liminf_{m\to\infty} \rho_m^{n,+}(x,t)dx \geq |\supp u_2^{n,+}(\cdot,t)|,$$ and applying Fatou's lemma to it yields the first inequality of \eqref{2ineqs}. The second inequality follows from the definition of the half-relaxed limit $u_1^{n,-}$ and the fact that $\limsup_{m\to\infty} \| \rho_m^{n,-}\|_\infty \leq 1$, which is due to \eqref{eq:bound_rho_m}.

We now combine the above claim with $\liminf_{m\to\infty} A_m^n(t) \leq Ce^t/n$ to conclude that the $A_n(t)$ defined in \eqref{eq:contraction_sup} satisfies $A_n(t) \leq Ce^t/n$ for all $n\in \mathbb{N}, t\geq 0$, which yields \eqref{eq:contraction_sup}. Applying this to \eqref{eq:comp3} then yields that $\chi_{\{u_1>0\}} = \chi_{\{u_2>0\}} = \chi_{\{U>0\}}$ almost everywhere. So the proof of part (a) would be finished if we can show $\rho_\infty$ is also equal to these functions almost everywhere, which we postpone to step 4.

At the end of step 3, let us point out part (b) can be easily proved using the above bound on $A_n(t)$: note that $\Omega(t)=\{u_2(\cdot, t)>0\}$ is open due to the lower-semicontinuity of $u_2$, hence 
\[
\partial \Omega(t) = \overline{ \Omega(t)} \setminus \Omega(t) \subseteq \supp u_2^{n,+}(\cdot, t) \setminus \supp u_1^{n,-}(\cdot, t),
\]
where we used \eqref{eq:comp3}, \eqref{last0} and \eqref{last}  in the last inequality.  The above bound on $A_n(t)$ thus gives $|\partial \Omega(t)| \leq A_n(t) \leq Ce^t/n$ for all $n\in \mathbb{N}, t\geq 0$, and by sending $n\to\infty$ we obtain part (b) for $\Omega(t)$. In addition, the inequalities \eqref{eq:comp3}, \eqref{last0} and \eqref{last} also lead to $|\partial \Omega^1(t)| \leq A_n(t)$, hence we also have $|\partial \Omega^1(t)| =0$.

\medskip
4. To finish the proof of part (a), it suffices to relate $u_2$ and $\rho_\infty$ and show that 
\begin{equation}\label{goal_step1}
\rho_\infty = \chi_{\{u_2>0\}}  \text{ a.e.}
\end{equation}

The direction $\rho_\infty \geq \chi_{\{u_2>0\}}$ is easier: take any $(x_0,t_0)$ such that $a:= u_2(x_0,t_0)>0$. By definition of the half-relaxed limit $u_2$, there exists some positive $r_0$ and $N_0$, such that $p_m(x,t_0) \geq a/2>0$ for all $x \in B_{r_0}(x_0)$, $m>N_0$. Hence $\rho_m(x,t_0) \geq (\frac{m-1}{m} a)^{1/(m-1)}$  for $x,m$ as above. Combining this lower bound (which approaches 1 as $m\to\infty$) with the weak convergence of $\rho_m$ towards $\rho_\infty$ in Lemma \ref{lem:weak_conv}, we have $\rho_\infty(\cdot, t_0) \geq 1$ in $B_{r_0}(x_0)$. Since $x_0\in \{u_2(\cdot, t_0)>0\}$ is arbitrary, we have $\rho_\infty \geq \chi_{\{u_2>0\}}$.

For the other direction, we will use $p_{m}^{n,-}$. Using the definition of $S^{n,-}(t)$ (see \eqref{eq:st_def}) as well as \eqref{eq:comp3}, we have
\begin{equation}
\label{ineq_temp1}
 \limsup_{m\to\infty} \chi_{\{p_m^{n,-}(\cdot, t)>0\}} = \chi_{S^{n,-}(t)} \leq \chi_{\{u_2(\cdot, t)>0\}} \quad\text{ for any }n\in \mathbb{N}, t\geq 0.
\end{equation}

In addition, since $\limsup_{m\to\infty}\|\rho_m^{n,-}(t)\|_\infty \leq \limsup_{m\to\infty}\|\rho_m(t)\|_\infty \leq 1$ by \eqref{eq:bound_rho_m}, it implies
\begin{equation}
\label{ineq_temp2}
\limsup_{m\to\infty} \rho_m^{n,-}(\cdot,t) \leq \limsup_{m\to\infty} \chi_{\{\rho_m^{n,-}(\cdot, t)>0\}}=  \limsup_{m\to\infty} \chi_{\{p_m^{n,-}(\cdot, t)>0\}}.
\end{equation}

Finally we will relate $\rho_\infty$ with $\rho_m^{n,-}$. Note that for any continuous, bounded $f\geq 0$, we have the following (where we omit the $x$ dependence in integrals due to space limitation):
\begin{equation}
\label{ineq_temp3}
\int \rho_\infty(t) f dx = \lim_{m\to\infty}\int \rho_m(t) f dx= \lim_{m\to\infty}\lim_{n\to\infty}\int \rho_m^{n,-}(t) f dx \leq \int \limsup_{m,n\to\infty} \rho_m^{n,-}(t) f dx,
\end{equation}
where the first equality follows from Lemma \ref{lem:weak_conv}, the second one is due to the $L^1$ contraction property of Lemma \ref{PMEDwellposed}\ref{PMEDL1contraction} and the fact that $\rho_m^{n,-}\leq \rho_m$, and the last inequality follows from Fatou's lemma. This implies that \\$\rho_\infty(\cdot, t) \leq \limsup_{m, n\to\infty} \rho_m^{n,-}(\cdot,t)$ a.e., and this with \eqref{ineq_temp1}-\eqref{ineq_temp2} implies $\rho_\infty \leq  \chi_{\{u_2>0\}}$, which finishes the proof of \eqref{goal_step1}, thus yielding part (a).

\medskip

5.  To prove part (c), take any compact set $Q \subseteq \{(x,t): x\in \Omega(t)\}$. By definition of the half-relaxed limit $u_2$, for each $(x_0,t_0)\in Q$ there is some $r_0>0$, such that $p_m \geq u(x_0,t_0)/2$ in $B_{r_0}(x_0,t_0)$ for all sufficiently large $m$, which yields that $\rho_m$ uniformly approaches 1 in $B_{r_0}(x_0,t_0)$. The compactness of $Q$ then allows us to find a finite number of points $(x_i, t_i)$ such that $B_{r_i}(x_i,t_i)$ covers $Q$, implying that $\rho_m\to 1$ uniformly in $Q$.

Finally let us prove the $L^1$ convergence result, where we use the elementary inequality
\begin{equation}
\label{L1diff}
\|f-g\|_{L^1} = \int (g-f)_+ dx + \int (f-g)_+ dx \leq 2 \int (g-f)_+ dx + \left|\int (f-g) dx\right|.
\end{equation}
Let $f=\rho_m(\cdot, t)$, $g=\chi_{\Omega(t)} = \rho_\infty(\cdot, t)$. Since the mass of $\rho_m(\cdot, t)$ and $\rho_\infty(\cdot, t)$ are both preserved in time, we have 
\[
\left|\int (f-g) dx\right| \leq \|\rho_m(\cdot, 0) - \chi_{\Omega_0}\|_{L^1} = \big\|(\frac{m}{m-1}p_0)^{1/(m-1)} - \chi_{\Omega_0}\big\|_{L^1}\to 0 \text{ as }m\to\infty.
\]
To control $\int (g-f)_+ dx$, note that $g= 0$ a.e. in $\Omega(t)^c$, hence
$
\int (g-f)_+ dx = \int_{\Omega(t)} (1-\rho_m)_+dx.
$
Since $\Omega(t)$ is open, for any $\epsilon>0$ we can find a compact set $D\subseteq \Omega(t)$, such that $|\Omega(t)\setminus D|\leq \epsilon$. We can then apply the uniform convergence result of $\rho_m$ in $D$ to conclude that $\int (g-f)_+ dx  \leq 2\epsilon$ for sufficiently large $m$, and since $\epsilon>0$ is arbitrary we have  $\lim_{m\to\infty} \int (g-f)_+ dx = 0$. Plugging the above results into \eqref{L1diff} yields the $L^1$ convergence result.


%
%
%

\end{proof}

\section{Long time behavior of patch solutions in two dimensions} \label{long time section}

In this section, we investigate the long-time behavior of a patch solution $\rho_{\infty}$ in two dimensions, using the pressure variable characterization of the dynamics of $\rho_{\infty}$ obtained in section \ref{convviscsolsec}.
Throughout this section, we consider our spatial domain to be $\mathbb{R}^2$. By Theorem~\ref{convergence}, we know that $\rho_\infty(\cdot, t) = \chi_{\Omega(t)}$ for some $\Omega(t)\subseteq \mathbb{R}^2$ for all $t\geq 0$. Our goal is to show that, as $t\to\infty$, $\Omega(t)$ converges to the unique disk $B_0$ with the same mass and the center of mass as $\Omega_0$. (See Theorem~\ref{long_time})

We proceed as follows: in sections \ref{evolution of second moment} and \ref{some rearrangement inequalities}, we show that the second moment of $\rho_{\infty}(\cdot,t) = \chi_{\Omega(t)}$ decreases unless $\Omega(t)$ is a disk, from which we are able to conclude that $\Omega(t)$ cannot stay uniformly away from a disk for all times,  in terms of its {\it Fraenkel asymmetry}. In section \ref{convergence of energy}, we combine this with the gradient flow structure of $\rho_{\infty}$ to show that as $t \to +\infty$ the energy $E_\infty(\rho_\infty(t))$ approaches the minimum of $E_\infty$, with a quantitative estimate on the rate. Lastly, in section \ref{Convergence of rhoinfty}, we show that $\rho_\infty(\cdot, t)$  converges to $\chi_{B_0}$ strongly in $L^q$ for any $1\leq q<\infty$.

\subsection{Evolution of the second moment} \label{evolution of second moment}
Let $M_2[f] := \int_{\mathbb{R}^2}  f(x) |x|^2 dx$ denote the second moment of $f$. In this subsection, we investigate the evolution of the second moment of $\rho_\infty(\cdot, t) = \chi_{\Omega(t)}$.
Before we present the rigorous derivation of the evolution of the second moment, we begin with the following  heuristic computation. As described in the introduction, $\rho_\infty(\cdot, t)$ formally satisfies the transport equation
\[
\rho_t = \nabla \cdot (\rho  ( \nabla \bN \rho+ \nabla p)) ,
\]
where $p$ is a solution to \ref{P}. (See equation (\ref{transport111}).) By definition, $p(\cdot, t)$ solves $\Delta p = -1$ in $\Omega(t)$ and $p=0$ on $\partial \Omega(t)$. Hence, supposing that $\partial \Omega(t)$ is smooth, the evolution of $M_2[\rho_\infty(t)]$ is given by
\begin{equation}
\label{eq:dM2_dt}
\begin{split}
\frac{d}{dt} M_2[\rho_\infty(t)] 
&=- 2\int_{\mathbb{R}^2}  \rho_\infty \nabla \bN \rho_\infty  \cdot  x dx - 2\int_{\mathbb{R}^2}  \rho_\infty \nabla p \cdot  x dx\\
&=-\frac{1}{\pi}  \int_{\mathbb{R}^2} \int_{\mathbb{R}^2}   \rho_\infty(x) \rho_\infty(y)  \frac{( x  - y)\cdot  x}{|x-y|^2} dy dx - 2\int_{\Omega(t)}   \nabla p \cdot  x dx\\
&= -\frac{1}{2\pi} \int_{\mathbb{R}^2} \int_{\mathbb{R}^2}  \rho_\infty(x) \rho_\infty(y)  dy dx + 4 \int_{\Omega(t)}   p(x) dx\\
&=  -\frac{1}{2\pi} |\Omega(t)|^2  + 4 \int_{\Omega(t)}   p(x) dx=  -\frac{1}{2\pi} |\Omega_0|^2  + 4 \int_{\Omega(t)}   p(x) dx.\end{split}
\end{equation}
In the second equality, we use that, in two dimensions, $\bN \rho_\infty  = \mathcal{N} * \rho_\infty$ with $\mathcal{N}(x) = \frac{1}{2\pi} \log |x|$.
In the third equality, we symmetrize $x$ and $y$ in the first integral (hence the extra factor of $\frac{1}{2}$), and in the last equality, we use that $\rho_\infty$ preserves its mass (which is $|\Omega_0|$) for all time. 

In the following proposition, we rigorously obtain the time evolution of $M_2[\rho_\infty(t)]$ by analyzing the evolution of the second moments for each $\rho_m$ and sending $m\to +\infty$, using our convergence results from the previous sections. We show that the evolution of the second moment indeed satisfies a time-integral form of \eqref{eq:dM2_dt}, with the exception that we must substitute $p(x)$ with $u_1(x)$, the half-relaxed limit of $p_m$ defined in Lemma \ref{sub}, to take into account the fact that $\Omega(t)$ may not have smooth boundary for all time.
 
\begin{proposition} 
\label{prop:dm2_dt}
Let $\Omega_0 \subseteq \mathbb{R}^2$ be a bounded domain with Lipschitz boundary, and let $\rho_\infty(\cdot, t) = \chi_{\Omega(t)}$ be the gradient flow of $E_\infty$ with initial data $\rho_0 = \chi_{\Omega_0}$. 
Then for any $T>0$,
\begin{equation}
\label{eq:m2}
M_2[\rho_\infty(T)] - M_2[\rho_0] \leq    - \frac{1}{2\pi} |\Omega_0|^2 T + 4\int_0^T \int_{\Omega(t)} u_1(x,t) dx dt,
\end{equation}
where $u_1$ is the half-relaxed limit of $p_m$, defined in Lemma \ref{sub}, and $\Omega(t) = \{ u_2(\cdot, t) >0 \}$, as defined in Theorem \ref{convergence} \ref{convergence part b}.
\end{proposition}

\begin{proof}
For any $m>1$, let $\rho_m$ be the weak solution of \ref{pme} with initial data $(\frac{m-1}{m} p_0)^{1/(m-1)}$, where $p_0$ is given by equation (\ref{initial}). Let $p_m:= \frac{m}{m-1} \rho_m^{m-1}$ be the corresponding solution of \ref{Pm}. Taking $|x|^2$ as our test function, we have for any $T>0$,
\begin{align}
\label{eq:m2_m}
&\underbrace{\int_{\mathbb{R}^2} \rho_m(x,T) |x|^2 dx}_{=:I_1} - \underbrace{\int_{\mathbb{R}^2} \rho_m(x,0) |x|^2 dx}_{=:I_2} \\
&\quad  = - 2\underbrace{\int_0^T \int_{\mathbb{R}^2} \rho_m \nabla \Phi_{1/m}(x,t) \cdot x dxdt}_{=:I_3} + 4\underbrace{\int_0^T \int_{\mathbb{R}^2} \rho_m^m(x,t) dxdt}_{=:I_4} . \nonumber
\end{align}
(Since $\rho_m$ has compact support in $[0,T]$, our test function is not required to have compact support since we can always take a cut-off sufficiently far away.)
As $m\to +\infty$,  Lemma \ref{lem:weak_conv} yields that $I_1$ converges to $M_2[\rho_\infty(T)]$ and $I_2$ converges to $M_2[\rho_\infty(0)]$.

To show the convergence of $I_3$, we decompose the integral into two parts:
\[
I_3 = \int_0^T \int_{\mathbb{R}^2} \rho_m \nabla \Phi(x,t) \cdot x dxdt +  \int_0^T \int_{\mathbb{R}^2} \rho_m \nabla (\Phi_{1/m}(x,t) - \Phi(x,t)) \cdot x dxdt =: I_{31} + I_{32}.
\]
Since $\nabla \Phi(x,t) \cdot x \in C(\mathbb{R}^d)$ for any $t$, Lemma \ref{lem:weak_conv} again gives that 
\[ I_{31} \xrightarrow{m \to +\infty} \int_0^T \int_{\mathbb{R}^2} \rho_\infty \nabla \Phi(x,t) \cdot x dxdt = \frac{1}{4\pi} \int_0^T \left(\int_{\mathbb{R}^2} \rho_\infty dx \right)^2dt = \frac{1}{4\pi} |\Omega_0|^2 T, \]
 where the last two equalities are obtained by symmetrizing $x$ and $y$ in the integrand and using conservation of mass, as in equation (\ref{eq:dM2_dt}).

To control $I_{32}$, we first bound $\|\nabla \Phi_{1/m} - \nabla \Phi\|_{L^2(\mathbb{R}^2)}$. By Proposition \ref{N mu bounds}, 
\begin{equation}
\label{eq:l2diff}
\|\nabla \Phi_{1/m} - \nabla\Phi\|_{L^2(\mathbb{R}^2)}\leq  W_2(\rho_\infty* \psi_{1/m}, \rho_\infty) \leq \frac{1}{m} \int \psi(x) |x|^2 dx ,
\end{equation}
  where, in the last step, we apply \cite[Lemma 7.1.10]{AGS}.
Hence
\[
|I_{32}| \leq \int_0^T \|\rho_m(\cdot, t) |x| \|_{L^2} \|\nabla \Phi_{1/m} - \nabla\Phi\|_{L^2(\mathbb{R}^2)}dt \to 0 \text{ as }m\to +\infty,
\]
where the fact that $\sup_{t\in[0,T]}\sup_{m\geq 1} \|\rho_m(\cdot, t) |x| \|_{L^2} < +\infty$ is a consequence of Lemma \ref{compact_spt}, which ensures $\rho_m$ is uniformly bounded and compactly supported. Combining the estimates on $I_{31}$ and $I_{32}$ yields that  $I_3 \to \frac{1}{4\pi} |\Omega_0|^2 T$ as $m\to +\infty$.

Finally, we consider $I_4$. We will show that
\begin{equation}
\label{eq:t4}
\limsup_{m\to\infty} \int_0^T \int_{\mathbb{R}^2} \rho_m^m(x,t) dxdt \leq \int_0^T \int_{\Omega(t)} u_1(x,t) dx dt.
\end{equation}
The proof is then finished by taking $\limsup_{m \to +\infty}$ on both sides of \eqref{eq:m2_m}.

To show \eqref{eq:t4}, first note that, since $p_m := \frac{m}{m-1}\rho_m^{m-1}$, we may write $\rho_m^m =\frac{m-1}{m}\rho_m p_m$ and apply Remark \ref{remark:bound_rho_m} to obtain
\begin{equation}
\label{eq:t4_1}
\limsup_{m\to\infty} \int_0^T \int_{\mathbb{R}^2} \rho_m^m(x,t) dxdt \leq  \limsup_{m\to\infty}\int_0^T\int_{\mathbb{R}^2} p_m(x,t) dxdt.
\end{equation}
It remains to show that 
\begin{equation}
\label{eq:t4_2}
 \limsup_{m\to\infty}\int_0^T \int_{\mathbb{R}^2} p_m(x,t) dxdt \leq \int_0^T \int_{\mathbb{R}^2}u_1(x,t) dxdt.
\end{equation}
For any $n\in \mathbb{N}$, define
\[ u_{1,n}(x,t) := \displaystyle\sup_{\substack{m>n \\ |(x,t) - (y,s)| < 1/n}} {p_m}(y,s) . \]
Note that $u_{1,n}$ is decreasing in $n$ and $\displaystyle\lim_{n\to\infty} u_{1,n} = u_1$ by definition of $u_1$. For each $n\in \mathbb{N}$, we have $p_m \leq u_{1,n}$ for all $m>n$, hence 
\[
\limsup_{m\to\infty}\int_0^T \int_{\mathbb{R}^2} p_m(x,t) dxdt \leq \int_0^T \int_{\mathbb{R}^2} u_{1,n}(x,t) dxdt \quad\text{ for all }n\in \mathbb{N}.
\]
We can then take $n\to +\infty$ in the above inequality and apply the monotone convergence theorem. By Theroem \ref{convergence} \ref{convergence part a}, $\Omega(t) = \{ u_1(\cdot, t) >0 \}$ almost everywhere. Thus, inequality \eqref{eq:t4_2} holds.
\end{proof}

\subsection{Some rearrangement inequalities} \label{some rearrangement inequalities}
In this subsection, we digress a bit to obtain an upper bound for the quantity
\begin{equation}
\label{def:F}
F(\Omega) = -\frac{1}{2\pi}|\Omega|^2 + 4 \int_\Omega p(x) dx,
\end{equation}
where  $\Omega$ is a bounded set with smooth boundary and $p:\bar \Omega\to\mathbb{R}$ satisfies $-\Delta p = 1$ in $\Omega$ and $p=0$ on $\partial \Omega$. This quantity appears in our heuristic computation for the evolution of the second moment of $\rho_\infty(t)$, where we show $\frac{d}{dt} M_2[\rho_\infty] \leq F(\Omega(t))$. Likewise, $\int_0^T F(\Omega(t))dt$ would have appeared on the right hand side of our rigorous result, given in equation \eqref{eq:m2}, if the boundary of $\Omega(t)$ were smooth for all time. While in this subsection we only aim to control $F(\Omega)$ for smooth domains, in the next subsection we discuss how to use this bound to control the right hand side of \eqref{eq:m2}, even when the boundary of $\Omega(t)$ is not smooth.

The following result, due to Talenti \cite{Talenti}, shows that $F(\Omega)\leq 0$, with equality if and only if $\Omega$ is a disk. We sketch the proof below for the sake of completeness. In the subsequent proposition, we will modify the proof to get a stronger inequality.

\begin{proposition}[{c.f. \cite[Theorem 1]{Talenti}}]
\label{prop:omega}
Let $\Omega \subseteq \mathbb{R}^2$ be a bounded domain with smooth boundary, and let $F(\Omega)$ be as  in \eqref{def:F}. Then we have
\begin{equation}
\label{ineq:f}
F(\Omega) \leq 0,
\end{equation}
and the equality is achieved if and only if $\Omega$ is a disk.
 \end{proposition}

\begin{proof}
First, note that maximum principle yields that $p \geq 0$ in $\bar \Omega$ and $p>0$ in $\Omega$. For any $k \in [0, \sup_{\Omega} p)$, let us define
$$
\Omega_k := \{x\in \Omega: p(x)>k\} \quad \hbox{ and } g(k) := |\Omega_k|.
$$ Note that $g(0) = |\Omega|$.
By definition of $p$ and the divergence theorem, we have
\begin{equation}
\label{eq:temp1}
g(k) = \int_{\Omega_k} -\Delta p(x) dx = \int_{\partial\Omega_k} -n\cdot \nabla p \,d\sigma=  \int_{\partial\Omega_k} |\nabla p| d\sigma .
\end{equation}
On the other hand, by the co-area formula (c.f. \cite{Federer}),
\begin{equation}
\label{eq:temp2}
g(k) = \int_k^{\infty} \int_{\partial\Omega_s} \frac{1}{|\nabla p|} d\sigma ds, \quad 
g'(k) = -\int_{\partial\Omega_k} \frac{1}{|\nabla p|} d\sigma.
\end{equation}
Combining \eqref{eq:temp1} and \eqref{eq:temp2} and applying the Cauchy-Schwarz inequality,
\begin{equation}
\label{eq:temp3}
\begin{split}
g(k) g'(k) &= \left(\int_{\partial\Omega_k} |\nabla p| d\sigma\right) \left( - \int_{\partial\Omega_k} \frac{1}{|\nabla p|} d\sigma\right) \leq -P(\Omega_k)^2,
\end{split}
\end{equation}
where $P(\Omega_k)$ is the perimeter of $\Omega_k$.
For any bounded domain $E\subseteq \mathbb{R}^2$,  the isoperimetric inequality yields
\begin{equation}
\label{eq:isoperimetric}
 2\sqrt{\pi} \sqrt{|E|} \leq P( E) .
\end{equation}
Applying inequality \eqref{eq:isoperimetric} to $\Omega_k$ in \eqref{eq:temp3} gives
\[
g(k) g'(k) \leq -\left(2\sqrt{\pi} \sqrt{g(k)}\right)^2 = -4\pi g(k),
\]
hence $g(k)$ satisfies the differential inequality 
\begin{align} \label{g prime bound}
g'(k) \leq -4\pi \text{ for all } k \in \left(0,\sup_\Omega g \right) .
\end{align}
 Combining this with $g(0)=|\Omega|$ yields that $g(k) \leq (|\Omega| - 4\pi k)_+$ for all $k\geq 0$. Therefore,
\[
\int_{\Omega}   p(x) dx = \int_0^{\sup_\Omega p} g(k) dk \leq \int_0^{\infty} \left( |\Omega| - 4\pi k\right)_+dk = \frac{1}{8\pi} |\Omega|^2,
\]
which gives \eqref{ineq:f}.
In order to achieve equality, $\Omega_k$ must be a disk for almost every $k>0$, hence $\Omega$ must be a disk.
\end{proof}

We now prove a stronger version of the above inequality by replacing the isoperimetric inequality in the above argument (see \eqref{eq:isoperimetric}) by the following quantitative version due to Fusco, Maggi, and Pratelli  \cite{fmp}. 

\begin{lemma}[{c.f. \cite[Section 1.2]{fmp}}]
\label{lem:fmp}
Let $E \subseteq \mathbb{R}^2$ be a bounded domain. We define the \emph{Fraenkel asymmetry} $A(E) \in [0,1]$ as
\begin{equation*}
A(E) := \inf\left\{\frac{|E\triangle (x_0+rB)|}{|E|}: x_0\in\mathbb{R}^2, \pi r^2 = |E| \right\},
\end{equation*}
where $B$ is the unit disk. Then there is some constant $c\in(0,1)$, depending only on the dimension, such that
\begin{equation*}
P(E) \geq 2\sqrt{\pi} \sqrt{|E|} \left(1+cA(E)^2\right),
\end{equation*}
where $P(E) = \mathcal{H}^1(\partial E)$ denotes the perimeter of $E$.

 \end{lemma}

We begin with the following simple observation regarding the Fraenkel asymmetry.

\begin{lemma}
\label{lem:asym}
Let $E \subseteq \mathbb{R}^2$ be a bounded domain. For all $U\subseteq E$ satisfying $|U| \geq |E| (1-\frac{A(E)}{4})$, we have
\[
A(U) \geq \frac{A(E)}{4}.
\]
\end{lemma}
\begin{proof}
Assume the statement is not true, so there exists some disk $B_U$ with the same area as $U$ so that 
\[
\frac{|U\triangle B_U|}{|U|} < \frac{A(E)}{4}.
\]
Since $|U| = |B_U|$, we have $|U\triangle B_U| = 2(|U| - |U \cap B_U|)$. Hence the above inequality becomes 
\begin{equation*}
|U \cap B_U| > |U| \left(1-\frac{A(E)}{8}\right).
\end{equation*}
Let $B_E$ be a disk with the same area as $E$ that contains $B_U$. Then since $|U| \geq |E| (1-\frac{A(E)}{4})$,
\[
\begin{split}
|E\cap B_E| &\geq |U\cap B_U| > |E| \left(1-\frac{A(E)}{4}\right)\left(1-\frac{A(E)}{8}\right) \geq |E| \left(1-\frac{3A(E)}{8}\right)
\end{split}
\]
Therefore,
\[
\frac{|E\triangle B_E|}{|E|} = \frac{2(|E| - |E\cap B_E|)}{|E|} < \frac{3}{4} A(E),
\]
which contradicts the fact that $A(E)\leq |E\triangle B_E|/|E|$. This gives the result.
\end{proof}

With this lemma, we are now able to conclude a stronger upper bound on $F(\Omega)$ than provided by Proposition \ref{prop:omega}.
\begin{proposition}
\label{prop:refined} Under the same assumptions as Proposition \ref{prop:omega},  there exists a constant $c_0 \in (0,1)$,  such that
\begin{equation*}
F(\Omega) \leq-  c_0 A(\Omega)^3 |\Omega|^2.
\end{equation*}
 \end{proposition}

\begin{proof}
We follow the proof of Proposition \ref{prop:omega}, with the following difference: instead of applying the isoperimetric inequality \eqref{eq:isoperimetric} to the set $\Omega_k$ in inequality \eqref{eq:temp3}, we now apply the quantitative version from Lemma \ref{lem:fmp} to obtain
\[
g'(k) \leq -4\pi \left(1+cA(\Omega_k)^2\right)^2 \leq -4\pi \left(1+cA(\Omega_k)^2\right).
\]
To relate $A(\Omega_k)$ with $A(\Omega)$, we apply Lemma \ref{lem:asym} to obtain that $
A(\Omega_k) \geq \frac{A(\Omega)}{4}$ for any $k$ such that $g(k) \geq |\Omega| (1-\frac{A(\Omega)}{4})$.
In other words, we have
\begin{equation}
\label{eq:g_temp}
g'(k) \leq -4\pi  \left(1+\frac{cA(\Omega)^2}{16}\right) \text{ for all $k$ such that }g(k) \geq |\Omega| \left(1-\frac{A(\Omega)}{4}\right).
\end{equation}

We claim that this implies 
\begin{equation}
\label{eq:g}
g(k) \leq |\Omega| -4\pi  \left(1+\frac{cA(\Omega)^2}{16}\right)k \quad \text{ for all }k \in \left(0,\frac{A(\Omega)|\Omega|}{32\pi}\right).
\end{equation}
To see this, note that for all $k \in (0,\frac{ A(\Omega) |\Omega|}{32\pi})$, the right hand side of \eqref{eq:g} is greater than $|\Omega|(1-\frac{A(\Omega)}{4})$ since $1+cA(\Omega)^2/16\leq 2$. As a result, if \eqref{eq:g} is violated at some $k_0 \in (0,\frac{A(\Omega)|\Omega|}{32\pi}) $, then we must have $g(k) \geq |\Omega|(1-\frac{A(\Omega)}{4})$ for all $k\in (0,k_0)$, since $g$ is a decreasing function. We can then integrate \eqref{eq:g_temp} in $(0,k_0)$ to conclude that \eqref{eq:g} actually holds at $k_0$, a contradiction.

Let $h(k) = (|\Omega|-4\pi k)_+$. Inequality \eqref{eq:g} implies that $g(k) \leq h(k) - \frac{cA(\Omega)^3|\Omega|}{128}$ at $k=\frac{A(\Omega) |\Omega|}{32\pi}$. For $k>\frac{A(\Omega) |\Omega|}{32\pi}$, recall that by inequality (\ref{g prime bound}), we have $g'(k) \leq -4\pi$ for $k \in \left(0,\sup_\Omega g \right) $, and by definition of $h$, we have $h'(k) = -4\pi$ in $(0,|\Omega|/(4\pi))$. This gives that $g(k) \leq h(k) - \frac{cA(\Omega)^3|\Omega|}{128}$ for  $A(\Omega)|\Omega|/32\pi\leq k \leq |\Omega|/4\pi$. Since $A(\Omega) \leq 1$ this range of $k$ is larger than $|\Omega|/8\pi$, and since $g(k) \leq h(k)$ for all $k$, we have
\[
\int_0^\infty g(k) dk \leq \int_0^\infty h(k) dk - \frac{cA(\Omega)^3 |\Omega|^2}{2000} = \frac{|\Omega|^2 }{8\pi} - \frac{cA(\Omega)^3 |\Omega|^2}{2000}.
\]
Finally, this gives 
\[
F(\Omega) = \int_0^\infty g(k) - \frac{|\Omega|^2}{8\pi} \leq \frac{cA(\Omega)^3 |\Omega|^2}{2000},
\]
hence the result holds with $c_0 := \frac{c}{2000}$.
\end{proof}

\subsection{Convergence of energy functional as $t\to\infty$} \label{convergence of energy}

In this section, we aim to show that, along the solution $\rho_\infty(\cdot, t)$, the energy functional $E_\infty$ converges to its global minimizer as $t \to +\infty$. We begin by estimating the rate of change of the second moment along $\rho_\infty$.
Combining Proposition \ref{prop:refined} with our heuristic computation \eqref{eq:dM2_dt} suggests that 
\[
\frac{d}{dt} M_2[\rho_\infty(t)]\leq -c_0 A(\Omega(t))^3 |\Omega_0|^2.
\]
We now show that this inequality is indeed true in the time-integral sense, even if $\Omega(t)$ does not have smooth boundary.

\begin{proposition}
\label{prop:m2_a}
Let $\Omega_0 \subseteq \mathbb{R}^2$ be a bounded domain with Lipschitz boundary, and let $\rho_\infty(\cdot, t) = \chi_{\Omega(t)}$ be the gradient flow of $E_\infty$ with initial data $\rho_0 = \chi_{\Omega_0}$.  Then we have

\begin{equation}
\label{eq:m2_2}
M_2[\rho_\infty(T)] - M_2[\rho_0] \leq    -c_0 |\Omega_0|^2 \int_0^T A(\Omega(t))^3 dx dt,
\end{equation}
where $c_0 \in (0,1)$ is the constant given in Proposition \ref{prop:refined}.
\end{proposition}

\begin{proof}
Since the evolution of the second moment is already given by Proposition \ref{prop:dm2_dt}, it remains to show
\begin{equation}
\label{eq:goal_u1}
- \frac{1}{2\pi} |\Omega_0|^2 T + 4\int_0^T \int_{\Omega(t)} u_1(x,t) dx dt \leq -c_0 |\Omega_0|^2 \int_0^T A(\Omega(t))^3 dt,
\end{equation}
where $u_1$ is the half-relaxed limit of $p_m$ defined in Lemma \ref{sub}. 

Let $\Omega^1(t) = \{u_1(\cdot,t)>0\}$. By Theorem \ref{convergence} \ref{convergence part a} and \ref{convergence part b}, we have $\Omega^1(t) = \Omega(t)$ almost everywhere, so $A(\Omega(t)) = A(\Omega^1(t))$, and $|\partial\Omega^1(t)|=0$ for all $t\in [0,T]$.  Hence for any $\epsilon>0$ and $t\in [0,T]$, we can find a set $D(t)\subseteq \mathbb{R}^2$ with smooth boundary such that $ \overline{\Omega^1(t)} \subseteq D(t)$, and $|D(t) \setminus \Omega^1(t)| \leq \epsilon$. For any $t\geq 0$, we then have a classical solution $p(\cdot, t)$ such that $-\Delta p(\cdot, t) = 1$ in $D(t)$, and $p(\cdot, t) = 0$ on $\partial D(t)$ and $D(t)^c$. In addition, we may choose $D(t)$ so that $\partial D(t)$ is continuous in time with respect of Hausdorff distance of sets, which ensures that $p$ is continuous in time.

We first aim to show that 
\begin{equation}\label{order}
u_1(x,t) \leq p(x,t).
\end{equation} It suffices to show that $u_1(x,t) \leq a p(x,t)$ for any $a>1$. Towards a contradiction, assume that there exists some $a>0$, such that $\sup_{x\in\mathbb{R}^2, t\in [0,T]} (u_1 - ap) > 0$. 
Since  $p$ is continuous in both space and time, and $u_1$ is upper semicontinuous by definition as the half-relaxed limit, $u_1-ap$ achieves a strictly positive maximum at some $(x_0,t_0)$. Furthermore, since $p\geq 0$, we have $u_1(x_0,t_0)>0$. Again using that $(u_1, \Sigma_1)$ is a subsolution of \ref{Pinfty}, we have that $-a\Delta p(x_0,t_0)\leq 1$, which implies that $-\Delta p(x_0,t_0) < 1$. However, since $x_0 \in \Omega^1(t_0) \subseteq D(t_0)$,
we must have $-\Delta p(x_0,t_0)=1$, which gives the contradiction. 

We now show inequality \eqref{eq:goal_u1}. Since  $|D(t) \setminus \Omega^1(t)| \leq \epsilon$,  there exists $C$ depending on $|\Omega_0|$ so that $A(\Omega^1(t)) = A(\Omega(t)) \leq A(D(t)) + C\epsilon$. Combining this observation with \eqref{order} and Proposition \ref{prop:refined}, we obtain the following bound for the left hand side of \eqref{eq:goal_u1}, where $C$ depends on $\Omega_0$ and $T$:
\[
\begin{split}
&\int_0^T \left(-\frac{1}{2\pi}|\Omega_0|^2 + \int_{\Omega(t)} u_1(x,t) dx\right)dt \leq \int_0^T   \left(-\frac{1}{2\pi}|D(t)|^2 + \int_{D(t)} p(x,t) dx\right)dt + C\epsilon\\
&\quad \leq -c_0 \int_0^T A(D(t))^3 |D(t)|^2 dt + C\epsilon \leq  -c_0 |\Omega_0|^2 \int_0^T A(\Omega(t))^3  dt + C\epsilon.
\end{split}
\]
Sending $\epsilon \to 0$ gives the result.
\end{proof}

\begin{corollary}
\label{cor:m2}
Under the assumptions of Proposition \ref{prop:m2_a},  for any $T>0$, there exists some $t_0 \in (0,T)$, such that 
\begin{equation}
\label{eq:A_small}
A(\Omega(t_0)) \leq C(\Omega_0) T^{-1/3},
\end{equation}
where $C(\Omega_0) :=\left (M_2[\chi_{\Omega_0}] /c_0|\Omega_0|^2 \right)^{1/3}$, for $c_0$ as in Proposition \ref{prop:refined}.
\end{corollary}

\begin{proof}
Fix  $T>0$. Towards a contradiction, assume $A(\Omega(t_0)) > C(\Omega_0) T^{-1/3}$ for all $t_0 \in (0,T)$. By Proposition \ref{prop:m2_a} and the definition of $C(\Omega_0)$,
\[
M_2[\rho_\infty(T)]  \leq M_2[\rho_0]   -c_0 |\Omega_0|^2 \int_0^T A(\Omega(t))^3 dx dt \leq M_2[\rho_0]   -c_0 |\Omega_0|^2 T (C(\Omega_0) T^{-1/3})^3 = 0,
\]
which contradicts with the fact that $M_2[\rho_\infty(t)]$ must be positive for all time.
\end{proof}

The above corollary does not directly yield that  $\lim_{t\to\infty} A(\Omega(t)) = 0$. 
To show this and conclude that $\Omega(t)$ converges to a disk, we will use the fact that the energy $E_\infty$ is decreasing in time along $\rho_\infty(\cdot, t)$. In the next lemma, we show that if $A(\Omega)$ is small, then the energy is close to its minimum.

\begin{lemma} \label{A energy estimate}
Let $\Omega \subseteq \mathbb{R}^2$ be a bounded domain, and let $B_\Omega \subseteq \mathbb{R}^2$ be a disk with $|B_\Omega| = |\Omega|$. Then,
\begin{equation*}
0\leq E_\infty(\chi_\Omega) - E_\infty(\chi_{B_\Omega}) \leq 40   |\Omega| (1+|\Omega|+M_2[\chi_\Omega]) \sqrt{A(\Omega)}
\end{equation*}
\end{lemma}

\begin{proof}
The first inequality is a direct consequence of Riesz rearrangement inequality \cite[Theorem 3.7]{LiebLoss}. To prove the second one, let us first rewrite $E_\infty(\chi_\Omega) - E_\infty(\chi_{B_\Omega})$ as
\[
\begin{split}
E_\infty(\chi_\Omega) - E_\infty(\chi_{B_\Omega}) &= \frac{1}{2\pi} \iint_{\mathbb{R}^2\times \mathbb{R}^2} (\chi_\Omega(x) - \chi_{B_\Omega}(x))(\chi_\Omega(y)+ \chi_{B_\Omega}(y)) \log|x-y| dx dy\\
&=: \frac{1}{2\pi} I_1+   \frac{1}{2\pi} I_2,
\end{split}
\]
where $I_1$ and $ I_2$ denote the integral in the domains $|x-y|\leq 1$ and $|x-y|> 1$, respectively. 

First, we consider $I_1$. Note that for any $x\in \mathbb{R}^2$, we have 
\[
\left| \int_{y\in B(x,1)} (\chi_\Omega(y)+\chi_{B_\Omega}(y)) \log|x-y| dy\right| \leq 2\left| \int_{|x-y|\leq 1}  \log|x-y| dy\right| = \pi,
\]
hence 
\begin{equation*}
I_1 \leq \|\chi_\Omega - \chi_{B_\Omega}\|_1  \left\| \int_{y\in B(x,1)} (\chi_\Omega(y)+\chi_{B_\Omega(y)}) \log|x-y| dy\right\|_\infty \leq \pi \|\chi_\Omega - \chi_{B_\Omega}\|_1.
\end{equation*}

Now, we consider $I_2$. For $|x-y|>1$, $\log |x-y| \leq |x-y| \leq |x|+|y| \leq (1+|x|)(1+|y|)$, so
\begin{equation*}
\begin{split}
I_2 &\leq \left(\int_{\mathbb{R}^2} |\chi_\Omega(x) - \chi_{B_\Omega}(x)|(1+|x|) dx\right) \left( \int_{\mathbb{R}^2} (\chi_\Omega(y) + \chi_{B_\Omega}(y))(1+|y|) dy \right)\\
&\leq \|\chi_\Omega - \chi_{\Omega_B}\|_1^{1/2} \left(\int_{\mathbb{R}^2}  |\chi_\Omega(x) - \chi_{B_\Omega}(x)|(1+|x|)^2 dx\right)^{1/2} \\  
&\quad\cdot (2|\Omega|)^{1/2} \left(\int_{\mathbb{R}^2}  |\chi_\Omega(y) + \chi_{B_\Omega}(y)|(1+|y|)^2 dy\right)^{1/2}\\
&\leq 2\sqrt{2}|\Omega|^{1/2} \|\chi_\Omega - \chi_{\Omega_B}\|_1^{1/2} (M_2[\chi_\Omega] + M_2[\chi_{B_\Omega}] + 2|\Omega|).
\end{split}
\end{equation*}
Combining the above estimates on $I_1$ and $I_2$ with the facts that $\|\chi_\Omega - \chi_{B_\Omega}\|_1 \leq 2|\Omega|$ and $1/ \pi \leq 1$, we have
\begin{equation}
\label{eq:Ediff}
E_\infty(\chi_\Omega) - E_\infty(\chi_{B_\Omega})\leq  \frac{1}{2\pi} I_1+   \frac{1}{2\pi} I_2 \leq |\Omega|^{1/2} \|\chi_\Omega - \chi_{B_\Omega}\|_1^{1/2} (1+M_2[\chi_\Omega] + M_2[\chi_{B_\Omega}]+2|\Omega|).
\end{equation}

The proof is then split into the following two cases: $A(\Omega) \geq 1/2$ and $A(\Omega) < 1/2$.

 \emph{Case 1: $A(\Omega) \geq 1/2$.} In this case, we have 
$ \|\chi_\Omega - \chi_{B_\Omega}\|_1 \leq 2|\Omega| \leq 4A(\Omega)|\Omega|$ for any disk $B_\Omega$ with the same measure as $\Omega$. Since $E_\infty$ is invariant under translations, we can simply choose $B_\Omega$ to be centered at 0. Such a choice directly yields $M_2[\chi_{B_\Omega}] \leq M_2[\chi_\Omega]$, hence  \eqref{eq:Ediff} becomes 
\[
E_\infty(\chi_\Omega) - E_\infty(\chi_{B_\Omega}) \leq|\Omega|^{1/2} (4A(\Omega)|\Omega|)^{1/2} (1+2|\Omega|+2M_2[\chi_\Omega] )\leq  4|\Omega|(1+|\Omega|+M_2[\chi_\Omega]) \sqrt{A(\Omega)},
\]
which gives the result.

 \emph{Case 2: $A(\Omega) < 1/2$.} In this case, we choose $B_\Omega$ to be the disk minimizing $|\Omega \triangle B_\Omega|$, which then gives 
\begin{equation}
\label{eq:l1diff}
\|\chi_\Omega - \chi_{B_\Omega}\|_1 = A(\Omega) |\Omega|.
\end{equation}
This choice of $B_\Omega$ no longer directly gives us $M_2[\chi_{B_\Omega}]\leq M_2[\chi_\Omega]$, but we claim that we still have $M_2[\chi_{B_\Omega}] \leq 36 M_2[\chi_\Omega]$.
To see this, first note that $A(\Omega)< 1/2$ implies $|B_\Omega \setminus \Omega| <  |B_\Omega|/2$. Also, a simple computation yields that for any $x,y \in B_\Omega$, we have $|x|^2 \leq (|y|+|x-y|)^2 \leq 2|y|^2 + 2|x-y|^2 \leq 2|y|^2 + 8|\Omega|/\pi$. Therefore,
\begin{align*}
& M_2[\chi_{B_\Omega}] = \int_{B_\Omega}|x|^2 dx \leq |B_\Omega| \max_{x\in B_\Omega} |x|^2 \leq 2|B_\Omega \cap \Omega|  \max_{x\in B_\Omega} |x|^2 \leq 2\int_{B_\Omega \cap \Omega}\left(2|y|^2 + \frac{8}{\pi}|\Omega|\right) dy \nonumber \\
& \leq 4M_2[\chi_\Omega] + \frac{16}{\pi} |\Omega|^2 \leq 4M_2[\chi_\Omega] + 32 \left( \frac{|\Omega|^2}{2 \pi} \right) \leq 4M_2[\chi_\Omega] + 32 \left( \int_0^{\sqrt{|\Omega|/\pi}} r^2 \cdot 2\pi r dr \right) \leq 36 M_2[\chi_\Omega], \label{eq:m2_b}
\end{align*}
Combining this and equation \eqref{eq:l1diff} with inequality \eqref{eq:Ediff} then yields 
\[
E_\infty(\chi_\Omega) - E_\infty(\chi_{B_\Omega}) \leq  |\Omega| (1+37 M_2[\chi_\Omega]+2|\Omega|) A(\Omega)^{1/2},
\]
which completes the proof.
\end{proof}

Combining the above results, we are now able to show that, along the solution $\rho_\infty(t)$, the energy functional $E_\infty$ is converging towards its global minimizer with an explicit rate.
\begin{theorem}
\label{thm:E_conv}
Let $\Omega_0 \subseteq \mathbb{R}^2$ be a bounded domain with Lipschitz boundary, and let $\rho_\infty(\cdot, t) = \chi_{\Omega(t)}$ be the gradient flow of $E_\infty$ with initial data $\rho_0 = \chi_{\Omega_0}$.
Suppose $B_0$ is a disk with the same area as $\Omega_0$.
 Then, for any $t>0$, we have
\begin{equation*}
0\leq E_\infty(\chi_{\Omega(t)}) - E_\infty(\chi_{B_0}) \leq C_1(|\Omega_0|, M_2[\Omega_0]) t^{-1/6},
\end{equation*}
where $C_1(|\Omega_0|, M_2[\Omega_0]) = C_2 |\Omega|^{2/3}(|\Omega_0|+M_2[\Omega_0])^{7/6}$ and $C_2$ is a universal constant.
\end{theorem}

\begin{proof}
By Corollary \ref{cor:m2}, we have that, for any $t>0$, there exists some $t_0 \in (0,t)$, so that 
\begin{equation*}
A(\Omega(t_0)) \leq\left ( M_2[\chi_{\Omega_0}] / c_0|\Omega_0|^2 \right)^{1/3} t^{-1/3}.
\end{equation*}
By definition of the discrete gradient flow and the lower semicontinuity of $E_\infty$, $E_\infty(\rho_\infty(t))$ is nonincreasing in time. Therefore, at time $t$, we may apply Lemma \ref{A energy estimate} to conclude
\[
\begin{split}
E_\infty(\chi_{\Omega(t)}) - E_\infty(\chi_{B_0}) &\leq E_\infty(\chi_{\Omega(t_0)}) - E_\infty(\chi_{B_0}) \leq 40   |\Omega(t_0)| ( 1 +|\Omega_0|+M_2[\chi_\Omega(t_0)]) \sqrt{A(\Omega(t_0))}\\
&\leq 40 |\Omega_0| ( 1 + |\Omega_0|+M_2[\chi_{\Omega_0}])  \left (M_2[\chi_{\Omega_0}] / c_0|\Omega_0|^2\right)^{1/6} t^{-1/6}\\
&\leq C_2 |\Omega_0|^{2/3}( 1 +|\Omega_0|+M_2[\Omega_0])^{7/6} t^{-1/6}.
\end{split}
\]

\end{proof}

\begin{remark}
While the rate in Theorem \ref{thm:E_conv} is probably not optimal, the following example shows that the optimal power cannot go beneath $-1$. For $0< \epsilon \ll 1$, let $\Omega_0^\epsilon = B(x_\epsilon, \epsilon)\cup B(0,R_\epsilon)$, where $x_\epsilon := (\epsilon^{-1}, 0) \in \mathbb{R}^2$, and $R_\epsilon := \sqrt{1-\epsilon^2}$ is chosen such that $|\Omega_0^\epsilon|= \pi$. This definition ensures that $M_2[\Omega_0^\epsilon]$ is uniformly bounded for all $\epsilon<1$. Since $\partial_r (\mathcal{N}*\chi_{B(0,1)})(r) \sim r^{-1}$ for $r\gg 1$, the extra $\pi\epsilon^2$ amount of mass will stay outside $B(0,(2\epsilon)^{-1})$ for all $t\in [0,c_1 \epsilon^{-2}]$, where $c_1>0$ is independent of $\epsilon$. During this time interval, the free energy is at least $c_2\epsilon^2 |\log \epsilon|$ greater than its global minimizer for some $c_2>0$. Hence $E_\infty(\chi_{\Omega^\epsilon(T_\epsilon)}) - E_\infty(\chi_{B(0,1)}) \gtrsim T_\epsilon^{-1} |\log T_\epsilon|$ for $T_\epsilon = c_1 \epsilon^{-2}$, implying that the optimal power of $t$ in Theorem \ref{thm:E_conv} cannot be less than $-1$.
\end{remark}

\subsection{Convergence of $\rho_\infty(t)$ as $t\to\infty$} \label{Convergence of rhoinfty}
We now conclude our study of asymptotic behavior by showing  that, as $t\to\infty$, $\rho_\infty(t)$ converges to $\chi_{B_0}$ in $L^q$ for any $1\leq q<\infty$, where $B_0$ is the disk with the same area and the center of mass as $\Omega_0$. We begin with the following lemma, which ensures that the center of mass of $\rho_\infty(t)$ is preserved for all time.

\begin{lemma}
\label{lem:center}
Let $\Omega_0 \subseteq \mathbb{R}^2$ be a bounded domain with Lipschitz boundary, and let $\rho_\infty(\cdot, t) = \chi_{\Omega(t)}$ be the gradient flow of $E_\infty$ with initial data $\rho_0 = \chi_{\Omega_0}$. Then for any $T>0$, we have $\int_{\mathbb{R}^2} \rho_\infty(x,T) x dx = \int_{\mathbb{R}^2} \rho_\infty(x,0) x dx$.  
\end{lemma}

\begin{proof}
We proceed as in the proof of Proposition \ref{prop:dm2_dt}.
For any $m>1$, let $\rho_m$ be the weak solution of \ref{pme} with initial data $(\frac{m-1}{m}p_0)^{1/(m-1)}$, where $p_0$ is as in equation \ref{initial}.  For $i=1$ or $2$, we take our test function to be $x_i$, the $i$-th component of $x$. Then, for any $T>0$,
\begin{equation}
\label{eq:center}
\int_{\mathbb{R}^2} \rho_m(x,T) x_i dx - \int_{\mathbb{R}^2} \rho_m(x,0) x_i dx = - \int_0^T \int_{\mathbb{R}^2} \rho_m \partial_i \Phi_{1/m}(x,t)  dxdt.
\end{equation}
By Lemma \ref{lem:weak_conv}, the left hand side of \eqref{eq:center} converges to $\int_{\mathbb{R}^2} \rho_\infty(x,T) x_i dx - \int_{\mathbb{R}^2} \rho_\infty(x,0) x_i dx$ as $m\to\infty$. The right hand side can be controlled in the same way as the term $I_3$ in the proof of Proposition \ref{prop:dm2_dt}, which gives
\[
\lim_{m\to\infty}
\int_{\mathbb{R}^2} \rho_m \partial_i \Phi_{1/m}(x,t)  dx = \int_{\mathbb{R}^2} \rho_\infty \partial_i \Phi (x,t)  dx = \frac{1}{2\pi} \iint_{\mathbb{R}^2\times \mathbb{R}^2} \rho_\infty(x) \rho_\infty(y) \frac{x_i-y_i}{|x-y|^2} dx dy = 0.
\]
Hence, sending $m\to\infty$ in \eqref{eq:center}, we have $\int_{\mathbb{R}^2} \rho_\infty(x,T) x_i dx = \int_{\mathbb{R}^2} \rho_\infty(x,0) x_i dx$ for $i=1,2$, which finishes the proof.
\end{proof}

With this control on the center of mass of $\rho_\infty(x,t)$ in hand, we now turn to the proof of the main result. 

\begin{theorem}\label{long_time}
Let $\Omega_0 \subseteq \mathbb{R}^2$ be a bounded domain with Lipschitz boundary, and let $B_0 \subseteq \mathbb{R}^2$ be a disk such that $|B_0| = |\Omega_0|$ and $\int_{B_0} x dx= \int_{\Omega_0}x dx$. Let $\rho_\infty(\cdot, t) = \chi_{\Omega(t)}$ be the gradient flow of $E_\infty$ with initial data $ \chi_{\Omega_0}$. Then for any $1\leq q< +\infty$, we have 
\[
\lim_{t\to\infty}\|\rho_\infty(\cdot, t) - \chi_{B_0}\|_{L^q(\mathbb{R}^2)} = 0.
\]
\end{theorem}

\begin{proof}
We first show that, for any $f \in C_b(\mathbb{R}^2)$, the space of bounded, continuous functions,
\begin{equation}
\label{eq:weak_conv_f}
\lim_{t\to\infty}\int_{\mathbb{R}^2} \rho_\infty(x,t) f(x) dx = \int_{\mathbb{R}^2} \chi_{B_0} f(x) dx.
\end{equation}
To show this,  take any diverging time sequence $(t_n)_{n=1}^\infty$.
By Proposition \ref{prop:m2_a}, $M_2[\rho_\infty(t_n)]$ is uniformly bounded for all $n$. Hence by Prokhorov's Theorem \cite[Theorem 5.1.3]{AGS}, there exists a subsequence $(t_{n_k})_{k=1}^\infty$ and $\mu \in L^1_+((1+|x|)^2 dx)$ so that 
\[
\lim_{k\to\infty }\int_{\mathbb{R}^2} \rho_\infty(x,t_{n_k}) f(x) dx = \int_{\mathbb{R}^2} \mu(x) f(x) dx\]
for all $f \in C_b(\mathbb{R}^2)$.
Choosing suitable test functions $f$, we have $\int \mu dx = |\Omega_0|$ and  $\|\mu\|_\infty \leq \sup_{t\geq 0} \|\rho_\infty(\cdot, t)\|_\infty = 1$. In addition, by letting the test function $f$ approach $f(x) = x$, we have
\begin{equation}
\label{eq:centerofmass}
\lim_{k\to\infty }\int_{\mathbb{R}^2} \rho_\infty(x,t_{n_k}) x dx = \int_{\mathbb{R}^2} \mu(x) x dx.
\end{equation}

Since the energy functional $E_\infty$ is lower-semicontinuous with respect to weak-* convergence of probability measures \cite[Proposition 4.5]{CraigOmega}, by Theorem \ref{thm:E_conv},
\[
E_\infty(\mu) \leq \liminf_{k\to\infty} E_\infty(\rho_\infty(t_{n_k})) = E_\infty(\chi_{B_0}) .
\]
As the only global minimizers of $E_\infty$ are translations of $\chi_{B_0}$, $\mu$ must equal some translation of $\chi_{B_0}$ almost everywhere. Finally, recall that Lemma~\ref{lem:center} and the definition of $B_0$ give that $\int_{\mathbb{R}^2} \rho_\infty(x,t) x dx = \int  \chi_{\Omega_0} x dx = \int \chi_{B_0} x dx$ for all time. Combining this with \eqref{eq:centerofmass}, we obtain $\int_{\mathbb{R}^2} \mu(x) x dx = \int \chi_{B_0} x dx$, leading to $\mu = \chi_{B_0}$ a.e.. Thus, any diverging time sequence contains a subsequence satisfying \eqref{eq:weak_conv_f}, so we conclude that  \eqref{eq:weak_conv_f} must hold.

We now show that $\rho_\infty(\cdot, t)\to \chi_{B_0}$ in $L^1(\mathbb{R}^2)$. Since $0\leq \rho_\infty \leq 1$, we have $\rho_\infty\leq \chi_{B_0}$ a.e. in $B_0$ and $\rho_\infty\geq \chi_{B_0}$ a.e. in $B_0^c$. Hence 
\[
\|\rho_\infty(\cdot, t) - \chi_{B_0}\|_1 = 2 \int_{\mathbb{R}^2} (\chi_{B_0}-\rho_\infty(x,t) ) \chi_{B_0}dx.
\]
Thus, by choosing $f \in C_b(\mathbb{R}^2)$ sufficiently close to $\chi_{B_0}$ and applying \eqref{eq:weak_conv_f}, we can show that, for any $\epsilon >0$, $\|\rho_\infty(\cdot, t) - \chi_{B_0}\|_1 \leq \epsilon$ for sufficiently large $t$. This shows that $\rho_\infty(\cdot, t)\to \chi_{B_0}$ in $L^1(\mathbb{R}^2)$. Finally, for  $1<q<\infty$, the convergence in $L^q$ follows directly from the $L^1$ convergence and the fact that $\|\rho_\infty(\cdot, t) - \chi_{B_0}\|_\infty \leq 1$.
\end{proof}

\section{Appendix} \label{appendix section}

\subsection{Definition of viscosity solutions of \ref{Pinfty}} \label{Pinftyviscdef}

We begin by recalling some notation. For $Q \subseteq \R^d \times (0,\infty)$, we write $f \in C^{2,1}(Q)$ if $f$ is twice continuously differentiable in $x$ and once in $t$. We say that \emph{$u - \varphi$ has a local maximum (minimum) zero  $(x_0,t_0)$ in $Q$} if there exists $\epsilon >0$ such that 
\[ \varphi(x_0, t_0)= u(x_0,t_0) \text{ and }\varphi \geq u \ (\varphi \leq u) \text{ in }Q \cap (B_\epsilon(x_0) \times (t_0 - \epsilon, t_0+ \epsilon)). \]
In other words, $\varphi$ touches $u$ from above (below) at $(x_0,t_0)$ with respect to $Q$.

Likewise, given an open set $\Omega \subseteq \Rd$ and a function $h: \Omega \times [0,+\infty) \to \R$, we  denote its upper and lower semicontinuous envelopes by
\begin{align} \label{semicontinuousenvelope}
h^*(x,t): = \lim_{\epsilon \to 0} \sup_{\substack{|x-y| \leq \epsilon , \\ |t-s| \leq \epsilon}} h(y,s) , \quad h_*(x,t): = \lim_{\epsilon \to 0} \inf_{\substack{|x-y| \leq \epsilon , \\ |t-s| \leq \epsilon}} h(y,s).
\end{align}
Note that $h^*$ is the smallest upper semicontinuous function satisfying $h \leq h^*$, and $h_*$ is the largest lower semicontinuous function satisfying $h \geq h_*$.

Now, we define the notion of viscosity subsolution, supersolution, and solution of \ref{Pinfty}. Instead of proceeding as above and defining solutions of \ref{Pinfty} by comparison with classical sub- and supersolutions, we follow an approach reminiscent of Kim \cite{Kim} and Alexander, Kim, and Yao \cite{AKY}.  While the former would ease our proof of the comparison theorem, Theorem \ref{comparison}, the latter is more natural from the perspective of the convergence theorem, Theorem \ref{convergence}. One notable difference in the definition below from those of Alexander, Kim, and Yao \cite{AKY} is the separation of the solution and the set evolution in our notion of subsolutions. 
\begin{definition}[subsolution of \ref{Pinfty}] \label{subsolution}
An upper semicontinuous function $u: \Rd \times (0, +\infty) \to [0, +\infty)$, paired with a space-time set $\Sigma= \cup_{t>0}(\Omega(t)\times \{t\})$, is a \emph{viscosity subsolution} of \ref{Pinfty} if
\begin{enumerate}[label = (\alph*)]
\item $ \{u(\cdot,t)>0 \}  \subseteq \overline{\Omega(t)}$  and  $\Sigma \cap\{t\leq t_0\} \subseteq \overline{\Sigma\cap\{t<t_0\}}$ for every $t_0>0$; \label{doesnotjumpup}
\item for all $\varphi \in C^{2,1}(\Rd \times (0, +\infty))$ so that $u-\varphi$ has a local maximum zero at $(x_0, t_0)$ in $\overline{\Sigma} \cap \{t \leq t_0\}$, \label{touchabove}
\begin{enumerate}[label = (\roman*)]
	\item  if $x_0\in \Omega(t_0)^\circ$ or $u(x_0,t_0)>0$, then $-\Delta \varphi(x_0,t_0) \le 1$; \label{ta1}
	\item if $x_0 \in \partial \Omega(t_0)$, $u(x_0, t_0) = 0$, and $|\nabla \varphi|(x_0,t_0) \ne 0$, then \label{ta2}
\[\min(-\Delta \varphi - 1, \varphi_t - |\nabla\varphi|^2-\nabla \varphi \cdot \nabla \Phi)(x_0, t_0) \le 0.\] 
	\end{enumerate}

\end{enumerate}

We will say that $u: \Rd \times [0, +\infty) \to [0, +\infty)$ has compactly supported initial data $u_0$ if, in addition,
\begin{enumerate}[label = (\alph*)] \setcounter{enumi}{2}
\item $u(\cdot, 0) = u_0(\cdot)$ and  $\overline{\{u_0>0\}} = \overline{\Sigma}\cap\{t=0\}$. 
\end{enumerate}
\end{definition}
We introduce the set $\Sigma$ for technical reasons, to allow the possibility that $u$ becomes zero in the evolving set $\Omega(t)$. 
Condition \ref{doesnotjumpup} ensures that a subsolution does not jump up from zero.
Condition \ref{touchabove}\ref{ta2} ensures that limits of viscosity solutions are viscosity solutions, since it is possible that the boundary collapses in a limit and boundary points of the limiting functions become interior points of the limit.

\begin{definition}[supersolution of \ref{Pinfty}] \label{supersolution}
A lower semicontinuous function $v: \Rd \times (0, +\infty) \to (0, +\infty)$ is a \emph{viscosity supersolution} of \ref{Pinfty} with initial data $v_0$ if for all $\varphi \in C^{2,1}(\Rd \times (0, +\infty))$ so that $v - \varphi$ has a local minimum zero at $(x_0, t_0)$ with respect to $\Rd \cap \{ t \leq t_0\}$, 
	\begin{enumerate}[label = (\roman*)]
	\item if $(x_0, t_0) \in \{v > 0\}$, $-\Delta \varphi(x_0,t_0) \ge 1$;
	\item if $(x_0, t_0) \in \partial \{v>0\}, v(x_0, t_0) = 0$,
\begin{align}\label{bdryTouchingCond}
  |\nabla \varphi|(x_0, t_0) \ne 0, \mbox{ and } \{\varphi > 0\} \cap \{v >0 \} \cap B \ne \emptyset  
  \mbox{ for some ball } B \mbox{ centered at } (x_0,t_0)
\end{align} then $\max(-\Delta \varphi - 1, \varphi_t - |\nabla\varphi|^2-\nabla \varphi \cdot \nabla \Phi)(x_0, t_0) \ge 0$.
	\end{enumerate}
Will we say that $v: \Rd \times [0,+\infty) \to (0, +\infty)$ has initial data $v_0$ if $v(\cdot, 0) = v_0(\cdot)$
\end{definition}
Condition \eqref{bdryTouchingCond} ensures that $\varphi$ touches $v$ from below in a non-degenerate way.

\begin{definition} \label{Pinfty solution def}
A lower semi-continuous function $u$ is a \emph{viscosity solution} of \ref{Pinfty} in $\Rd\times (0,\infty)$ with compactly supported initial data $u_0$ if ${(u^*, \{u>0\})}$ and $u$ are respectively viscosity sub- and supersolutions of \ref{Pinfty} with initial data $u_0$.
\end{definition}

The following lemma illustrates the fact that the solution of \ref{Pinfty} is entirely characterized by its support.

\begin{lemma}\label{extra}
Suppose $u$ is a viscosity solution of \ref{Pinfty} in $\Rd\times (0,\infty)$ and $\overline{\{u^*>0\}} = \overline{\{u>0\}}$. Then, for each $t>0$, $u(\cdot,t)=(h_t)_*$, where 
$$
h_t(x)= \inf\{\alpha(x): - \Delta \alpha\geq 1\hbox{ in an open set } E\hbox{ containing } \overline{\{u(\cdot,t)>0\}}; \alpha \geq 0 \hbox{ on } \overline{E}.\}
$$ 
\end{lemma}

\begin{proof}
By the definition of a viscosity supersolution, $-\Delta u(\cdot,t) \geq 1$ in $\{u(\cdot,t)>0\}$, so $(h_t)_* \leq u(\cdot,t)$. On the other hand,  by the definition of a viscosity subsolution, $-\Delta u^* (\cdot,t) \leq 1$ in $\Rd$ and $u^*(\cdot,t)$ is supported in $\overline{\{u(\cdot,t)>0\}}$. Therefore $u^*(\cdot,t) \leq \alpha$ for any candidate function $\alpha(x)$ in the definition of $h_t$, so $u^*(\cdot,t) \leq (h_t)^*$. Consequently, we conclude that $u(\cdot,t)= (h_t)_*$. 

\end{proof}

\subsection{Further properties of gradient flows of $E_\infty$, $\tilde E_\infty$, and $E_{m}$} \label{further properties of energies}

In this section, we collect several results on the gradient flows of $E_\infty$, $\tilde E_\infty$, and $E_m$.
We begin by proving Proposition \ref{N mu bounds}, which provides elementary estimates on the Newtonian potential of a bounded, integrable function. We use these estimates to conclude that $E_\infty$ is $\omega$-convex along generalized geodesics. (See \cite[Theorem 4.3, Proposition 4.4]{CraigOmega}.)

\begin{proof}[Proof of Proposition \ref{N mu bounds}]
The fourth inequality is a classical potential theory result (c.f.\cite[Proposition 2.1]{CarrilloLisiniMainini},  \cite[Lemma 2.1]{Yudovich}), and the fifth inequality is due to Loeper \cite[Theorem 2.7]{Loeper}. (While Loeper only considers the case $d \geq 3$, the same argument applies in $d =2$.)

For the bounds on $\grad \bN \rho$ and $\Delta \bN \rho$, note that if $B = B_1(0)$,
\begin{align*}
 \| \grad \bN \rho \|_\infty &\leq  \|\grad \mathcal{N} \|_{L^1(B)} \|\rho\|_{\infty} + \|\grad \mathcal{N} \|_{L^\infty(B^c)} \|\rho\|_{1}  \leq C_d \quad \text{ and } \quad \| \Delta \bN \rho \|_\infty = \| \rho \|_{L^\infty} \leq 1.
\end{align*} 
Likewise, for the lower bound on $\int \bN \rho d \rho $, if we let $\mathcal{N}^-(x)$ denote the negative part of $\mathcal{N}(x)$,
\begin{align*}
\int \bN \rho d \mu  \geq \int \mathcal{N}^-*\rho(x) d \mu(x) \geq - \|\mathcal{N}^-*\rho\|_\infty \geq -\|\mathcal{N}^- \|_{L^1(B)} \|\rho\|_\infty - \|\mathcal{N}^-\|_{L^\infty(B^c)} \|\rho\|_1 \geq - C_d .
\end{align*}
\end{proof}

Next, we prove Proposition \ref{W2 time lipschitz}, which ensures that $\rho_\infty$ is Lipschitz in time, with respect to the Wasserstein metric.

\begin{proof}[Proof of Proposition \ref{W2 time lipschitz}]
By \cite[Theorem 3.11]{CraigOmega}, the function $S(t): D(E_\infty) \to D(E_\infty): \rho_\infty(\cdot, 0) \mapsto \rho_\infty(\cdot,t)$ is a semigroup, i.e. $S(t + s) = S(t) S(s) \mu$ for $t, s \geq 0$. Therefore, it suffices to show that $W_2(\rho_\infty(t), \rho_\infty(0)) \leq 2C_d t$ for all $t \geq 0$.

Let $\rho^n_\tau$ be the discrete gradient flow of $E_\infty$ with initial data $\rho= \rho_\infty(0)$ and time step $\tau >0$, as defined by equation \ref{Einfty DGF}. By \cite[Theorem 3.8]{CraigOmega}, if we take $\tau = t/n$ for any $t \geq 0$, then $\lim_{n \to +\infty} W_2(\rho^n_{t/n}, \rho_\infty(t)) = 0$.
Therefore,
 \begin{align*}
 W_2(\rho_\infty(t), \rho_\infty(0)) =  \lim_{n \to +\infty} W_2(\rho^n_{t/n}, \rho) \leq \lim_{n \to +\infty} \sum_{i=1}^n W_2(\rho^i_{t/n}, \rho^{i-1}_{t/n})  \leq 2 C_d t ,
 \end{align*}
 where the last inequality follows from Lemma \ref{W2 one step}, which ensures $W_2(\rho^i_{t/n},\rho^{i-1}_{t/n}) \leq 2 C_d(t/n).$
\end{proof}

We now turn to the proof of Proposition \ref{propertiesofspecialPhi}, which concerns the regularity of $\grad \bN \rho_\infty(x,t)$ in space and time.

\begin{proof}[Proof of Proposition \ref{propertiesofspecialPhi}]
The fact that $\grad \bN \rho_\infty(x,t)$ is log-Lipschitz in space is an immediate consequence of Proposition \ref{N mu bounds}.
We now consider the continuity with respect to time. By Proposition \ref{W2 time lipschitz}, $\rho_\infty$ is Lipschitz in time with respect to the Wasserstein metric, so it suffices to translate this into continuity in time with respect to the Euclidean norm.

Fix $\psi \in C^\infty_c(\Rd)$ so that  $\supp \psi \subseteq B_{1}(0)$ and $\|\psi\|_\infty \leq 1$, and let $\Phi(x,t) = \bN \rho_\infty(x,t)$ and $\Phi_{1/m}:= \Phi*\psi_{1/m}$.
Combining the fifth inequality in Proposition \ref{N mu bounds} with  Proposition \ref{W2 time lipschitz},
\begin{align*}& |\grad \Phi_{1/m}(x,t) - \grad \Phi_{1/m}(x,s)| = | \psi_{1/m} * ( \grad \bN \rho(x,t) - \grad \bN \rho(x,s))| \\
&\quad  \leq \|\psi_{1/m}\|_{L^2(\Rd)} \|\grad \bN \rho_\infty(t) - \grad \bN \rho_\infty(s) \|_{L^2(\Rd)} \leq m^{d/2} W_2(\rho_\infty(t), \rho_\infty(s)) \leq 2 C_d m^{d/2} |t-s| .
\end{align*}

We now use this inequality controlling the continuity in time of $\grad \Phi_{1/m}(x,t)$ to estimate the continuity in time of $\grad \Phi(x,t)$. By Proposition \ref{N mu bounds},
\begin{align*}
|\grad \Phi(x,t) - \grad \Phi_{1/m}(x,t)| &= \left| \int \left( \grad \Phi(x,t) - \grad \Phi(x-y,t) \right) \psi_{1/m}(y) dy  \right| \leq C_d \int \sigma(|y|) \psi_{1/m}(y) dy \\
&\leq C_d \sigma(1/m) \int \psi_{1/m}(y) dy = C_d \sigma(1/m) .
\end{align*}
Therefore,
\begin{align*}
&|\grad \Phi(x,t) - \grad \Phi(x,s)| \\
&\quad \leq |\grad \Phi(x,t) - \grad \Phi_{1/m}(x,t)| + |\grad \Phi_{1/m}(x,t) - \grad \Phi_{1/m}(y,t)| + | \grad \Phi_{1/m}(y,t) - \grad \Phi(y,t)| \\
&\quad \leq 2C_d \sigma(1/m) + 2 C_d m^{d/2} |t-s| .
\end{align*}
Let $p = 1/2d$. Since $|t-s| < e^{(-1-\sqrt{2})/2}$, if we choose $m= |t-s|^{(-2/d)(1-p)} \geq 1$, we have $m^{d/2} |t-s| =  |t-s|^{p}$, which takes care of the second term in the above inequality. Furthermore,  $q = 1/(2(2-1/d))<1/2$ ensures $|\log(x)| \leq x^{-1/2} \leq x^{q-1}$ for $0 \leq x \leq 1$. Therefore,
\[ \sigma(1/m) \leq  \left. \begin{cases} (1/m)^{q}  &\text{ if } 1/m < e^{(-1-\sqrt{2})/2} \\ 
 3/m  & \text{ if } 1/m \geq  e^{(-1-\sqrt{2})/2}   \end{cases} \right\} \leq 3 (1/m)^q = 3|t-s|^{(2q/d)(1-p)}= 3|t-s|^p . \]
Therefore, $|\grad \Phi(x,t) - \grad \Phi(x,s)|  \leq 10C_d |t-s|^{1/2d}$, which gives the result.
\end{proof}

In the next proposition, we show that, while the discrete time sequence corresponding to $\tilde E_\infty$ may not be unique, the distance between any two such sequences converges to zero as the time step $\tau \to 0$.

\begin{proposition} \label{distance between DGFs Einfty tilde}
Fix $T>0$ and initial data $\rho \in D(E_\infty)$ and let  $\tilde \rho^n_\tau$ and $\tilde \mu^n_\tau$ be two choices for the time discrete time sequence corresponding to $\tilde{E}_\infty$, as defined in Definition \ref{Discrete time def} \ref{tildeEinfty DGF}. Then there exist positive constants $N$ and $C$, depending on the dimension, $T$, and $E_\infty(\rho),$ so that for $\tau = t/n$ and all $0 \leq t \leq T$ and $n >N$,
\begin{align*}
 f_\tau^{(2n)}(W_2^2(\tilde \rho^n_\tau, \tilde \mu^n_\tau)) &\leq C  \omega(\tau)  . \end{align*} 
\end{proposition}

\begin{proof}
By Corollary \ref{lem:time_conti}, we have the following crude bound for all $i = 1, \dots, n$,
\[ W_2(\tilde \rho^i_\tau, \tilde \mu^i_\tau ) \leq W_2(\tilde \rho^i_\tau, \rho )+ W_2( \tilde \mu^i_\tau, \rho ) \leq 4 C_d T . \]
To obtain a more refined bound, we use Proposition \ref{contraction inequality}. First, we estimate the behavior of the energy $\tilde E_\infty$ along the discrete time sequence. By Proposition \ref{N mu bounds}, Lemma \ref{W2 one step}, and the definition of $\tilde \rho^i_\tau$ as a minimizer
\begin{align*}
 &\tilde E_\infty(\tilde \rho^{i-1}_\tau;  \rho^{i}_\tau) \\
 &= \tilde E_\infty(\tilde \rho^{i-1}_\tau;  \rho^{i-1}_\tau) + \tilde E_\infty(\tilde \rho^{i-1}_\tau;  \rho^{i}_\tau) - \tilde E_\infty(\tilde \rho^{i-1}_\tau;  \rho^{i-1}_\tau) =\tilde E_\infty(\tilde \rho^{i-1}_\tau;  \rho^{i-1}_\tau) +  \int \bN \tilde \rho^{i-1}_\tau d (\rho^{i}_\tau - \rho^{i-1}_\tau)  \\
 & \leq \tilde E_\infty(\tilde \rho^{i-2}_\tau;  \rho^{i-1}_\tau) + C_d W_2(\rho^i_\tau, \rho^{i-1}_\tau) \leq \tilde E_\infty(\tilde \rho^{i-2}_\tau;  \rho^{i-1}_\tau) +2 C_d^2 \tau \leq \dots \leq \tilde E_\infty(\rho; \rho_\tau^1) + 2 C_d^2 T  .
\end{align*}
Likewise, we may control the first term on the right hand side by
\[ \tilde E_\infty (\rho; \rho_\tau^1)  = 2 E_\infty(\rho) + \tilde E_\infty (\rho; \rho_\tau^1) -  \tilde E_\infty (\rho; \rho) = 2 E_\infty(\rho) + \int \bN \rho d (\rho^{1}_\tau - \rho) \leq 2 E_\infty(\rho) + 2C_d^2 \tau . \]
Thus, there exists $C>0$ (which we allow to change from line to line) depending only on the dimension, $T$, and $E_\infty(\rho)$ so that 
\[ \tilde E_\infty(\tilde \rho^{i-1}_\tau;  \rho^{i}_\tau) \leq C.\]
Likewise, by Proposition \ref{N mu bounds}, $\tilde E_\infty(\cdot ;\cdot)$ is uniformly bounded below by $-C_d$.

Due to these estimates, we may apply Proposition \ref{contraction inequality} to conclude that there exist positive constants $C$ and $N$ depending on the dimension, $T$, and $E_\infty(\rho)$ so that for $\tau = t/n$, $0 \leq t \leq T$, and $n > N$,
\begin{align*}
 &f_\tau^{(2)}(W_2^2(\tilde \rho^i_\tau, \tilde \mu^i_\tau)) \\
 &\leq W_2^2(\tilde \rho^{i-1}_\tau,\tilde \mu^{i-1}_\tau) + C_d \tau \omega(C W_2(\tilde \mu^i_\tau, \tilde \mu^{i-1}_\tau)) + 2 \tau(\tilde E_\infty(\tilde \rho^{i-1}_\tau; \rho^i_\tau) - \tilde E_\infty(\tilde \rho^{i}_\tau; \rho^i_\tau)) +  C \tau^2 .
 \end{align*}
By Lemma \ref{W2 one step} \ref{tilde rho tau one step}, we may bound the second term by $C_d \tau \omega(C \tau)$ and the third term by $4 C_d^2 \tau^2$.
Therefore, for all $i =1 , \dots, n$,
\begin{align} \label{iterate contraction0}
 &f_\tau^{(2)}(W_2^2(\tilde \rho^i_\tau, \tilde \mu^i_\tau)) \leq W_2^2(\tilde \rho^{i-1}_\tau,\tilde \mu^{i-1}_\tau) +C \tau \omega(\tau) .
 \end{align}

We now show that, for all $j = 1, \dots, n$,
\begin{align} \label{iterate contraction1}
 f_\tau^{(2j)}(W_2^2(\tilde \rho^n_\tau, \tilde \mu^n_\tau)) &\leq W_2^2(\tilde \rho^{n-j}_\tau,\tilde \mu^{n-j}_\tau)+ 2C \tau \omega(\tau) j . \end{align} 
Once we have this, taking $j = n$ gives the result. We prove (\ref{iterate contraction1}) by induction. The base case, when $j = 1$, is a consequence of (\ref{iterate contraction0}). Suppose that the result holds for $j-1$,
 \begin{align*}
 f_\tau^{(2(j-1))}(W_2^2(\tilde \rho^n_\tau, \tilde \mu^n_\tau)) &\leq W_2^2(\tilde \rho^{n-j+1}_\tau,\tilde \mu^{n-j+1}_\tau)+ 2C \tau \omega(\tau) (j-1) . \end{align*} 
By Proposition \ref{f tau prop}, applying $f_\tau^{(2)}$ to both sides,
 \begin{align*}
 f_\tau^{(2j)}(W_2^2(\tilde \rho^n_\tau, \tilde \mu^n_\tau)) &\leq f_\tau^{(2)}(W_2^2(\tilde \rho^{n-j+1}_\tau,\tilde \mu^{n-j+1}_\tau))+ 2C \tau \omega(\tau) (j-1) + C \tau^2 \\
 &\leq W_2^2(\tilde \rho^{n-j}, \tilde \mu^{n-j}) + 2C \tau \omega(\tau) j \end{align*} 
 where the second inequality is a consequence of (\ref{iterate contraction0}) and the fact that $C \tau^2 \leq C \tau \omega(\tau)$. This gives the result.
\end{proof}

Now, we turn to the proof that the discrete time sequence $\rho_{\tau,m}^n$ corresponding to $E_m$ converges to the solution of \ref{pme} as the time step goes to zero.

\begin{proposition} \label{time dependent assumption prop}
Given initial data $\rho \in D(E_\infty)$, let $\rho_{\tau,m}^n$ be the discrete time sequence given in Definition \ref{Discrete time def} \ref{Em DGF}.
Then, for any $t \geq 0$, $\rho^n_{t/n,m}$ converges as $n\to +\infty$ to a limit $\rho_m(t)$, and there exist positive constants $C$ and $N$ depending on the dimension, $E_\infty(\rho), $ and $T$ so that for all $n \geq N$, $m \geq d+1$, and $0 \leq t \leq T$,
\[ W_2 (\rho^n_{t/n,m}, \rho_m(t)) \leq C   n^{-1/16e^{4C_dT}} . \]
Furthermore, $\rho_m(t)$ is the unique weak solution of \ref{pme}.
\end{proposition}

\begin{proof}
Given initial data $\rho \in D(E_\infty)$, let $\rho^n_\tau$ be the discrete gradient flow of $E_\infty$, as in Definition \ref{Discrete time def} \ref{Einfty DGF}. Using this sequence, we define a time dependent energy $E^n_{\tau,m}$ by 
\begin{align*}
E^n_{\tau,m}(\nu):=  E_m(\nu; \rho^n_\tau) = \begin{cases} \frac{1}{m-1} \int_\Rd \nu(x)^m dx + \int_\Rd \psi_{1/m}* \bN \rho^n_\tau(x) d\nu(x)  &\text{ if } \nu \ll \mathcal{L}^d, \\
+ \infty &\text{ otherwise.}
\end{cases}
\end{align*}
Then $\rho_{\tau,m}^n$ given in Definition \ref{Discrete time def} \ref{Em DGF} is the time varying discrete gradient flow of this energy in the sense that
\begin{align} \label{Em prox map}  \rho_{\tau,m}^n \in \argmin_{\nu \in \PR} \left\{ \frac{1}{2\tau} W_2^2(\rho_{\tau,m}^{n-1}, \nu) +  E^n_{\tau,m}(\nu)\right\} \text{ and } \rho^0_{\tau,m} := \rho. \end{align}
Consequently, we may apply the first author's results on convergence of the discrete gradient flow of time dependent energies \cite[Theorem A.3]{CraigOmega}, provided that we can show $E^n_m$ satisfies \cite[Assumption A.2]{CraigOmega}.

First, by \cite[Theorem 4.3, Proposition 4.4]{CraigOmega},  $E^n_\tau$ satisfies \cite[Assumption 2.18]{CraigOmega} uniformly for $n \in \mathbb{N}$, $m >1$, and $\tau >0$. In particular, there exists a solution to the minimization problem (\ref{Em prox map}) and $E^n_\tau$ is $\omega$-convex along generalized geodesics, for $\lambda_\omega = -C_d$ as in Proposition \ref{N mu bounds} and $\omega(x)$ as in equation (\ref{omega def}).

Next, we estimate the behavior of the energies and Wasserstein distance along the discrete gradient flow. By Lemma \ref{W2 one step} \ref{rho m tau one step}, for all $1 \leq i \leq n $,
\begin{align*}
W_2(\rho^{i}_{\tau,m}, \rho^{i-1}_{\tau,m}) \leq \sqrt{\frac{2\tau}{m-1} ( \| \rho^{i-1}_{\tau,m}\|_m^m - \|\rho^i_{\tau,m} \|_m^m)} + 2C_d \tau \leq \sqrt{2 \tau \left(1 + n \tau C_d^2/2 \right)} + 2 C_d \tau .
\end{align*}
where, in the second inequality, we use that $\|\rho^0\|_m^m \leq 1$. Likewise, by Corollary \ref{lem:time_conti},
\[ W_2(\rho_{m,\tau}^n,\rho) \leq \sqrt{4n\tau (1+8C_d^2 n \tau)} . \]
Finally, since Proposition \ref{N mu bounds} ensures $E^n_{\tau,m}$ is uniformly bounded below by $-C_d$, there exists a constant $\tilde{C}_d >0$, depending only on the dimension, so that
\begin{align*}
E^0_{\tau,m}(\rho) - E^n_{\tau,m}(\rho^n_{\tau,m}) \leq E_m(\rho, \rho) +C_d \leq 1 + \int \bN \rho(x) \psi_{1/m}*\rho(x) dx +C_d \leq \tilde{C}_d + E_\infty(\rho) ,
\end{align*}
where in the last inequality we use that $\bN \rho(x)$ is a continuous function with at most quadratic growth and  $\psi_{1/m}*\rho \xrightarrow[W_2]{m \to +\infty} \rho$, so $\int \bN \rho(x) \psi_{1/m}*\rho(x) dx \xrightarrow{m \to +\infty} \int \bN \rho(x) \rho(x) dx$.

It remains to show that $E^n_{\tau,m}$ possesses sufficient continuity in $n \tau$. To do this, we first estimate the continuity of $\rho^n_\tau$ in $n \tau$. By Lemma \ref{W2 one step}, we have the following crude bound
\[ W_2^2(\rho^n_\tau, \rho^k_h)\leq (2C_d (n \tau +kh))^2 \leq 16 C_d^2 T^2 . \]
Combining this with Proposition \ref{f tau prop} \ref{ode euler estimate} and \cite[Theorem 3.6]{CraigOmega}, we obtain that for any $T>0$, there exists $\bar{\tau}= \bar{\tau}(T,d)$ and $\bar{C}= \bar{C}(T,d)$ so that for all $0\leq h < \tau < \bar{\tau}$ and $k,n \in \mathbb{N}$ with $k h, n \tau \leq T$, 
\begin{align*}
 F_{2kh}( W_2^2(\rho^n_\tau, \rho^k_h))  &\leq \bar{C}\left[\sqrt{(n\tau - kh)^2+\tau^2n}  + h k\tomega(\sqrt{\tau}) +h^2k + \tomega(h^2)k \right] \\
 &\quad +2h (E_\infty(\rho) -\inf E_\infty)  + C_d \omega(16 C_d^2 T^2) T/n .
\end{align*}
Since $F_{t}(x)$ is decreasing in $t$, this implies there exists $\tilde{C}= \tilde{C}(T,d, E_\infty(\rho))$ so that for $0< \tau < \bar{\tau}$,
\begin{align} \label{time varying ex1}
& F_{2T}( W_2^2(\rho^n_\tau, \rho^k_h))  \leq \tilde{C}\left[\sqrt{(n\tau - kh)^2}  + \sqrt{\tau}|\log \tau|  \right] .
\end{align}
Since $F_{2T}(x)$ is strictly increasing and convex in $x$, $F_{2T}^{-1}(x)$ is strictly increasing and concave. Therefore,
\begin{align*} \sigma(x):= \sqrt{ F_{2T}^{-1}(\sqrt{x})}
\end{align*}
 is a continuous, nondecreasing, concave function that vanishes only at zero. In particular, $\sigma(x)$ is also subadditive, so (\ref{time varying ex1}) implies that, for some $C'= C'(T,d, E_\infty(\rho))$,
 \begin{align} \label{time varying ex2} W_2(\rho^n_\tau,\rho^k_h) \leq C'\left[ \sigma \left( (n \tau - kh)^2 \right) + \sigma \left(\tau |\log \tau|^2 \right)  \right].
\end{align}

We use this estimate to show that $E^n_{\tau,m}$ is continuous in $n \tau$, up to an error that decreases with $\tau$. Since $f := \bN(\psi_{1/m}* \rho^i_{\tau,m}) \in C^1$, by Lemma \ref{drift continuity estimate0},
\begin{align*}
&|E^n_{\tau,m}( \rho^i_{\tau,m}) - E^k_{h,m}( \rho^i_{\tau,m})| = | E_m(\rho^i_{\tau,m}; \rho^n_\tau)  - E_m(\rho^i_{\tau,m}; \rho^k_h) | =  \left| \int_\Rd  \bN (\psi_{1/m} * \rho^i_{\tau,m}) d(\rho^n_\tau - \rho^k_h) \right| \\
&\quad \leq C' \|\grad f\|_\infty \left[ \sigma \left( (n \tau - kh)^2 \right) + \sigma \left(\tau |\log \tau|^2 \right)  \right] .
\end{align*}
Finally, $\|\grad f \|_\infty$ is bounded uniformly in $m$, $i$, and $\tau$, since for $B = B_1(0)$, there exists $c$ depending only on the dimension (and which we allow to change from line to line) so that, for all $m \geq d+1$,
\begin{align*}
 \|\grad f \|_\infty &\leq \|\grad \mathcal{N} \|_{L^\infty(\Rd \setminus B)} + \|\grad \mathcal{N} \|_{L^{m'}(B)} \|\rho^i_{\tau,m} \|_{L^m(\Rd)} \leq  c +(1/\alpha_d)^{(m-1)/(m)}\|\rho^i_{\tau,m} \|_{L^m(\Rd)} \\
 &\leq c \left(1+ \|\rho\|_m + \left( (m-1)T C_d^2/2 \right)^{1/m} \right) \leq c
\end{align*}
where the fourth inequality uses Lemma \ref{W2 one step}.

Thus, \cite[Assumption A.2]{CraigOmega} is satisfied, so by \cite[Theorem A.3]{CraigOmega}, we conclude that for all $0 \leq t \leq T$, there exists $C= C(E_\infty(\rho),T,d)$ (which we allow to change from line to line) so
\begin{align*} F_{2t} \left(W_2^2(\rho^n_{t/n,m}, \rho_m(t)) \right) &\leq C \left[ t/\sqrt{n} + t \omega \left(\sqrt{t/n} \right) +\sigma(t^2/n) + \sigma(t/n |\log(t/n)|^2) \right] .
\end{align*}
Hence, using again that $F_{t}(x)$ is decreasing in $t$,
\begin{align*}
F_{2T} \left(W_2^2(\rho^n_{t/n,m}, \rho_m(t)) \right)
&\leq C \left[ n^{-1/2} \log n +  \sqrt{ F_{2T}^{-1}(t/\sqrt{n})} + \sqrt{ F_{2T}^{-1}(\sqrt{t/n} |\log(t/n)|)}  \right]. 
\end{align*}
For $0 \leq x \leq e^{-1-\sqrt{2}}$, $F_t(x) = x^{e^{C_d t}}$ and $n^{-1/2} \log n = O(n^{-1/4})$, so for $n$ sufficiently large,
\begin{align*}
 \left(W_2 (\rho^n_{t/n,m}, \rho_m(t)) \right)^{2e^{2C_d T}}
\leq C  ( n^{-1/8} )^{1/e^{2C_dT}}  \implies 
W_2 (\rho^n_{t/n,m}, \rho_m(t))
\leq C   n^{-1/16e^{4C_dT}} .
\end{align*}

Finally, it remains to show that the limit $\rho_m$ is the unique solution of \ref{pme}. Following a parallel argument as in Jordan, Kinderlehrer, and Otto's original work on the convergence of the discrete gradient flow to solutions of the Fokker-Planck equation \cite{JKO}, one can show that for all $\zeta \in C^\infty_0(\Rd \times [0, +\infty))$,
\begin{align} \label{weakPMED}
0=& \int_\Rd \rho_m(x,0) \zeta(x,0) dx + \int_0^{+\infty} \int_\Rd \rho_m(x,s) (\partial_s \zeta(x,s) - \grad \Phi_{1/m}(x,s) \grad \zeta(x,s)) dx ds \\ &\quad + \int_0^{+\infty} \int_\Rd \rho_m(x,s)^m \Delta \zeta(x,s) dx ds . \nonumber
 \end{align}

To conclude that $\rho_m(x,t)$ is the unique weak solution of \ref{pme}, it remains to show that for all $0< t < +\infty$,
\begin{align} \label{pmetodo1}
\int_0^t \int_\Rd |\rho_m(x,s)|^m dx ds < +\infty , \\
\int_0^t \left( \int_\Rd \left| \frac{\nabla \rho_m(x,s)^m}{\rho_m(x,s)} + \nabla \Phi_{1/m}(x,s) \right|^2 \rho_m(x,s) dx \right)^{1/2} ds < +\infty . \label{pmetodo2}
\end{align}
(See, for example, \cite[Theorem 6.1]{CarlierLaborde} and \cite[Theorem 7.1]{Laborde}. While these references do not consider the case of a time-dependent drift, an identical argument applies to the present case.)

To show (\ref{pmetodo1}) and (\ref{pmetodo2}), we define the following piecewise constant interpolations:
\begin{align*}
\bar{\rho}_{t/n,m}(x,s) := \rho^i_{t/n,m}(x) \ \text{ and } \ \bar{\rho}_{t/n}(x,s) := \rho^i_{t/n}  \  \text{ for } s \in ((i-1) t/n, i t/n] .
\end{align*}
Using \cite[Theorem 3.6]{CraigOmega} (see \cite[Appendix A.3]{CraigOmega} for the adaptation to ``time dependent'' gradient flows), one can show that 
\[ \bar{\rho}_{t/n,m}(x,s) \xrightarrow[W_2]{n \to +\infty} \rho_m(x,s) \text{ and } \bar{\rho}_{t/n}(x,s) \xrightarrow[W_2]{n \to +\infty} \rho_\infty(x,s) . \]

We begin with (\ref{pmetodo1}). By the lower semicontinuity of $\| \cdot \|_m^m$ with respect to Wasserstein convergence \cite[Lemma 3.4]{McCann}, Fatou's Lemma, and Lemma \ref{W2 one step}, which bounds $\|\rho^i_{t/n,m} \|_m^m$ uniformly in $i$ and $n$,
 \begin{align*}
 \int_0^t \int_\Rd |\rho_m(x,s)|^m dx ds &\leq \liminf_{n \to +\infty} \int_0^{t } \int_\Rd |\bar{\rho}_{t/n,m}(x,s)|^m dx ds  = 
 \liminf_{n \to +\infty} \sum_{i=1}^n \frac{t}{n} \int_\Rd |\rho_{t/n,m}^i(x)|^m dx <+\infty .
\end{align*}

We now turn to  (\ref{pmetodo2}). To ease notation, we recall the definition of the R\'enyi entropy and its metric slope  \cite[Theorem 10.4.6]{AGS},
\[ S_m(\mu) :=
\frac{1}{m-1} \int \mu(x)^m dx , \quad |\partial S_m|(\mu) = \left( \int_\Rd \left| \frac{\nabla \mu(x)^m}{\mu(x)} \right|^2 \mu(x) dx \right)^{1/2} . \]
By inequality (\ref{eq:phi_bound}),
\[ \int_0^t \left( \int_\Rd \left |\nabla \Phi_{1/m}(x,s) \right|^2 \rho_m(x,s) dx \right)^{1/2} ds \leq C_d t < +\infty . \]
Thus, by the triangle inequality and Jensen's inequality, it suffices to show  
\begin{align} \label{pmetodo2 suffices} \left(\int_0^t |\partial S_m|^2(\rho_m(s)) ds \right)^{1/2}< +\infty .
\end{align}

By \cite[Theorem 3.1.6, Theorem 10.4.13]{AGS} and the fact that $E_m$ is $\lambda$-convex for $\lambda = \lambda(m)$,
\begin{align*}
\frac{t}{n} \int \left| \frac{\nabla ( \rho_{t/n,m}^i)^m}{\rho_{t/n,m}^i} + \nabla \psi_{1/m}*\bN \rho^i_{t/n} \right|^2 d \rho_{t/n,m}^i  \leq \frac{1}{1+\lambda t/n} \left[E_m( \rho_{t/n,m}^{i-1};\rho_{t/n}^i) - E_m( \rho_{t/n,m}^i;\rho_{t/n}^i) \right] .
\end{align*}
Summing both sides from $i= 1, \dots, n$ and using the definition of $\bar{\rho}_{t/n,m}$, $\bar{\rho}_{t/n}$, and $E_m$¬,
\begin{align*}
& \int_0^t \int_\Rd \left| \frac{\nabla ( \bar{\rho}_{t/n,m}(x,s))^m}{\bar{\rho}_{t/n,m}(x,s)} + \nabla \psi_{1/m}*\bN \bar{\rho}_{t/n}(x,s) \right|^2 \bar{\rho}_{t/n,m}(x,s) dx ds  \\ &\quad \leq  \frac{1}{1+\lambda t/n} \left[E_m( \rho; \rho_{t/n}^1) - E_m( \rho_{t/n,m}^n; \rho_{t/n}^n) +  \sum_{i=1}^n \int_\Rd \bN \psi_{1/m}*\rho^i_{t/n,m}(x) \left[\rho^{i+1}_{t/n}(x) - \rho^i_{t/n}(x) \right] dx \right] , \\
&\quad \leq  \frac{1}{1+\lambda t/n} \left[E_m( \rho; \rho_{t/n}^1) +C_d +  2 C_m C_d t  \right],
\end{align*}
where in the last inequality we use Propositions \ref{drift continuity estimate0}, \ref{N mu bounds}, and \ref{W2 one step}, for $C_m$ chosen so that $\| \grad \bN\psi_{1/m}*\rho^i_{t/n,m} \|_\infty <C_m$. Taking the square root of both sides and applying the reverse triangle inequality for the $L^2(d \bar{\rho}_{t/n,m}(\cdot, s))$ norm and Proposition \ref{N mu bounds},
\begin{align*}
&  \left(\int_0^t |\partial S_m|^2(\bar{\rho}_{t/n,m}(s)) ds \right)^{1/2} \leq  \frac{1}{\sqrt{1+\lambda t/n}} \left[E_m( \rho; \rho_{t/n}^1) +C_d +  2 C_m C_d t  \right]^{1/2} + C_d .
\end{align*}
Taking the $\liminf_{n \to +\infty}$ and using that $|\partial S_m|$ is lower semicontinuous with respect to Wasserstein convergence \cite[Corollary 2.4.10]{AGS}, we conclude the result.

\end{proof}

 \noindent \textbf{Acknowledgements.}
The authors would like to thank Jos\'e Antonio Carrillo, Alessio Figalli, and  Alp\'ar M\'esz\'aros for their helpful comments.

\bibliographystyle{spmpsci}
\bibliography{interaction_energy}

\begin{thebibliography}{10}
\providecommand{\url}[1]{{#1}}
\providecommand{\urlprefix}{URL }
\expandafter\ifx\csname urlstyle\endcsname\relax
  \providecommand{\doi}[1]{DOI~\discretionary{}{}{}#1}\else
  \providecommand{\doi}{DOI~\discretionary{}{}{}\begingroup
  \urlstyle{rm}\Url}\fi

\bibitem{AKY}
Alexander, D., Kim, I., Yao, Y.: Quasi-static evolution and congested crowd
  transport.
\newblock Nonlinearity \textbf{27}(4), 823--858 (2014).
\newblock \doi{10.1088/0951-7715/27/4/823}.
\newblock \urlprefix\url{http://dx.doi.org/10.1088/0951-7715/27/4/823}

\bibitem{AGS}
Ambrosio, L., Gigli, N., Savar{\'e}, G.: Gradient flows in metric spaces and in
  the space of probability measures, second edn.
\newblock Lectures in Mathematics ETH Z\"urich. Birkh\"auser Verlag, Basel
  (2008)

\bibitem{AmbrosioSerfaty}
Ambrosio, L., Serfaty, S.: A gradient flow approach to an evolution problem
  arising in superconductivity.
\newblock Comm. Pure Appl. Math. \textbf{61}(11), 1495--1539 (2008).
\newblock \doi{10.1002/cpa.20223}.
\newblock \urlprefix\url{http://dx.doi.org/10.1002/cpa.20223}

\bibitem{BalagueCarrilloLaurentRaoul}
Balagu\'e, D., Carrillo, J., Laurent, T., Raoul, G.: Nonlocal interactions by
  repulsive-attractive potentials: Radial ins/stability.
\newblock Phys. D. \textbf{260}, 5--25 (2013)

\bibitem{BalagueCarrilloYao}
Balagu{\'e}, D., Carrillo, J.A., Yao, Y.: Confinement for repulsive-attractive
  kernels.
\newblock Discrete Contin. Dyn. Syst. Ser. B \textbf{19}(5), 1227--1248 (2014).
\newblock \doi{10.3934/dcdsb.2014.19.1227}.
\newblock \urlprefix\url{http://dx.doi.org/10.3934/dcdsb.2014.19.1227}

\bibitem{BenedettoCagliotiCarrilloPulvirenti}
Benedetto, D., Caglioti, E., Carrillo, J.A., Pulvirenti, M.: A non-{M}axwellian
  steady distribution for one-dimensional granular media.
\newblock J. Statist. Phys. \textbf{91}(5-6), 979--990 (1998).
\newblock \doi{10.1023/A:1023032000560}.
\newblock \urlprefix\url{http://dx.doi.org/10.1023/A:1023032000560}

\bibitem{BertozziCarrilloLaurent}
Bertozzi, A.L., Carrillo, J.A., Laurent, T.: Blow-up in multidimensional
  aggregation equations with mildly singular interaction kernels.
\newblock Nonlinearity \textbf{22}(3), 683--710 (2009).
\newblock \doi{10.1088/0951-7715/22/3/009}.
\newblock \urlprefix\url{http://dx.doi.org/10.1088/0951-7715/22/3/009}

\bibitem{Bertozzietal_RingPatterns}
Bertozzi, A.L., Kolokolnikov, T., Sun, H., Uminsky, D., von Brecht, J.: Ring
  patterns and their bifurcations in a nonlocal model of biological swarms.
\newblock Commun. Math. Sci. \textbf{13}(4), 955--985 (2015).
\newblock \doi{10.4310/CMS.2015.v13.n4.a6}.
\newblock \urlprefix\url{http://dx.doi.org/10.4310/CMS.2015.v13.n4.a6}

\bibitem{BertozziLaurentLeger}
Bertozzi, A.L., Laurent, T., L{\'e}ger, F.: Aggregation and spreading via the
  {N}ewtonian potential: the dynamics of patch solutions.
\newblock Math. Models Methods Appl. Sci. \textbf{22}(suppl. 1), 1140,005, 39
  (2012).
\newblock \doi{10.1142/S0218202511400057}.
\newblock \urlprefix\url{http://dx.doi.org/10.1142/S0218202511400057}

\bibitem{Blanchet}
{Blanchet}, A.: A gradient flow approach to the {K}eller-{S}egel systems.
\newblock to appear in RIMS Kokyuroku's lecture notes, preprint at
  \url{http://publications.ut-capitole.fr/16518/}

\bibitem{BlanchetCarlenCarrillo}
Blanchet, A., Carlen, E.A., Carrillo, J.A.: Functional inequalities, thick
  tails and asymptotics for the critical mass {P}atlak-{K}eller-{S}egel model.
\newblock J. Funct. Anal. \textbf{262}(5), 2142--2230 (2012).
\newblock \doi{10.1016/j.jfa.2011.12.012}.
\newblock \urlprefix\url{http://dx.doi.org/10.1016/j.jfa.2011.12.012}

\bibitem{BurgerFetecauHuang}
Burger, M., Fetecau, R., Huang, Y.: Stationary states and asymptotic behavior
  of aggregation models with nonlinear local repulsion.
\newblock SIAM J. Appl. Dyn. Syst. \textbf{13}(1), 397--424 (2014).
\newblock \doi{10.1137/130923786}.
\newblock \urlprefix\url{http://dx.doi.org/10.1137/130923786}

\bibitem{CaffarelliSalsa}
Caffarelli, L., Salsa, S.: A geometric approach to free boundary problems,
  \emph{Graduate Studies in Mathematics}, vol.~68.
\newblock American Mathematical Society, Providence, RI (2005).
\newblock \doi{10.1090/gsm/068}.
\newblock \urlprefix\url{http://dx.doi.org/10.1090/gsm/068}

\bibitem{CaffarelliVazquez}
Caffarelli, L., Vazquez, J.L.: Viscosity solutions for the porous medium
  equation.
\newblock In: Differential equations: {L}a {P}ietra 1996 ({F}lorence),
  \emph{Proc. Sympos. Pure Math.}, vol.~65, pp. 13--26. Amer. Math. Soc.,
  Providence, RI (1999).
\newblock \doi{10.1090/pspum/065/1662747}.
\newblock \urlprefix\url{http://dx.doi.org/10.1090/pspum/065/1662747}

\bibitem{CarlierLaborde}
{Carlier}, G., {Laborde}, M.: On systems of continuity equations with nonlinear
  diffusion and nonlocal drifts.
\newblock preprint at \url{http://arxiv.org/abs/1505.01304}

\bibitem{chvy}
Carrillo, J.A., Hittmeir, S., Volzone, B., Yao, Y.: Nonlinear
  aggregation-diffusion equations: Radial symmetry and long time asymptotics.
\newblock in preparation

\bibitem{CarrilloLisiniMainini}
Carrillo, J.A., Lisini, S., Mainini, E.: Uniqueness for {K}eller-{S}egel-type
  chemotaxis models.
\newblock Discrete Contin. Dyn. Syst. \textbf{34}(4), 1319--1338 (2014).
\newblock \doi{10.3934/dcds.2014.34.1319}.
\newblock \urlprefix\url{http://dx.doi.org/10.3934/dcds.2014.34.1319}

\bibitem{CarrilloMcCannVillani}
Carrillo, J.A., McCann, R.J., Villani, C.: Contractions in the 2-{W}asserstein
  length space and thermalization of granular media.
\newblock Arch. Ration. Mech. Anal. \textbf{179}(2), 217--263 (2006).
\newblock \doi{10.1007/s00205-005-0386-1}.
\newblock \urlprefix\url{http://dx.doi.org/10.1007/s00205-005-0386-1}

\bibitem{ChuangHuangDorsognaBertozzi}
Chuang, Y.L., Huang, Y., D'Orsogna, M., Bertozzi, A.: Multi-vehicle flocking:
  scalability of cooperative control algorithms using pairwise potentials.
\newblock IEEE International Conference on Robotics and Automation pp.
  2292--2299 (2007)

\bibitem{CraigOmega}
{Craig}, K.: Nonconvex gradient flow in the {Wasserstein} metric and
  applications to constrained nonlocal interactions.
\newblock preprint at \url{http://arxiv.org/abs/1512.07255}

\bibitem{DoyeWalesBerry}
Doye, J.P.K., Wales, D.J., Berry, R.S.: The effect of the range of the
  potential on the structures of clusters.
\newblock J. Chem. Phys. \textbf{103}, 4234--4249 (1995)

\bibitem{Federer}
Federer, H.: Curvature measures.
\newblock Trans. Amer. Math. Soc. \textbf{93}, 418--491 (1959)

\bibitem{FellnerRaoul2}
Fellner, K., Raoul, G.: Stable stationary states of non-local interaction
  equations.
\newblock Math. Models Methods Appl. Sci. \textbf{20}(12), 2267--2291 (2010).
\newblock \doi{10.1142/S0218202510004921}.
\newblock \urlprefix\url{http://dx.doi.org/10.1142/S0218202510004921}

\bibitem{FetecauHuang}
Fetecau, R.C., Huang, Y.: Equilibria of biological aggregations with nonlocal
  repulsive-attractive interactions.
\newblock Phys. D \textbf{260}, 49--64 (2013).
\newblock \doi{10.1016/j.physd.2012.11.004}.
\newblock \urlprefix\url{http://dx.doi.org/10.1016/j.physd.2012.11.004}

\bibitem{FetecauHuangKolokolnikov}
Fetecau, R.C., Huang, Y., Kolokolnikov, T.: Swarm dynamics and equilibria for a
  nonlocal aggregation model.
\newblock Nonlinearity \textbf{24}(10), 2681--2716 (2011).
\newblock \doi{10.1088/0951-7715/24/10/002}.
\newblock \urlprefix\url{http://dx.doi.org/10.1088/0951-7715/24/10/002}

\bibitem{fmp}
Fusco, N., Maggi, F., Pratelli, A.: Stability estimates for certain
  {Faber-Krahn}, isocapacitary and {Cheeger} inequalities.
\newblock Ann. Sc. Norm. Super. Pisa Cl. Sci. \textbf{8}(5), 51--71 (2009)

\bibitem{JKO}
Jordan, R., Kinderlehrer, D., Otto, F.: The variational formulation of the
  {F}okker-{P}lanck equation.
\newblock SIAM J. Math. Anal. \textbf{29}(1), 1--17 (1998).
\newblock \doi{10.1137/S0036141096303359}.
\newblock \urlprefix\url{http://dx.doi.org/10.1137/S0036141096303359}

\bibitem{Yudovich}
Judovi{\v{c}}, V.I.: Non-stationary flows of an ideal incompressible fluid.
\newblock \u Z. Vy\v cisl. Mat. i Mat. Fiz. \textbf{3}, 1032--1066 (1963)

\bibitem{KellerSegel}
Keller, E., Segel, L.: Initiation of slide mold aggregation viewed as an
  instability.
\newblock J. Theoret. Biol. \textbf{26} (1970)

\bibitem{KimPozar}
{Kim}, I., Pozar, N.: Porous medium equation to {H}ele-{S}haw flow with general
  initial density.
\newblock preprint at \url{http://arxiv.org/abs/1509.06287}

\bibitem{KimYao}
Kim, I., Yao, Y.: The {Patlak-Keller-Segel} model and its variations:
  properties of solutions via maximum principle.
\newblock SIAM Journal on Mathematical Analysis \textbf{44}(2), 568--602 (2012)

\bibitem{Kim}
Kim, I.C.: Uniqueness and existence results on the {H}ele-{S}haw and the
  {S}tefan problems.
\newblock Arch. Ration. Mech. Anal. \textbf{168}(4), 299--328 (2003).
\newblock \doi{10.1007/s00205-003-0251-z}.
\newblock \urlprefix\url{http://dx.doi.org/10.1007/s00205-003-0251-z}

\bibitem{KimLei}
Kim, I.C., Lei, H.K.: Degenerate diffusion with a drift potential: a viscosity
  solutions approach.
\newblock Discrete Contin. Dyn. Syst. \textbf{27}(2), 767--786 (2010).
\newblock \doi{10.3934/dcds.2010.27.767}.
\newblock \urlprefix\url{http://dx.doi.org/10.3934/dcds.2010.27.767}

\bibitem{Laborde}
{Laborde}, M.: On some non linear evolution systems which are perturbations of
  {W}asserstein gradient flows.
\newblock preprint at \url{http://arxiv.org/abs/1506.00126}

\bibitem{LiebLoss}
Lieb, E.H., Loss, M.: Analysis, \emph{Graduate Studies in Mathematics},
  vol.~14.
\newblock American Mathematical Society, Providence, RI (1997)

\bibitem{LiebYau}
Lieb, E.H., Yau, H.T.: The {C}handrasekhar theory of stellar collapse as the
  limit of quantum mechanics.
\newblock Comm. Math. Phys. \textbf{112}(1), 147--174 (1987).
\newblock \urlprefix\url{http://projecteuclid.org/euclid.cmp/1104159813}

\bibitem{LinZhang}
Lin, F., Zhang, P.: On the hydrodynamic limit of {G}inzburg-{L}andau vortices.
\newblock Discrete Contin. Dynam. Systems \textbf{6}(1), 121--142 (2000).
\newblock \doi{10.3934/dcds.2000.6.121}.
\newblock \urlprefix\url{http://dx.doi.org/10.3934/dcds.2000.6.121}

\bibitem{Lions}
Lions, P.L.: The concentration-compactness principle in the calculus of
  variations. {T}he locally compact case. {I}.
\newblock Ann. Inst. H. Poincar\'e Anal. Non Lin\'eaire \textbf{1}(2), 109--145
  (1984).
\newblock \urlprefix\url{http://www.numdam.org/item?id=AIHPC_1984__1_2_109_0}

\bibitem{Loeper}
Loeper, G.: Uniqueness of the solution to the {V}lasov-{P}oisson system with
  bounded density.
\newblock J. Math. Pures Appl. (9) \textbf{86}(1), 68--79 (2006).
\newblock \doi{10.1016/j.matpur.2006.01.005}.
\newblock \urlprefix\url{http://dx.doi.org/10.1016/j.matpur.2006.01.005}

\bibitem{MasmoudiZhang}
Masmoudi, N., Zhang, P.: Global solutions to vortex density equations arising
  from sup-conductivity.
\newblock Ann. Inst. H. Poincar\'e Anal. Non Lin\'eaire \textbf{22}(4),
  441--458 (2005).
\newblock \doi{10.1016/j.anihpc.2004.07.002}.
\newblock \urlprefix\url{http://dx.doi.org/10.1016/j.anihpc.2004.07.002}

\bibitem{MRS}
Maury, B., Roudneff-Chupin, A., Santambrogio, F.: A macroscopic crowd motion
  model of gradient flow type.
\newblock Math. Models Methods Appl. Sci. \textbf{20}(10), 1787--1821 (2010).
\newblock \doi{10.1142/S0218202510004799}.
\newblock \urlprefix\url{http://dx.doi.org/10.1142/S0218202510004799}

\bibitem{MRSV}
Maury, B., Roudneff-Chupin, A., Santambrogio, F., Venel, J.: Handling
  congestion in crowd motion modeling.
\newblock Netw. Heterog. Media \textbf{6}(3), 485--519 (2011).
\newblock \doi{10.3934/nhm.2011.6.485}.
\newblock \urlprefix\url{http://dx.doi.org/10.3934/nhm.2011.6.485}

\bibitem{McCann}
McCann, R.J.: A convexity principle for interacting gases.
\newblock Adv. Math. \textbf{128}(1), 153--179 (1997).
\newblock \doi{10.1006/aima.1997.1634}.
\newblock \urlprefix\url{http://dx.doi.org/10.1006/aima.1997.1634}

\bibitem{MelletPerthameQuiros}
Mellet, A., Perthame, B., Quiros, F.: A {H}ele-{S}haw problem for tumor growth.
\newblock preprint at \url{http://arxiv.org/abs/1512.069957}

\bibitem{PereaGomezElosegui}
Perea, L., G\'omez, G., Elosegui, P.: Extension of the {C}ucker--{S}male
  control law to space flight formations.
\newblock AIAA J. of Guidance, Control, and Dynamics \textbf{32}, 527--537
  (2009)

\bibitem{Poupaud}
Poupaud, F.: Diagonal defect measures, adhesion dynamics and {E}uler equation.
\newblock Methods Appl. Anal. \textbf{9}(4), 533--561 (2002).
\newblock \doi{10.4310/MAA.2002.v9.n4.a4}.
\newblock \urlprefix\url{http://dx.doi.org/10.4310/MAA.2002.v9.n4.a4}

\bibitem{RechtsmanStillingerTorquato}
Rechtsman, M., Stillinger, F., Torquato, S.: Optimized interactions for
  targeted self- assembly: application to a honeycomb lattice.
\newblock Phys. Rev. Lett. \textbf{95}(22) (2005)

\bibitem{Santambrogio}
Santambrogio, F.: Optimal transport for applied mathematicians.
\newblock Progress in Nonlinear Differential Equations and their Applications,
  87. Birkh\"auser/Springer, Cham (2015).
\newblock \doi{10.1007/978-3-319-20828-2}.
\newblock \urlprefix\url{http://dx.doi.org/10.1007/978-3-319-20828-2}.
\newblock Calculus of variations, PDEs, and modeling

\bibitem{Sugiyama}
Sugiyama, Y.: Global existence in sub-critical cases and finite time blow-up in
  super-critical cases to degenerate {K}eller-{S}egel systems.
\newblock Differential Integral Equations \textbf{19}(8), 841--876 (2006)

\bibitem{SunUminskyBertozzi}
Sun, H., Uminsky, D., Bertozzi, A.L.: Stability and clustering of self-similar
  solutions of aggregation equations.
\newblock J. Math. Phys. \textbf{53}(11), 115,610, 18 (2012).
\newblock \doi{10.1063/1.4745180}.
\newblock \urlprefix\url{http://dx.doi.org/10.1063/1.4745180}

\bibitem{Talenti}
Talenti, G.: Elliptic equations and rearrangements.
\newblock Ann. Scuola Norm. Sup. Pisa Cl. Sci. (4) \textbf{3}(4), 697--718
  (1976)

\bibitem{TopazBertozziLewis}
Topaz, C.M., Bertozzi, A.L., Lewis, M.A.: A nonlocal continuum model for
  biological aggregation.
\newblock Bull. Math. Biol. \textbf{68}(7), 1601--1623 (2006).
\newblock \doi{10.1007/s11538-006-9088-6}.
\newblock \urlprefix\url{http://dx.doi.org/10.1007/s11538-006-9088-6}

\bibitem{Vazquez}
V{\'a}zquez, J.L.: The porous medium equation.
\newblock Oxford Mathematical Monographs. The Clarendon Press, Oxford
  University Press, Oxford (2007).
\newblock Mathematical theory

\bibitem{Villani}
Villani, C.: Topics in optimal transportation, \emph{Graduate Studies in
  Mathematics}, vol.~58.
\newblock American Mathematical Society, Providence, RI (2003)

\bibitem{Wales}
Wales, D.: Energy landscapes of clusters bound by short-ranged potentials.
\newblock Chem. Eur. J. Chem. Phys. \textbf{11}, 2491--2494 (2010)

\end{thebibliography}

\end{document}